\documentclass[12,reqno]{amsart}

\usepackage{amsfonts}
\usepackage[active]{srcltx}
\usepackage{ulem}
\usepackage{amsmath}
\usepackage{amssymb}
\usepackage{amsthm}
\usepackage{mathtools}
\usepackage{bm} 
\usepackage{empheq}
\usepackage{mathptmx}

\usepackage{enumitem}
\usepackage{wasysym}
\usepackage{verbatim}
\usepackage{graphicx}

\usepackage[
bookmarks=true,         
bookmarksnumbered=true, 
colorlinks=true, pdfstartview=FitV, linkcolor=blue, citecolor=blue,
urlcolor=blue]{hyperref}
\usepackage{microtype}
\usepackage{upgreek}

\theoremstyle{plain}\newtheorem{theorem}{Theorem}[section]

\newtheorem{corollary}[theorem]{Corollary}

\newtheorem{lemma}[theorem]{Lemma}
\newtheorem{problem}[theorem]{Problem}
\newtheorem{proposition}[theorem]{Proposition}

\renewenvironment{proof}[1][Proof]{\textbf{#1.} }{\ \rule{0.5em}{0.5em} \par }
\theoremstyle{remark}

\theoremstyle{definition}
\newtheorem{remark}[theorem]{Remark}
\newtheorem{definition}[theorem]{Definition}
\newtheorem{example}[theorem]{Example}

\def\PP{\mathbb{P}}
\def\RR{\mathbb{R}}
\def\mR{\mathbb{R}}
\def\mI{\mathbb{I}}

\def\mE{{\mathbb E}}

\def\mH{{\mathbb H}}
\def\mI{{\mathbb I}}

\def\mL{{\mathbb L}}

\def\mN{{\mathbb N}}

\def\mP{{\mathbb P}}

\def\mR{{\mathbb R}}

\def\mW{{\mathbb W}}

\def\EE{\mathbb{E}}

\def\ZZ{\mathbb{Z}}

\def\cC{{\mathcal C}}

\def\cT{{[0, T]}}

\def\si{{\sigma}}

\def\et{{\eta}}

\def\Om{{\Omega}}

\def\al{{\alpha}}

\def\si{{\sigma}}

\def\tr{{ \hbox{ Tr} }}
\def\esssup {{ \hbox{ ess\ sup} }}
\def\supp {{ \hbox{ supp } }}

\def\si{{\sigma}}
\def\al{{\alpha}}

\renewcommand{\theequation}{\arabic{section}.\arabic{equation}}

\def\red{\color{red}}

\def\diam{\mathrm{diam}}

\newcommand{\ep}{\ensuremath{\varepsilon}}

\newcommand{\ce}{\begin{eqnarray*}}
\newcommand{\de}{\end{eqnarray*}}
\newcommand{\cen}{\begin{eqnarray}}
\newcommand{\den}{\end{eqnarray}}

\def\Om{\Omega}

\def\[{{\Big[}}
\def\]{{\Big]}}
\def\<{{\langle}}
\def\>{{\rangle}}
\def\({{\Big(}}
\def\){{\Big)}}

\def\bx{{\mathbf{x}}}
\def\tr{{\rm tr}}

\def\bt{\begin{theorem}}
\def\et{\end{theorem}}
\def\bl{\begin{lemma}}
\def\el{\end{lemma}}
\def\br{\begin{remark}}
\def\er{\end{remark}}
\def\bx{\begin{Example}}
\def\ex{\end{Example}}
\def\bd{\begin{definition}}
\def\ed{\end{definition}}
\def\bp{\begin{proposition}}
\def\ep{\end{proposition}}
\def\bc{\begin{corollary}}
\def\ec{\end{corollary}}
\def\ba{\begin{assumption}}
\def\ea{\end{assumption}}

\def\cC{{\mathcal C}}

\def\cF{{\mathcal F}}

\def\cJ{{\mathcal J}}

\def\cM{{\mathcal M}}

\def\cR{{\mathcal R}}

\def\cT{{\mathcal T}}

\renewcommand{\theequation}{\arabic{section}.\arabic{equation}}
\let\Section=\section
\def\section{\setcounter{equation}{0}\Section}

\begin{document}

\title[Well-posedness of SDE with discontinuous and unbounded drift]{Strong solution of stochastic differential equations with discontinuous and unbounded coefficients}

\author{ \sc  Yaozhong Hu  }
 \thanks{Y.Hu was supported by the NSERC discovery fund and a centennial  fund of University of Alberta,
National Natural Science Foundation of China (12261046).}
\address{Department of Mathematical and  Statistical
 Sciences, University of Alberta at Edmonton,
Edmonton, Canada, T6G 2G1}
  \email{ yaozhong@ualberta.ca
}
\author{\sc Qun Shi }
\address{School of Mathematics and Statistics, Jiangxi Normal University,
 Nanchang, Jiangxi 330022, P.R.China and Department of Mathematical and  Statistical
 Sciences, University of Alberta at Edmonton}
  \email{shiq3@mail2.sysu.edu.cn}
\thanks{Q.Shi was supported by China Overseas Education Fund Committee,
National Natural Science Foundation of China (11901257, 12261046)  and an NSERC discovery fund. \newline
 }

\date{}
\maketitle


\begin{abstract}
In this paper we study the existence and uniqueness of   strong solution  of the following $d$-dimensional stochastic differential equation (SDE)  driven by Brownian motion:
\ce
dX_t=b(t,X_t)dt+\sigma(t,X_t)dB_t, ~~~X_0=x,
\de
where $B$ is a $d$-dimensional standard Brownian motion;  the diffusion coefficient $\sigma$ is a  H\"older continuous and uniformly non-degenerate $d\times d$ matrix-valued function and  the drift coefficient   $b $ may be discontinuous and unbounded, not necessarily  in $\mathbb{L}_p^q$,      extending the previous works   to discontinuous and  unbounded drift coefficient situation.
The idea is to combine the Zvonkin's transformation with the Lyapunov function approach.  Zvonkin's transformation is a one-to-one (and quasi-isometric) transformation of a phase space that allows us to pass from a diffusion process with nonzero
drift coefficient to a process without drift.
To this end, we need to establish a local version of  the   connection between the  solutions of the SDE up to the  exit time of a bounded connected open set  $D$
and the associated
partial differential equation on this domain.   As an interesting by-product, we establish a localized  version of the Krylov estimates (Theorem \ref{th19}) and a localized   version of the stability result of the stochastic differential equations of discontinuous coefficients (Theorem \ref{th22}).

\end{abstract}


\setcounter{equation}{0}
\renewcommand{\theequation}{\arabic{section}.\arabic{equation}}
\numberwithin{equation}{section}

\maketitle

\section{\bf Introduction}
Let $(\Omega,{\mathcal F},{\mathbb P}, (\mathcal  F_t)_{t\geq0})$ be a complete filtered probability space with a  filtration $(\mathcal  F_t)_{t\geq0} $  satisfying  the usual conditions,   namely,    $(\mathcal  F_t)_{t\geq0} $ is a right continuous  increasing family of Borel fields,  which is also complete in the sense that   if $\mathcal{N}\subseteq A$ and $\PP(A)=0$, then $\mathcal{N}\in \cF_0$.       Let $ (B_t =(B_t^1, \cdots, B_t^d)^t)_{t\geq0}$  be a $d$-dimensional ${\mathcal  F_t}$-adapted standard   Brownian motion, where and later   ${}^t $ denotes the transpose of a vector or matrix.

In this work we study the existence and uniqueness of   strong solution to  the following stochastic differential equation (SDE):
\begin{equation}\label{sec1-eq1}
dX_t=b (t,X_t)dt +\sum_{i=1}^d \sigma_i(t,X_t)dB_t^i =b (t,X_t)dt +\sigma(t,X_t)dB_t, \quad X_0=x\in \RR^d,
\end{equation}
where   the diffusion coefficients  $\si_i:\mR_+\times\mR^d\rightarrow \mR^d $, $i=1, \cdots, d$ satisfy   $\sigma=(\si_1, \cdots, \si_d):  \mR_+\times\mR^d\rightarrow \mR^d\otimes\mR^d$ is   uniformly elliptic and      H\"older   continuous with respect to spatial variable in any bounded domain $D$ (connected open subset of $\mathbb{R}^d$).   
The drift coefficient $b : \mR_+\times\mR^d\rightarrow\mR^d$ 
is  measurable and may  be discontinuous.   Unlike in the existing literature 
we do not assume the integrablity of $b$   on the whole space 
$\RR^d$. Instead, we assume that   it is integrable in any space time bounded domain
$[0, T]\times D$ and there is a Lyapunov function $V$ associated with \eqref{sec1-eq1}.  
More precisely, we make the following assumptions about these  coefficients. 
\begin{enumerate}
\item[$({\mathbf H}  ^\sigma)$:]   For any bounded domain $D\subseteq \RR^d$ there exist constants   $\alpha\in(0,1]$, $\kappa=\kappa_{\al, D}>1$,  such that for all $(t,x)\in[0,T]\times D$,
$$
\kappa^{-1}|\xi|^2\leq|\sigma^t(t,x)\xi|^2\leq \kappa|\xi|^2, ~~~\forall\xi\in\mR^d\,,
$$
and
\ce
\|\sigma(t,x)-\sigma(t,y)\|_{HS}\leq \kappa |x-y|^\alpha,
\de
where    $\sigma^t$ stands for  the  transpose of the matrix $\sigma$ and where
for a matrix $A$, $\|A\|_{HS}=\tr(A^tA)$ stands for its Hilbert-Schmidt norm.   We also assume
\begin{equation}
\begin{split}
 \nabla\sigma_i=&\left(\frac{\partial \sigma_{ki}}{\partial x_j}\right)^t_{k,j=1,2,\ldots,d}
 =\begin{pmatrix} \frac{\partial \sigma_{1i}}{\partial x_1}&\cdots&\frac{\partial \sigma_{di}}{\partial x_1}\\
 	\vdots &\cdots&\vdots\\
 	\frac{\partial \sigma_{1i}}{\partial x_d}&\cdots&\frac{\partial \sigma_{di}}{\partial x_d}\\
 \end{pmatrix}\in \mL^q_p([0,T]\times D)\,,
\\\
&i=1, \cdots, d\,,\quad \hbox{for certain $p$ and $q$ satisfying }
\end{split}
\label{e.1.2}
\end{equation}
\begin{equation}
\frac{d}{p}+\frac{2}{q}<1, \label{e.a.1.3}
\end{equation}
where $\si_i=(\si_{1i}, \cdots, \si_{di})^t$ and  $\mL^q_p([0,T]\times D) $ is defined by \eqref{e.2.1} in the next section.

The  constraint \eqref{e.a.1.3}
is the same as in  \cite{Xia} so that we can apply the results there to ensure
 the local  existence and uniqueness of (\ref{sec1-eq1}).

For the  drift coefficient we make the following assumptions.
\item[$({\mathbf H}^{b } )$:]  $b$ is  integrable on
any bounded domain of ${\mR}^d$,
namely,
$$
b\in \mL_p^q([0, T]\times D) <\infty\,, \quad \forall \ \hbox{bounded
domain $D$}\,,
$$
for    $p$ and $q$ appearing in \eqref{e.a.1.3}.

\item[$({\bf H}^{L})$:]   There exists a non-negative  (Lyapunov)  function   $V\in C^{1,2}([0, T]\times\mR^d)$ [the set of all functions $\{f(t,x), t\in [0, T], x\in \mR^d\}$  which are continuously differentiable in $t$ and have continuous derivatives with respect to the spatial variable $x$  up to second order]  satisfying
\ce
\lim_{R\rightarrow\infty}\inf_{0\leq t\leq T, |x|=R}V(t,x)=\infty
\de
 and
\begin{equation}
L_tV(t,x)\leq CV(t,x), ~~~\forall t\in[0,T], ~~~x\in\mR^d\,,  \label{e.1.3}
\end{equation}
for some  constant $C>0 $,  where $L_t $ is the differential operator  associated with \eqref{sec1-eq1}:
\begin{equation}
L_t:=\frac{\partial}{\partial t}+L^{\sigma,b}=\frac{\partial}{\partial t}
+\sum^d_{i =1} b_i(t,x)\frac{\partial }{\partial{x_i} }+\frac{1}{2}\sum^d_{i,j=1}\sum^d_{k=1}(\sigma^{ik}\sigma^{jk})(t,x)\frac{\partial^2}{\partial{x_i}\partial{x_j}}\,.
\end{equation}
\end{enumerate}

The objective of this work is to show that under the above assumptions $ ({\mathbf H}^ \sigma) $, $({\mathbf H}^ {b }) $ and $({\mathbf H}^L)$, the equation
\eqref{sec1-eq1} has a unique strong solution (e.g.  Theorem \ref{th11}).

For stochastic differential equations with discontinuous (yet bounded)  coefficients   there are rather complete  general results about the weak solution and we   refer to the classical work  \cite{SV} and references therein. For  the strong solutions there have been also some important progresses.
Among them let us mention  the works
\cite{KR,  vere,  Zhang, ZY}.  Let us   emphasize an important point that in these works, the assumptions  that $\si$ is uniformly elliptic on the whole space $\RR^d$ and $b\in \mL_p^q=\mL_p^q(\RR^d)$  are   usually  needed,  which generally   require   that $\si$ and $b$ are  bounded
when $|x|\to \infty$.   In \cite{Xia},  the authors  show the weak differentiability of the unique strong solution with respect to the starting point $x$ as well as Bismut-Elworthy-Li's derivative formula for SDE (\ref{sec1-eq1}) under the condition that 1) $\sigma$ is bounded, uniformly continuous,  and nondegenerate, 2)  $b\in{\widetilde\mL}^{q_1}_{p_1}$, $\nabla\sigma\in{\widetilde\mL}^{q_2}_{p_2}$ for some $p_1,q_1\in[2,\infty)$ with $\frac{d}{p_i}+\frac{2}{q_i}<1$, $i=1,2$, where ${\widetilde\mL}^{q_i}_{p_i}$ are some localized spaces of ${\mL}^{q_i}_{p_i}$. But unbounded drift coefficients may still not be contained in this space  because   if $b$ is unbounded,   then $|||b|||_{{\widetilde\mL}^{q_i}_{p_i}}:=\sup_{z\in\mR}||b\cdot\chi^z_r||_{{\mL}^{q_i}_{p_i}}=\infty$ with $\chi\in C_c^\infty(\mR^d)$ and $\chi^z_r(x):=\chi_r(x-z):=\chi(\frac{x-z}{r})$, $r>0$.
The idea  in all of these works is to use the Zvonkin transformation to transform equation \eqref{sec1-eq1} with discontinuous coefficients to an equation without drift and hence the problem of discontinuity of the drift coefficient disappears.

Recently,  the more and more    interest in studying  SDE (\ref{sec1-eq1}) with discontinuous coefficients is partly due to its more and more important role  played  in applications. For example,  the threshold Ornsten-Uhlenbeck processes (e.g. \cite{HX}  and the references therein)  widely used in application contain discontinuous but piecewise linear drift.
 In \cite{Flandoli}, F. Flandoli, M. Gubinelli and E. Priola give an interesting    example of partial differential equation that may lack uniqueness, but it becomes  well-posed when adding  a suitable noise
meaning that the introduction of the noise can sometimes regularize the irregular system.
Let us also point out that in these equations
the drift coefficient $b$ is not  in $\mL_p^q(\RR^d)$.

Let us mention  a recent work \cite{Zhang1}  in which  Zhang and Zhao show the existence and uniqueness of the martingale solution (weak solution) to the above homogeneous (time independent) SDE (\ref{sec1-eq1}) under the following assumptions:
\begin{enumerate}
\item[(1)] the drift coefficient $b$ is decomposed into $b=b_1+b_2$,  where
\begin{equation}
\frac{\langle x,b_1(x)\rangle}{\sqrt{1+x^2}}\leq-\kappa_0|x|^\vartheta+\kappa_1,~~~|b_1(x)|\leq\kappa_2(1+|x|^\vartheta)\label{e.1.5}
\end{equation}
for some $\vartheta\geq0$ and $\kappa_0,\kappa_1,\kappa_2>0$,
\ce
b_2\in \mH^{-\alpha}_p ~~for~~~ some~~~ \alpha\in(0,1/2] ~~~and~~~ p\in\left(\frac{d}{1-\alpha},\infty\right)\,.
\de
\item[(2)] The diffusion coefficient $\sigma$ is uniformly elliptic  and
\ce
||(-\Delta)^{\beta/2}\sigma||_{L^r}<\infty~~~~for~~~ some~~~ \beta\in[\alpha,1],  ~~~r\in(d/\beta,\infty)\,.
\de
\end{enumerate}
They obtain sharp two-sided as well  as the  gradient estimates of the heat kernel associated to the above SDE. Moreover, they study the ergodicity and global regularity of the invariant measures of the associated semigroup.  However, there is no study on   strong solution.

If $b$ and $\sigma$ are locally Lipschitzian,  then  it is well-known that
the assumption $({\bf H} ^L)$  is the famous Lyapunov type condition to guarantee  non explosion of the equation (e.g.  \cite[p.346, Theorem 10.2]{Hu}). So, in some sense our condition is to combine the Lyapunov condition and the integrability condition. Our new set of conditions can   be applied to some new situations. We give only some very special   examples as follows.

\begin{example}(\cite{HX}) One example is the threshold Ornstein-Ulenbeck processes:
\begin{equation}
dX_t=\sum_{i=1}^n (\beta_i-\alpha_i X_t)I_{
\{\theta_{i-1}
\le X_t<\theta_i\}}  dt +\sigma dB_t\,,
\end{equation}
where $\beta_1, \cdots , \beta_n, \alpha_1, \cdots, \alpha_n, -\infty=\theta_0<
\theta_1<\cdots <\theta_{n-1}<\theta_n=\infty$ are constants.   It is clear that for all parameters $\beta_i, \alpha_i, \theta_i$,  the drift coefficient
$b(x)=\sum_{i=1}^n (\beta_i-\alpha_i x)I_{
\{\theta_{i-1}
\le x<\theta_i\}} $ and the diffusion coefficient $\si(x)=\si$ satisfy  \eqref{e.1.3} with $V(x)=|x|^2$ and with some constant $C$.  But it does  not satisfy \eqref{e.1.5} unless $\alpha_1, \cdots, \alpha_n$ are  positive.
\end{example}

\begin{example} Another interesting example is
\begin{equation}
dX_t=\sum_{k=1}^n \left(\sum_{j=1}^{m_k}  \beta_{k,j} X_t^j\right)  I_{
\{\theta_{k-1}
\le X_t<\theta_k\}}  dt +\sigma dB_t\,,
\end{equation}
where $\beta_{k,j}$ are constants and $\beta_{k, m_k}\not=0$  and  $ -\infty=\theta_0<
\theta_1<\cdots <\theta_{n-1}<\theta_n=\infty$ are constants.   It is clear that  if $m_1,  m_n$  are odd, and if  $\beta_{1,n_1},\beta_{n,m_n}$ are negative, then \eqref{e.1.3} is satisfied
with  $V(x)=|x|^2$ and with some constant $C$.  But \eqref{e.1.5} will generally not be satisfied.
\end{example}

In the above assumption (${\bf H}^b$), $b$ may not be locally Lipschitzian  and may be unbounded in
 $\RR^d$ so that  we cannot use the existing theory
 to solve \eqref{sec1-eq1}.  To study the well-posedness of this equation, let us
 consider a twice differentiable function $g:[0, T]\times \RR^d\to \RR$. Then the  It\^o formula yields that
$$
 dg(t, X_t)=L_t g(t, X_t) dt +
 \sum_{i=1}^d (\nabla_x  g(t, X_t) )^t \sigma_i (t, X_t) dB_t^i \,.
$$
If we  can construct a one-to-one (and quasi-isometric) transformation $\Psi=(\Phi_1, \cdots, \Phi_d)^t:[0, T]\times \RR^d  \to \RR^d$ such that $L_t \Phi_i(t, x)=0$. Then the transformation   $  Y_t=\Phi(t, X_t)$       (called Zvonkin transformation (e.g.  \cite{Zvonkin})) will satisfy an SDE  without drift.  As we know by the Yamada-Wanatabe theory and its extension, the H\"older continuity of the diffusion coefficient  can ensure the existence and uniqueness of the strong solution  (e.g. \cite{Zhang}).
Thus, to  obtain  the existence and uniqueness of strong solution,  one need to study  the equation
 $L_t \Phi_i(t, x)=0$ first which was solved in    \cite{Xia}  under the conditions that $\si$ satisfies  (${\bf H}^\si$) and  $
 b\in \mL_p^q([0, T]\times \RR^d) <\infty
 $
 for    $p$ and $q$  satisfying  \eqref{e.a.1.3}.

Since we only assume that $b$ is bounded on bounded domain,    our strategy  to solve \eqref{sec1-eq1} is to  consider first the solvability of the associated parabolic differential equation  with Cauchy-Dirichlet problem on bounded domain
\begin{equation}\label{sec1-eq2}
\left\{
                  \begin{array}{lll}
                  L_t u+f=\partial_tu+L^{\sigma}u+b\cdot\nabla u+f=0,  ~~~(t,x)\in[0,T]\times D,\\
                  u(T,x)=0,~~~x\in D,\\
                  u(t,x)=g(t,x),   ~~~(t,x)\in[0,T]\times\partial D\,,
                  \end{array}
                    \right.
\end{equation}
where $D$ is  bounded nonempty domain (a connected open subset)  of $\mR^d$  so that  $\partial D\in C^2$, $g$ is integrable  function on $[0,T]\times\partial D$,  and $L^{\sigma}$ is introduced later in \eqref{sec3-eq1}.   Then, with the aid of the solution
to the above equation we   find a $C^1$-diffeomorphism $\Phi$ (given in Section 4) to transform  \eqref{sec1-eq1} to another
equation    with only diffusion term. This transformed equation  has a unique strong solution in bounded nonempty domain. Finally,  by a stopping time argument combined with the Lyapunov function $V$ (given in assumption
(${\bf H}^L$))   we obtain the  unique strong solution in $\mR^d$
of  SDE (\ref{sec1-eq1}).

There are new challenges in each of our above steps. First, when the coefficients are nice equation
\eqref{sec1-eq2} is a classical Cauchy-Dirichlet problem in bounded domain. However, when the coefficient is only locally integrable, it seems there is no study on such equation, to our best knowledge. Most of the works are on
the parabolic problem for whole space $\RR^d$.
There are also  studies  on  elliptic problem for general  (including bounded) domain in the works \cite{KR,  vere,  Zhang, ZY}.  In \cite{Krylov1},  it is mentioned that
the results on elliptic problem on general domain can be extended to
parabolic equations.

In the previous works on the strong solution with integrable drift coefficients, an important
technique is the so-called the Zvonkin transformation which reduced the original SDE to an SDE without drift terms. Now we need study   SDE on a bounded domain. Since the solution will
stay  bounded  only in some finite (stopping)  time, we need to  study the solution with a
finite (random) life time for the equation. Not much such study  is available. To obtain the strong solution of  an SDE on the bounded   domain, we extend the coefficients  to the whole
Euclidean space $\RR^d$. However, there is no  available extension theorem  we can immediately  use.
We need to modify some existing results so that they are applicable to our situation. Some of these
are given in Section \ref{s.2} and some are given in Appendix (Section \ref{s.a}).
The stochastic differential equations   on a bounded  domain
is studied in detail in Section \ref{s.4}. For the existence and uniqueness  one of the most important
tools is the Krylov estimate.  We deduce the Krylov type estimate for stochastic differential equations
on bounded domain, namely, a localized version of the Krylov estimate,
 also in Section \ref{s.4}. In Section 5, we prove the main result on  the existence and uniqueness of strong solutions to SDE (\ref{sec1-eq1}) under Lyapunov condition.
In Section 2, we recall some well-known results and give
briefly some preliminaries about the Sobolev differentiabilities of random vector fields.
The last section contains  some technical results obtained and used in the paper.

%
%

Throughout this paper, we use the following convention: $C$ with or without subscripts will denote a positive constant, whose value may be different  in different places, and whose dependence on the parameters can be traced from the calculations.

\section{Prelimiaries}\label{s.2}
We first introduce some spaces and notations for later use.  Let $p,q\in[1,\infty)$,   $T>0$  and let $D$ be a bounded, connected domain in $\mR^d$ with $C^2$ boundary.  We denote by $\mL^q_p([0,T]\times D)$ the space of all real-valued Borel functions on $[0,T]\times D$ such that
\begin{equation}
\|f\|_{\mL^q_p([0,T]\times D)}:=\left(\int_0^T\left(\int_D|f(t,x)|^pdx\right)^{q/p}dt\right)^{1/q} \,. \label{e.2.1}
\end{equation}
For $p,q=\infty$,
\ce
\|f\|_{\mL^\infty_\infty([0,T]\times D)}:=\esssup_{t\in[0,T]}\esssup_{x\in D}|f(t,x)|,
\de
where $\esssup$ denotes the essential maximum norm (e.g. \cite[p.172, Section 12]{analysis2}).
When $q=p$ we denote $\mL^p([0,T]\times D):=\mL^p_p([0,T]\times D)$.   When $D=\RR^d$, we denote
$\mL_p^q=\mL_p^q[0,T]:=\mL^q_p([0,T]\times \RR^d)$.

For any positive integer $m$ and any $p\ge 1$,  $W^m_p(D)$ is used to denote  the usual Sobolev space over  the domain $D\subseteq\mR^d$ with norm
\ce
\|f\|_{W^m_p(D)}:=\sum_{k=0}^m\|\nabla_x^kf\|_{L_p(D)} \,,
\de
where $\nabla_x^k$ denotes the $k^{th}$-order gradient operator on spatial variable and $L_p(D)$ stands
for the Lebesgue space of functions $f$ such that
$$
\|f\|^p_{L_p(D)}:=\int_D|f(x)|^pdx<+\infty\,,   ~~~1\leq p<\infty,
$$
and
$$
\|f\|_{L_{\infty}(D)}:=\esssup_{x\in D}|f(x)|.
$$

For $\beta\in\mR$, let $\mH^\beta_p:=(\mI-\Delta)^{-\frac{\beta}{2}}(L_p)$ be the usual Bessel potential space with norm (we refer to \cite{Adams, Stein, Triebel1})
\ce
\|f\|_{\mH^\beta_p}=\|(\mI-\Delta)^{\frac{\beta}{2}}f\|_p\,,
\de
where $\|\cdot\|_p$ is the usual $L_p$-norm.

By \cite[Theorem 2.23]{Taira}, we can construct an extension operator,
\ce
T: \mH^\beta_p(D)\rightarrow \mH^\beta_p(\RR^d)\,.
\de
We can use this property to define  $\mH^\beta_p(D)$ as    restrictions to $D$ of functions in $\mH^\beta_p(\RR^d)$
with the norm
$$
\|f\|_{\mH^\beta_p(D)}=\inf\|F\|_{\mH^\beta_p(\RR^d)},
$$
where the infimum is taken over all $F\in \mH^\beta_p(\RR^d)$  which equal $f$  on $D$.

Notice that for $m\in\mN$ and $1<p<\infty$ (refer to \cite[Theorem V.3]{Stein} for the whole space $\RR^d$, \cite[Theorem 7.63]{Adams} for the whole space $\RR^d$ and also for suitable regular domain),
\ce
\|f\|_{\mH^m_p}\asymp\|f\|_{W^m_p}\,,
\de
where  and throughout the paper we use $A\lesssim B$ to denote that fact that there is a constant $C$ such that $A\le CB$  and  $A\asymp B$ means $A\lesssim B$ and $B\lesssim A$.
For $\beta\in[0,2)$ and $p\in(1,\infty)$, by Michlin's multiplier theorem (see {\red  } \cite[page 88]{Triebel1}), we know
\ce
\|f\|_{\mH^\beta_p}=\|(\mI-\Delta)^{\frac{\beta}{2}}f\|_p\asymp\|f\|_p+\|(-\Delta)^{\frac{\beta}{2}}f\|_p\,.
\de
Let $C(\mR^d)$ be the collection of all continuous functions in $\mR^d$, equipped with the norm
\ce
||f||_{C(\mR^d)}:=\sup_{x\in\mR^d}|f(x)|.
\de
Obviously, $C(\mR^d)$ is a Banach space. Let $k\in \mN$. Then
\ce
C^k(\mR^d):=\{f\in C(\mR^d):D^\alpha f\in C(\mR^d) ~~~if~~~ |\alpha|\leq k\}
\de
is a  Banach space  equipped with the norm
\ce
||f||_{C^k(\mR^d)}=\sum_{|\alpha|\leq k}||D^\alpha f||_{C(\mR^d)}.
\de
Here,  we use standard notations: $\alpha=(\alpha_1, \cdots , \alpha_d)$ with $\alpha_j\in \mN_0:=\{0, 1,2, \cdots\}$ is a multiindex, $|\alpha|=\sum^d_{j=1}\alpha_j$ and
\ce
D^\alpha f(x)=\frac{\partial^{|\alpha|}f}{\partial_{x_1}^{\alpha_1}\cdots\partial_{x_d}^{\alpha_d}}(x).
\de
For any positive number $0<\delta<1$,  let $\cC^\delta$ be the usual H\"older space with finite norm
\ce
\|f\|_{\cC^\delta(\mR^d)}:=\sup_{x\in\mR^d}|f(x)|+\sup_{x\neq y,x,y\in\mR^d}\frac{|f(x)-f(y)|}{|x-y|^\delta}<\infty.
\de

For any  $0<\delta\not \in \mN_0$, we put
\ce
\delta=[\delta]+\{\delta\},
\de
where $[\delta]=\max\{k\in  \ZZ; k\le \delta\}$ is the integer part of $\delta$, $0<\{\delta\}<1$. Denote the H\"older space (refer to \cite[Chapter 1.2]{Triebel})
\begin{eqnarray}
 \cC^\delta(\mR^d) :&=&\left\{f\in C(\mR^d): \|f\|_{\cC^\delta(\mR^d)}=||f||_{C^{[\delta]}(\mR^d)}\right. \nonumber\\
&&\left.\qquad\quad +\sum_{|\alpha|=[\delta]} 
\sup_{x\neq y,x,y\in\mR^d}\frac{|D^\alpha f(x)-D^\alpha f(y)|}{|x-y|^{\delta-[\delta]}} \right\}. \label{e.def_holder}
\end{eqnarray}
By Sobolev's embedding theorem (\cite[Corollary 7.11]{David}), we have
\begin{equation}\label{sec2-eq1}
\|f\|_{\cC^\delta(D)}\leq C\|f\|_{\mH^\beta_p(D)}, ~~~~~\beta-\delta>\frac{d}{p},~~~\delta\geq0.
\end{equation}
In particular, we take $\delta=0$ to obtain  $\|f\|_\infty\leq C\|f\|_{\mH^\beta_p}$, as $\beta>\frac{d}{p}$.

Furthermore, we also need the following Sobolev space: for $p,q\in[1,\infty), m\ge 0$   we denote by  $\mW^{m,q}_p([0,T]\times D)$
the set of all Borel functions $\left\{f(t,x)\,,
0\le t\le T, x\in \mR^d\right\}$  such that
\ce
\|f\|_{\mW^{m,q}_p([0, T]\times D)}:=\left(\int_0^T   \left[
\| \partial_t f(t) \|_{L_p(D)}^q +\sum_{k=0}^m  \|\nabla^kf(t) \|_{L_p(D)}^q \right]dt \right)^{1/q}< \infty\,,
\de
where $\nabla^k$  is the gradient with respect to   $x$ only.
By ${\mathring\mW}^{1,q}_p([0,T]\times D)$ we mean the subset of $\mW^{1,q}_p([0,T]\times D)$ consisting of all functions 
vanishing on the boundary $\partial D$. For $k\in\{2,3,\ldots\}$
let
\ce
{\mathring\mW}^{k,q}_p([0,T]\times D)={\mathring\mW}^{1,q}_p([0,T]\times D)\cap\mW^{k,q}_p([0,T]\times D).
\de
The norm in ${\mathring\mW}^{k,q}_p([0,T]\times D)$, $k\in\{1,2,\ldots\}$, is taken to be the same as in $\mW^{k,q}_p([0,T]\times D)$.
The introduction of ${\mathring\mW}^{k,q}_p([0,T]\times D)$ is to express the Dirichlet boundary condition.    When we say  $u=g$ on $\partial D$
we understand the  following  condition
\ce
u-g\in{\mathring\mW}^{k,q}_p([0,T]\times D)\,.
\de
Since $D$ is in $C^1$, the surface measure $\si_{\partial D}$ is well-defined and $L_p([0,T]\times \partial D)$ is defined with respect to the $dt\times \si_{\partial D}$ (e.g. \cite[p. 260]{Krylov1}  and  \cite[Section 6.3]{kufner}),  where $dt$ is Lebesuge measure on $[0, T]$.  We can also introduce the  function  class $\mW^{k,q}_p([0,T]\times \partial D)$, $k=0, 1,2,... $ (e.g. \cite[Section 6.7]{kufner}).  By   \cite[Theorem 13.7.2]{Krylov1}, namely, by the equivalence between Bessel space  and Sobolev space on boundary after regarding  $(t,x)$ as a new $(d+1)$-dimensional variable, a function
$g$   in $\mW^{k,q}_p([0,T]\times \partial D)$   can be extended to a function on  $[0,T]\times D$ so that it is in $  \mW^{k,q}_p([0,T]\times D)$.   We shall use this property in the future instead of giving the definition of
$\mW^{k,q}_p([0,T]\times \partial D)$.

Let $f$ be a locally integrable function on $\mR^d$. The classical Hardy-Littlewood maximal function is defined by
\begin{equation}
\cM f(x):=\sup_{0<r<\infty}\frac{1}{|A_r|}\int_{A_r}|f(x+y)|dy,
\label{e.def_maximal}
\end{equation}
where $A_r:=\{x\in\mR^d:|x|<r\}$.  An important property of this operator (see    \cite[Appendix A]{Crippa} and  \cite[Theorem 1.1]{Stein} ) is
\begin{equation}
\|\cM f(x)\|_{L_p(\mR^d)} \le C_{d,p}
\|f\|_{L_p(\mR^d)}\,,
\end{equation}
for  any $p>1$ and $f\in L_p(\mR^d)$. We also need to use a similar   property of this operator
on a domain $D$, which is  given in Appendix.

\section{\bf Solvability of parabolic differential equations with Dirichlet boundary }
\label{s.3}

\noindent Denote
\begin{equation}\label{sec3-eq1}
L^\sigma:=\frac{1}{2}a^{ij}(t,x)\frac{\partial^2}{\partial{x^i}\partial{x^j}}\,,
\quad\hbox{where}\quad
a^{ij}(t,x)=\sum^d_{k=1}(\sigma^{ik}\sigma^{jk})(t,x)\,.
\end{equation}
This section is devoted to the  study of  the following 
  Cauchy-Dirichlet problem in bounded domain $D$ for non-smooth coefficient parabolic equation:
\begin{equation}\label{sec3-eq000}
\left\{
                  \begin{array}{lll}
                  \partial_tu+L^{\sigma}u+b\cdot\nabla u+f=0,  ~~~(t,x)\in[0,T]\times D,\\
                  u(T,x)=0,~~~x\in D,\\
                  u(t,x)=g(t,x),   ~~~(t,x)\in[0,T]\times\partial D\,,
                  \end{array}
                    \right.
\end{equation}
where $D$ is a  bounded nonempty domain (connected open subset) of $\mR^d$ so that  $\partial D\in C^2$  and  $g\in L_1([0,T]\times\partial D)$ (as earlier, with respect to $dt\times \si_{ \partial D}$,
where $dt$ is the Lebesuge measure and $\sigma_{\partial D}$ is the surface measure).
	  We assume that the coefficients $\si$ and $b$ satisfy
({\bf H}$^\si$), ({\bf H}$^b$)  on  $D$. We mentioned a few time that $\partial D\in C^2$. We recall this concept in the following definition.

\begin{definition}\label{df0}Let $k\in\{1,2,\cdots\}$ and  let  $D$ be a  domain in $\mR^d$.  We   say that the domain $D$ is of class $C^k$   (denoted   by $D\in C^k$ (or $\partial D\in C^k$)),    if there are numbers $\kappa, \rho_0$ such that for any point $z\in\partial D$ there exists a diffeomorphism $\psi=\psi^z$ (we omit the dependence of $\psi$ on $z$)
 from   a ball $B_{\rho_0}(z)$  of center $z$ with radius $\rho_0$     onto a domain $D^z\subseteq\mR^d$ such that
\begin{eqnarray}\label{}
&(i)&D^z_+:=\psi(B_{\rho_0}(z)\cap D)\subseteq\mR_+^d
:=\left\{ x=(x^1, x^2, \cdots, x^d)\in \RR^d, x^1>0\right\}\nonumber\\
&&\qquad\qquad ~~~ and~~~ \psi(z)=0\in \RR^d,\nonumber\\
&(ii)&\psi(B_{\rho_0}(z)\cap \partial D)=D^z\cap
\partial \mR_+^d,\nonumber\\
&(iii)&\psi\in C^k(\overline{B}_{\rho_0}(z)), \psi^{-1}\in C^k(\overline{D}^z), \nonumber\\
&&\qquad\qquad ~~~and ~~~
|\psi|_{C^k(B_{\rho_0}(z))}+|\psi^{-1}|_{C^k(D^z)}\leq \kappa .
\end{eqnarray}
\end{definition}
Such  diffeomorphism $\psi$ is said to straighten  or flatten  the boundary
of $D$ near $z$.
Throughout this paper    we fix a domain $D\in \partial C^2$.

\subsection{Operators in a neighborhood of  a boundary  point}
For the moment we fix $z\in \partial D$.  To ease our notation,
for functions $v$ and $\widehat{v}$ defined in $[0,T]\times(B_{\rho_0}(z)\cap D)$ and $[0,T]\times D^z_+$, respectively, we write
\begin{equation}\label{sec-equ0}
v(t,x)=\widehat{v}(t,y)
\end{equation}
if this equality holds for
\begin{equation}\label{sec-equ00}
y=\psi(x)=\psi^z(x)=(\psi^1(x), \cdots, \psi^d(x)) \,.
\end{equation}

This equality is also  used to introduce a function $\widehat{v}
$ on $[0,T]\times D^z_+$ if we are given a function $v$ on $[0,T]\times (B_{\rho_0}(z)\cap D)$, and vice versa.    In this subsection if in a formula   we have both $x$ and $y$, we will always  assume that they are related by (\ref{sec-equ00}).

Let $v\in C^{1,2}([0,T]\times(B_{\rho_0}(z)\cap D))$. In the domain $D^z_+$ define the function $\widehat{v}(t,y)$ by relationship (\ref{sec-equ0}). Obviously, in $[0,T]\times(B_{\rho_0}(z)\cap D)$,
\begin{eqnarray}\label{e.3.7}
\partial_t v(t,x)&=&\partial_t\widehat{v}(t,y), \nonumber\\
v_{x^i}(t,x)&=&\widehat{v}_{y^r}(t,y)\psi^r_{x^i}(x),\nonumber\\
v_{x^ix^j}(t,x)&=&\widehat{v}_{y^r}(t,y)\psi^r_{x^ix^j}(x)+\widehat{v}_{y^ry^l}(t,y)\psi^r_{x^i}(x)\psi^l_{x^j}(x)\,,
\end{eqnarray}
where the Einstein summation convention is used.
In fact, the above formulas are also true if $v\in \mW^{2,q}_p([0,T]\times(B_{\rho_0}(z)\cap D))$ by a limiting argument.

\begin{lemma}\label{le1}  For any   function $\zeta\in C_0^\infty$ such that $\zeta(t,x)=0$ for $|x|\geq\rho_0/2$, $t\in[0,T]$ and $0\leq\zeta\leq1$ everywhere,  where $\rho_0$ is the one appearing in Definition \ref{df0}
set
\begin{equation}\label{sec-equ000}
\zeta^z(t,x)=\zeta(t,x-z),\quad \ ~~~\widehat{\zeta}^z(t,y)=\zeta^z(t,x).
\end{equation}
Let $z\in\partial D$  and let $v$, $\widehat{v}$ be functions on $[0,T]\times (B_{\rho_0}(z)\cap D)$ and $[0,T]\times D^z_+$, respectively, related by (\ref{sec-equ00}). Then
\begin{equation}
\zeta^zv\in\mathring{\mW}^{k,q}_p([0,T]\times (B_{\rho_0}(z)\cap D)) \Longleftrightarrow \widehat{\zeta}^z\widehat{v}\in\mathring{\mW}^{k,q}_p([0,T]\times D^z_+).\label{e.3.8}
\end{equation}
In addition, for $r=0,1,\ldots,k$,
\begin{equation}
\frac1C \|\zeta^zv\|_{\mW^{r,q}_p([0,T]\times (B_{\rho_0}(z)\cap D))}\leq  \|\widehat{\zeta}^z\widehat{v}\|_{\mW^{r,q}_p([0,T]\times D^z_+)}\leq C\|\zeta^zv\|_{\mW^{r,q}_p([0,T]\times (B_{\rho_0}(z)\cap D))}\,, \label{e.3.9}
\end{equation}
where $C=C(d,p,q,k,\kappa ,\zeta)$. Furthermore, the assertion  \eqref{e.3.8} is also true if we replace $\mathring{\mW}^{k,q}_p([0,T]\times (B_{\rho_0}(z)\cap D))$   by  $\mW^{k,q}_p([0,T]\times (B_{\rho_0}(z)\cap D))$.
\end{lemma}

\begin{proof}   By relationships  (\ref{sec-equ0}), (\ref{sec-equ00}) and (\ref{sec-equ000}), we have
\ce
\int_0^T\left(\int_{B_{\rho_0}(z)\cap D}|\zeta^zv|^p(t,x)dx\right)^{q/p}dt=\int_0^T\left(\int_{D^z_+}|\widehat{\zeta}^z\widehat{v}|^p(t,y)|J|dy\right)^{q/p}dt\,,
\de
where $J$ is the Jacobian   of the map $\psi^{-1}$, i.e.
$
J:=\frac{\partial\psi^{-1}}{\partial y}
$.  This shows \eqref{e.3.9} when $r=0$.

Without loss of generality, we assume $v\in C^{1,1}([0,T]\times(B_{\rho_0}(z)\cap D))$. Since
\ce
(\zeta^zv)_{x^i}(t,x)&=&v(t,x)\zeta^z_{x^i}(t,x)+\zeta^z(t,x)v_{x^i}(t,x)\\
&=&v(t,x)\widehat{\zeta}^z_{y^r}(t,y)\psi^r_{x^i}(x)+\zeta^z(t,x)\widehat{v}_{y^r}(t,y)\psi^r_{x^i}(x)\\
&=&(\widehat{\zeta}^z\widehat{v})_{y^r}(t,y)\psi^r_{x^i}(x),
\de
we have
\ce
\int_0^T\left(\int_{B_{\rho_0}(z)\cap D}|\zeta^zv|_{x^i}^p(t,x)dx\right)^{q/p}dt\leq K^q_1\int_0^T\left(\int_{D^z_+}|\widehat{\zeta}^z\widehat{v}|_{y^r}^p(t,y)|J|dy\right)^{q/p}dt.
\de
This shows  the first  inequality in \eqref{e.3.9} when $r=1$.

For general  $k$, we have
\ce
\int_0^T\left(\int_{B_{\rho_0}(z)\cap D}|\partial^k(\zeta^zv)|^p(t,x)dx\right)^{q/p}dt\leq C_{\kappa ,q}\int_0^T\left(\int_{D^z_+}\sum_{i=1}^k|\partial^i(\widehat{\zeta}^z\widehat{v})|^p(t,y)|J|dy\right)^{q/p}dt.
\de
Owing to $\partial_t(\zeta^zv)(t,x)=\partial_t(\widehat{\zeta}^z\widehat{v})(t,x)$,  we see
\ce
\int_0^T\left(\int_{B_{\rho_0}(z)\cap D}|\partial_t(\zeta^zv)|^p(t,x)dx\right)^{q/p}dt=\int_0^T\left(\int_{D^z_+}|\partial_t(\widehat{\zeta}^z\widehat{v})|^p(t,y)|J|dy\right)^{q/p}dt.
\de
Thus, we have
\ce
\|\zeta^zv\|_{\mW^{k,q}_p([0,T]\times (B_{\rho_0}(z)\cap D))}\leq C\|\widehat{\zeta}^z\widehat{v}\|_{\mW^{k,q}_p([0,T]\times D^z_+)}.
\de
This proves the first inequality in \eqref{e.3.9}.
As for the second inequality in \eqref{e.3.9}, we just need to make the change of variables  by  using  $\psi$ instead of $\psi^{-1}$.

If $(\zeta^zv)(t,x)=0$ for $x\in\partial(B_{\rho_0}(z)\cap D)$,   then $(\widehat{\zeta}^z\widehat{v})(t,y)=(\zeta^zv)(t,x)=0$ for $y\in\partial D^z_+$.
This shows that the assertion  \eqref{e.3.8} is also true if we replace $\mathring{\mW}^{k,q}_p([0,T]\times (B_{\rho_0}(z)\cap D))$   by  $\mW^{k,q}_p([0,T]\times (B_{\rho_0}(z)\cap D))$.
\end{proof}

Let $v\in C^{1,2}([0,T]\times(B_{\rho_0}(z)\cap D))$. 
Denote
\begin{eqnarray}\label{eq1}
L^0_tv(t,x) := \partial_tv(t,x)+\frac{1}{2} \sum _{i,j=1}^d a^{ij}(t,x)v_{x^ix^j}(t,x)
\end{eqnarray}
and
\begin{eqnarray}\label{2}
\widetilde{L}^z_t\widehat{v}(t,y)
:=  \partial_t\widehat{v}(t,y)+\frac{1}{2}
\sum _{r,l=1}^d \widetilde{a}^{rl}(t,y)\widehat{v}_{y^ry^l}(t,y)+\sum _{r =1}^d\widetilde{b}^r(t,y)\widehat{v}_{y^r}(t,y)\,,
\end{eqnarray}
where
\begin{equation}
\begin{cases}
\widetilde{a}^{rl}(t,y):=\sum _{i,j=1}^d a^{ij}(t,x) \psi^r_{x^i}(x)\psi^l_{x^j}(x)\,;      \\
\widetilde{b}^r(t,y):=\sum _{i,j=1}^d a^{ij}(t,x)\psi^r_{x^ix^j}(x)
\end{cases}
\end{equation}
and we recall that $x$ is related to $y$ by
\eqref{sec-equ00}.

We are going to discuss the equation
\begin{equation}
L^0_t v(t,x)=f(t,x)\label{e.3.13}
\end{equation}
for some function $f$ or
\begin{equation}
\tilde L_t \widehat{v}(t,y)=\widehat{f}(t,y)\,,
\label{e.3.14}
\end{equation}
where $\widehat{f}(t,y)=f(t,x)$.

\begin{lemma}\label{le2} 
Let $\sigma$ satisfy  $({\bf  H}^\sigma )$ on the domain $D$.
\begin{enumerate}
\item[(i)] For any $\psi$ and $D^z$  defined in  Definition \ref{df0} and  for any $y,  y_1,y_2\in D^z$, we have
\ce
|\widetilde{a}^{rl}(t,y)|\leq C,~~~|\widetilde{a}^{rl}(t,y_1)-\widetilde{a}^{rl}(t,y_2)|\leq C |y_1-y_2|^\alpha\,,
\de
where $C=C(d,\kappa)$.\\
\item[(ii)]  A function $v\in\mW^{2,q}_p([0,T]\times(B_{\rho_0}(z)\cap D))$  satisfies \eqref{e.3.13} in $[0,T]\times(B_{\rho_0}(z)\cap D)$ if and only if   $\widehat{v}\in\mW^{2,q}_p([0,T]\times D^z_+)$    satisfies \eqref{e.3.14}   in $[0,T]\times D^z_+$.\\
\item[(iii)]  The operator $\widetilde{L}_t$ is parabolic in $[0,T]\times D^z$, that is,  there is a $\widetilde{\kappa}>0$  such that
\ce
\widetilde{a}^{rl}(t,y)\theta^r\theta^l\geq\widetilde{\kappa}|\theta|^2, ~~~\forall\theta\in\mR^d, ~~~t\in[0,T],~~~y\in D^z.
\de
\end{enumerate}
\end{lemma}

\begin{proof}(i) The first assertion is obvious.

\noindent (ii) Using the computations in \eqref{e.3.7}, we see that if $v$ satisfies \eqref{e.3.13}, then $\tilde v$ satisfies
 \eqref{e.3.14}. The reverse part is similar.

\noindent (iii)  Notice that
\begin{equation*}
\begin{split}
\widetilde{a}^{rl}(t,y)\theta^r\theta^l
=&a^{rl}(t,x)\psi^r_{x^i}(x)\psi^l_{x^j}(x)\theta^r\theta^l\\
 =&a^{rl}(t,x)(\theta\psi)^r_{x^i}(x)(\theta\psi)^l_{x^j}(x)\geq\kappa|(\theta\cdot\psi)_{x}(x)|^2\,.
\end{split}
\end{equation*}
Letting  $\phi:=\psi^{-1}$ we have $\phi(y)=x$ and
\ce
\psi^r_{x^i}(x)\phi^i_{y^j}(y)=\delta_{rj}=
\left\{
                    \begin{array}{lll}
                     0, ~~~if ~~~r\neq j,\\
                     1, ~~~if ~~~r=j\,,
                    \end{array}
                    \right.
\de

\ce
(\theta\cdot\psi)^r_{x^i}(x)\phi^i_{y^j}(y)=\theta^r\delta_{rj}=\theta^j\,.
\de
Thus, we have
\ce
|\theta|^2\leq C\kappa^2|(\theta\psi)_{x}|^2.
\de
This yields
\ce
\widetilde{a}^{rl}(t,y)\theta^r\theta^l\geq\widetilde{\kappa}|\theta|^2
\de
with $\widetilde{\kappa}$ depending  on $d,\kappa  $.
\end{proof}

Next,  we fix a function $\eta\in C^\infty_0$ so  that
\begin{equation*}
\eta (t,x)=
\begin{cases}
1,& \hbox{ for $t\in[0,T], |x|\leq\rho_0/2$},\\
 0,&  \hbox{for $t\in[0,T], |x|\geq  3 \rho_0/4$},\\
\end{cases}
\end{equation*}
and $0\leq\eta\leq1$ everywhere. Define
\begin{eqnarray}\label{0-0}
&&\eta^z(t,x) := \eta(t,x-z),\qquad \widehat{\eta}^z(t,y):=\eta^z(t,x), \nonumber \\
&&\widehat{L}_t^z(y) := \widehat{\eta}^z(t,y)\widetilde{L}^z_t(y)+(1-\widehat{\eta}^z(t,y))(\Delta+\partial_t).
\end{eqnarray}
Let $\zeta^z$ be the function introduced in Lemma \ref{le1}.   It is easy to verify
\begin{equation}\label{3}
  \widehat{L}_t^z(\zeta^z\cdot)=\widetilde{L}_t^z(\zeta^z\cdot).
\end{equation}

Now, we first deal with the case that $b$ is identically equal to zero in operator $L_t$ and recall that  \\
\ce
L^0_t:=\partial_t+\frac{1}{2}\sum^d_{i,j=1}\sum^d_{k=1}(\sigma^{ik}\sigma^{jk})(t,x)\frac{\partial^2}{\partial{x_i}\partial{x_j}}\,,
\de
where the coefficient $\sigma$ is defined on $D$ and satisfies  the condition
(${\bf H}^\sigma$) on $D$. By Proposition \ref{prop6.01} we can extend $\sigma$  to a function on    $\RR^d$ satisfying
(${\bf H}^\sigma$) on the whole space $\RR^d$.
The  relation  \eqref{3}   and  \cite[Theorem 7.2.8 with $b=0$]{Krylov1} allow us to find a $\lambda_0\geq1$ depending only on $\kappa,\rho_0,p,q,d,\alpha$ such that for $\lambda\geq\lambda_0$,  the operator
\ce
\lambda-\widehat{L}_t^z=\lambda I-\widehat{L}_t^z\  : \mW^{2,q}_p([0,T]\times\mR^d)\rightarrow \mL^q_p([0,T]\times\mR^d)
\de
is invertible,  where $I$ is the identity operator. On the other hand,  applying a time dependent version of    \cite[Lemma 8.2.1]{Krylov}  yields that  for $\lambda\geq\lambda_0$  the operator
\begin{equation}\label{l1}
\lambda-\widehat{L}_t^z: {\mathring\mW}^{2,q}_p([0,T]\times\mR_+^d)\rightarrow \mL^q_p([0,T]\times\mR_+^d)
\end{equation}
is invertible. We denote this  inverse   by
$$
\widehat{\cR}^z_{\lambda,t}:=\left(\lambda-\widehat{L}_t^z\right)^{-1}\,.
$$
 Define
\begin{eqnarray}\label{0-1}
\Psi &: & w = w(t,y)\rightarrow \Psi w(t,x)=w(t,\psi(x)),\nonumber\\
\Psi^{-1} &:&v = v(t,x)\rightarrow \Psi^{-1} v(t,y)=v(t,\psi^{-1}(y)),\nonumber\\
R^z_{\lambda,t}&:&  f = f(t,x)\rightarrow R^z_{\lambda,t}f(t,x)=\Psi\widehat{\cR}^z_{\lambda,t}\Psi^{-1}[\eta^zf](t,x)\,. \nonumber\\
\end{eqnarray}
From the definitions  of $\Psi$ and $\Psi^{-1}$ it follows
\ce
\widehat{v}=\Psi^{-1}v,~~~~v=\Psi \widehat{v}
\de
and   the identities   \eqref{2} and (\ref{3}) imply that
\begin{equation}\label{4}
\begin{split}
&\widehat{L}^z_t\Psi^{-1}(\zeta^zv)=\widetilde{L}^z_t\Psi^{-1}(\zeta^zv)=\Psi^{-1}L_t(\zeta^zv),\\
&\Psi\widehat{L}^z_t\Psi^{-1}(\zeta^zv)=L_t(\zeta^zv), ~~~\widehat{L}^z_t(\widehat{\zeta}^z\widehat{v})=\Psi^{-1}L_t\Psi(\widehat{\zeta}^z\widehat{v}).
\end{split}
\end{equation}

\begin{theorem}\label{th0}
(i) If $\zeta^zv\in{\mathring\mW}^{2,q}_p([0,T]\times D)$, then for any $\lambda\geq\lambda_0$, we have
\ce
\zeta^zv=R^z_{\lambda,t}(\lambda-L^0_t)(\zeta^zv)
\quad \hbox{on $[0,T]\times(B_{\rho_0}(z)\cap D)$}.
\de
(ii) There is a constant $C$ depending only on $\kappa,\rho_0,p,q,d,\alpha$ such that for $\lambda\geq\lambda_0$ and $f\in L^q_p([0,T]\times D)$, we have
\ce
\zeta^zR^z_{\lambda,t}f\in{\mathring\mW}^{2,q}_p([0,T]\times(B_{\rho_0}(z)\cap D)),
\de
and
\begin{gather}\label{5}
\lambda\|\zeta^zR^z_{\lambda,t}f\|_{\mL^q_p([0,T]\times D)}+\lambda^{1/2}\|\zeta^zR^z_{\lambda,t}f\|_{\mW^{1,q}_p([0,T]\times D)}
  +\|\zeta^zR^z_{\lambda,t}f\|_{\mW^{2,q}_p([0,T]\times D)}
  \nonumber \\
\qquad\qquad+\|\partial_t(\zeta^zR^z_{\lambda,t}f)\|_{\mL^q_p([0,T]\times D)}
  \leq C\|f\|_{\mL^q_p([0,T]\times(B_{\rho_0}(z)\cap D))}.
\end{gather}
\end{theorem}
\begin{proof}(i) Define $w:=\zeta^zv,  f:=\lambda w-L^0_tw$. Then by Lemmas \ref{le1} and \ref{le2} we have $\widehat{w}\in{\mathring\mW}^{2,q}_p([0,T]\times D^z_+)$.
Since $\eta^zf=f$ in $([0,T]\times (B_{\rho_0/2}(z))\cap D)$,  we see
\ce
\widehat{f}=\Psi^{-1}f=\Psi^{-1}[\eta^zf].
\de
This  yields
\ce
\widehat{w}(t,y)=(\lambda-\widehat{L}^z_t)^{-1}\widehat{f}(t,y)=\widehat{\cR}^z_{\lambda,t}\widehat{f}(t,y)=\widehat{\cR}^z_{\lambda,t}\Psi^{-1}[\eta^zf](t,y),
\de
and
\ce
w(t,x)=\Psi\widehat{w}(t,x)=\Psi\widehat{\cR}^z_{\lambda,t}\Psi^{-1}[\eta^zf](t,x)=R^z_{\lambda,t}f(t,x)=R^z_{\lambda,t}(\lambda-L^0_t)(\zeta^zv)
\de
for $y\in D^z_+$ and $x=\psi^{-1}(y)\in B_{\rho_0/2}(z))\cap D$.

\noindent (ii) By  Lemma \ref{le1} and  equation (\ref{l1}), we have for $\lambda\geq\lambda_0\geq1$,
\ce
&&\lambda\|\zeta^zR^z_{\lambda,t}f\|_{\mL^q_p([0,T]\times D)}+\lambda^{1/2}\|\zeta^zR^z_{\lambda,t}f\|_{\mW^{1,q}_p([0,T]\times D)}\\
&\;&\qquad\qquad +\|\zeta^zR^z_{\lambda,t}f\|_{\mW^{2,q}_p([0,T]\times D)}+\|\partial_t(\zeta^zR^z_{\lambda,t}f)\|_{\mL^q_p([0,T]\times D)}\\
&&\leq C\bigg(\lambda\|\widehat{\zeta}^z\widehat{\cR}^z_{\lambda,t}\widehat{f}\|_{\mL^q_p([0,T]\times {\mR}^d_+)}+\lambda^{1/2}\|\widehat{\zeta}^z\widehat{\cR}^z_{\lambda,t}\widehat{f}\|_{\mW^{1,q}_p([0,T]\times {\mR}^d_+)}\\
&&\qquad\qquad +\|\widehat{\zeta}^z\widehat{\cR}^z_{\lambda,t}\widehat{f}\|_{\mW^{2,q}_p([0,T]\times {\mR}^d_+)}+\|\partial_t(\widehat{\zeta}^z\widehat{\cR}^z_{\lambda,t}\widehat{f})\|_{\mL^q_p([0,T]\times {\mR}^d_+)}\bigg)\\
&&\leq C\Big(\lambda\|\widehat{\cR}^z_{\lambda,t}\Psi^{-1}[\eta^zf]\|_{\mL^q_p([0,T]\times {\mR}^d_+)}+\lambda^{1/2}\|\widehat{\cR}^z_{\lambda,t}\Psi^{-1}[\eta^zf]\|_{\mW^{1,q}_p([0,T]\times {\mR}^d_+)}\\
&\;&\qquad\qquad +\|\widehat{\cR}^z_{\lambda,t}\Psi^{-1}[\eta^zf]\|_{\mW^{2,q}_p([0,T]\times {\mR}^d_+)} +\|\partial_t(\widehat{\cR}^z_{\lambda,t}\Psi^{-1}[\eta^zf])\|_{\mL^q_p([0,T]\times {\mR}^d_+)}\Big)\\
&&\leq C\Big(\lambda\|\widehat{{\cR}}^z_{\lambda,t}\Psi^{-1}[\eta^zf]\|_{\mL^q_p([0,T]\times {\mR}^d_+)}+\lambda^{1/2}\|D\widehat{{\cR}}^z_{\lambda,t}\Psi^{-1}[\eta^zf]\|_{\mL^q_p([0,T]\times {\mR}^d_+)}\\
&\;&\qquad\qquad +\|D^2\widehat{{\cR}}^z_{\lambda,t}\Psi^{-1}[\eta^zf]\|_{\mL^q_p([0,T]\times {\mR}^d_+)} +\|\partial_t(\widehat{{\cR}}^z_{\lambda,t}\Psi^{-1}[\eta^zf])\|_{\mL^q_p([0,T]\times {\mR}^d_+)}\Big)\\
&&\leq C\|\Psi^{-1}[\eta^zf]\|_{\mL^q_p([0,T]\times\mR^d_+)}\leq C\|\eta^zf\|_{\mL^q_p([0,T]\times(B_{\rho_0}(z)\cap D))}\\
&&\leq C\|f\|_{\mL^q_p(([0,T])\times(B_{\rho_0}(z)\cap D))}.
\de
The theorem is then proved.
\end{proof}

\begin{lemma}\label{re3}If $f\in \mL^q_p([0,T]\times(B_{\rho_0}(z)\cap D))$, then
$$(\lambda-L^0_t)R^z_{\lambda,t}f=f \quad  \hbox{  on $[0,T]\times(B_{\rho_0/2}(z)\cap D)$}\,.
$$
\end{lemma}
\begin{proof}
Denote $f_1:=R^z_{\lambda,t}f$, and $f_2:=(\lambda-L^0_t)f_1$.  It suffices to prove $f=f_2$ in $[0,T]\times(B_{\rho_0/2}(z)\cap D)$.
Since
\begin{equation*}
f\in \mL^q_p([0,T]\times(B_{\rho_0}(z)\cap D)) \Rightarrow f_1 \in \mW^{2,q}_p([0,T]\times(B_{\rho_0}(z)\cap D)),
\end{equation*}
we have  by Lemma \ref{le2} (ii)
\begin{equation*}
\widehat{f}_1\in \mW^{2,q}_p([0,T]\times D^z_+)
\quad\hbox{and} \quad \widehat{f}_2=(\lambda-\widehat{L}^z_t)\widehat{f}_1  \quad \hbox{in $[0,T]\times D^z_+$}\,.
\end{equation*}
Therefore, we have $\widehat{f}_1=(\lambda-\widehat{L}^z_t)^{-1}\widehat{f}_2=\widehat{\cR}^z_{\lambda,t}\widehat{f}_2$ in $[0,T]\times D^z_+$. On the other hand, by definition of $\eta^z$, we see  $\eta^zf=f$ in $[0,T]\times B_{\rho_0/2}(z)$,  and then we have
\ce
 \widehat{f}_1&=&\Psi^{-1}f_1=\Psi^{-1}R^z_{\lambda,t}f=\Psi^{-1}(\Psi\widehat{\cR}^z_{\lambda,t}\Psi^{-1}[\eta^zf])\\
&  =&\widehat{\cR}^z_{\lambda,t}\Psi^{-1}[\eta^zf]=\widehat{\cR}^z_{\lambda,t}\Psi^{-1}f=\widehat{\cR}^z_{\lambda,t}\widehat{f}.
\de
This yields
\ce
\widehat{\cR}^z_{\lambda,t}\widehat{f}_2=\widehat{\cR}^z_{\lambda,t}\widehat{f}.
\de
Since  $\widehat{\cR}^z_{\lambda,t}$ and $\Psi^{-1}$ are linear operators, we conclude  $f=f_2$ in $[0,T]\times(B_{\rho_0/2}(z)\cap D)$.
\end{proof}

\subsection{Piece together the neighborhoods of boundary points}
Take a function $\xi\in C^\infty_0( \mR^d)$ such that $0\leq\xi\leq1$,   and
\ce
\xi(x)=
\begin{cases}
0 &\hbox{$~~~when  ~~~  ~~|x|\geq\rho_0/2,  ~~~~$}\\
1 &\hbox{$~~~ when  ~~~  ~~|x|\leq\rho_0/4$}\,. \\
\end{cases}
\de
Next, take points $z_1,z_2,\ldots,z_n\in\partial D$ so that
\ce
\begin{cases}|z_i-z_j|\geq\rho_0/8\quad \hbox{for $i\neq j$  and }\\
\hbox{the boundary  $\partial D$ is covered by $B_{\rho_0/8}(z_i)$\,. }
\end{cases}
\de
This is possible because we assume that  $D$ is bounded  so that
 the number $n$ of the points $z_i$  can be chosen to be finite. Observe that $n$ can be estimated through $\rho_0,d$ and $\diam(D)$.

Denote
\begin{equation}
\xi^i( x)=\xi( x-z_i),~~i=1,2,\ldots,n.
\label{e.def_xii}
\end{equation}
These function are used to smooth a function near  the boundary $\partial D$ of the
domain $D$.
To  smooth a function inside the domain we introduce another function
$\xi^0\in C^\infty_0(  \mR^d)$ such that $0\leq\xi^0\leq1$,
\begin{equation}\label{sec3-eq00}
\left\{
                    \begin{array}{lll}
                     \xi^0( x)=0 ~~~for ~~~x\in D ~~~with ~~~dist(x,\partial D)\leq\rho_0/16,\\
                     \xi^0( x)=1 ~~~for ~~~x\in D ~~~with ~~~dist(x,\partial D)\geq\rho_0/8.
                    \end{array}
                    \right.
\end{equation}
This is possible, for instance, by mollifying the indicate function  of
\ce
 D\setminus\{x:dist(x,\partial D)\leq3\rho_0/32\} .
\de
Notice that
\ce
\sum_{i=1}^n(\xi^i( x))^2\geq1 ~~if ~~~x\in \overline{D}~~~and ~~~dist(x,\partial D)\leq\rho_0/8.
\de
Therefore,
\ce
\overline{\xi}:=\sum_{i=0}^n(\xi^i( x))^2\ge  1~~~in~~ \overline{D}.
\de
Moreover, $\overline{\xi}$ and its  derivatives of any order are
bounded in $\overline{D}$ by a constant  depending only on $d,n,\rho_0$, and the order of the derivative. Finally, define
\begin{equation}\label{zetai}
\zeta^i( x):=\xi^i\overline{\xi}^{-1/2},   ~~~i=0,1,\ldots,n.
\end{equation}
It is easy to see  that  all $\zeta^i$ are infinitely differentiable and
\ce
\sum_{i=0}^n(\zeta^i( x))^2=1\,,\quad \hbox{on
$[0,T]\times D$}\,.
\de
Denote
\ce
R^{(i)}_{\lambda,t}=R^{z_i}_{\lambda,t}, ~~i=1,\ldots,n
\de
for $\lambda\geq\lambda_0\geq1$ and $z_i\in\partial D$.
Let $R^{(0)}_{\lambda,t}$ be the inverse operator of  $(\lambda-L^0_t): \mW^{2,q}_p([0,T]\times\mR^d)\longrightarrow \mL^q_p([0,T]\times\mR^d)$.

We may increase $\lambda_0$ if needed   so that \cite[Theorem 7.2.8]{Krylov1} can  be   applied
(with $b=0$) to yield that for $\lambda\geq\lambda_0$
\begin{align}
\lambda\|R^{(0)}_{\lambda,t}f\|_{\mL^q_p([0,T]\times\mR^d)}&+\lambda^{1/2}\|R^{(0)}_{\lambda,t}f\|_{\mW^{1,q}_p([0,T]\times\mR^d)} +\|R^{(0)}_{\lambda,t}f\|_{\mW^{2,q}_p([0,T]\times\mR^d)}\nonumber \\
&+\|\partial_tR^{(0)}_{\lambda,t}f\|_{\mL^q_p([0,T]\times\mR^d)}\leq C\|f\|_{\mL^q_p([0,T]\times\mR^d)}\,. \label{eq12}
\end{align}
\begin{lemma}\label{le7} Let   $u\in{\mathring\mW}^{2,q}_p([0,T]\times D)$ and set  $
f=\lambda u-L^0_tu $. If $\lambda\geq\lambda_0$,  then in $[0,T]\times D$
\begin{equation}\label{7}
u  =\sum_{0\leq k\leq n}\zeta^kR^{(k)}_{\lambda,t}(\zeta^kf-L_{t,k}u ),
\end{equation}
where
\begin{equation}\label{li}
L_{t,k}u:=a^{ij}\zeta^k_{x^ix^j}u+2a^{ij}\partial_i\zeta^k\partial_ju.
\end{equation}
\end{lemma}
\begin{proof} Since $\zeta^{0}u\in\mW^{2,q}_p([0,T]\times \mR^d)$ and by the definition of $R^{(0)}_{\lambda,t}$,  we have
\ce
\zeta^0u=R^{(0)}_{\lambda,t}(\lambda(\zeta^{0}u)-L^0_t(\zeta^{0}u))\,.
\de
Thus by Theorem \ref{th0} (i)  we have
\ce
u&=&u\sum_{k=0}^n(\zeta^k( x))^2=\sum_{k=0}^n\zeta^k(\zeta^ku)=\sum_{k=0}^n\zeta^kR^{(k)}_{\lambda,t}((\lambda-L^0_t)\zeta^ku)\\
&=&\sum_{k=0}^n\zeta^kR^{(k)}_{\lambda,t}(\lambda(\zeta^ku)-\zeta^kL^0_tu-L_{t,k}u)
=\sum_{k=0}^n\zeta^kR^{(k)}_{\lambda,t}(\zeta^kf-L_{t,k}u).
\de
This completes  the proof.
\end{proof}
\subsection{Zero drift term and zero Dirichlet boundary condition}
Now we turn   to study
\begin{equation}\label{9}
\lambda u-L^0_tu=f
\end{equation}
in ${\mathring\mW}^{2,q}_p([0,T]\times D)$.
  Instead of solving this  equation directly,  we shall   solve equation (\ref{7}), which is equivalent to \eqref{9}. In fact,  by Lemma \ref{le7} we see that if $u$ satisfies \eqref{9}, then
  it satisfies  \eqref{7}.   In next lemma we prove the opposite part of Lemma \ref{le7}, namely,    any solution of equation (\ref{7}) is indeed a solution of (\ref{9}).

\begin{lemma}\label{le10}  Let  $\lambda\geq\lambda_0$
      and let  $f\in \mL^q_p([0,T]\times D)$,  where $\lambda_0$ is the one in \cite[Theorem 7.2.8]{Krylov1}.  If $u\in{\mathring\mW}^{1,q}_p([0,T]\times D)$ is a solution of (\ref{7}), then $u\in{\mathring\mW}^{2,q}_p([0,T]\times D)$. Furthermore, there exists a constant $\lambda_1\geq (1\vee \lambda_0)$, depending only on $d,p,q,\alpha,\kappa,\rho_0$ and $\diam(D)$, such that  for all
$\lambda\ge \lambda_1$, the solution $u$ of
  \eqref{7} satisfies (\ref{9}) in $[0,T]\times D$.
\end{lemma}
\begin{proof} Let $u\in{\mathring\mW}^{1,q}_p([0,T]\times D)$ be  a solution of (\ref{7}).
Owing to the  fact that $L_{t,i}$, $(i=1,\ldots,n)$
defined by (\ref{li}) are  first order differentiable operators, we have
\ce
\zeta^if-L_{t,i}u\in \mL^q_p([0,T]\times D).
\de
Hence by Lemma \ref{re3}
\ce
\zeta^iR^{(i)}_{\lambda,t}(\zeta^if-L_{t,i}u)\in {\mathring\mW}^{2,q}_p([0,T]\times D).
\de
From  the expression   (\ref{7}) for $u$ it is easy to see
that $u\in{\mathring\mW}^{2,q}_p([0,T]\times D)$.

Next, denote $h:=\lambda u-L^0_tu$.    To complete  the proof, we only need   to
show that as long as  $\lambda$ is sufficiently large, then $f=h$.
By Lemma \ref{le7},  $u=\sum_{0\leq i\leq n}\zeta^iR^{(i)}_{\lambda,t}(\zeta^ih-L_{t,i}u)$.
On the other hand,  $u$ is a solution of (\ref{7}),  ie,  $u=\sum_{0\leq i\leq n}\zeta^iR^{(i)}_{\lambda,t}(\zeta^if-L_{t,i}u)$.

Thus, we have
\ce
0=\sum_{0\leq i\leq n}\zeta^iR^{(i)}_{\lambda,t}(\zeta^if-L_{t,i}u)-\sum_{0\leq i\leq n}\zeta^iR^{(i)}_{\lambda,t}(\zeta^ih-L_{t,i}u)=\sum_{0\leq i\leq n}\zeta^iR^{(i)}_{\lambda,t}(\zeta^i(f-h))\,.
\de
This means that   ${\bf R}_{\lambda,t}(f-h)=0$, where
\ce
{\bf R}_{\lambda,t}f=\sum_{0\leq i\leq n}\zeta^iR^{(i)}_{\lambda,t}(\zeta^if)\,,
\de
which is an operator associated with  the operator $\lambda  -L^0_t$
(It is a regularizer    of the inverse of $\lambda  -L^0_t$).
  Thus, we only need to show  the following:
\begin{equation}
\begin{split}
&\hbox{If   $\upnu\in L^q_p([0,T]\times D)$  and if  ${\bf R}_{\lambda,t}(\upnu)=0$ for   $\lambda\ge \lambda_0$,}\\
&\qquad \qquad \hbox{  then $\upnu=0$ in $[0,T]\times D$\,.
}\end{split}
\label{e.critical1}
\end{equation}
This is the objective of the following  paragraphs.
By Lemma \ref{re3} we have
\ce
(\lambda-L^0_t)R^{(i)}_{\lambda,t}(\zeta^i)=\zeta^i
\de
on the set in $(0,T)\times D$,   where $\zeta^i $ is defined by (\ref{zetai}). Therefore, if    ${\bf R}_{\lambda,t}(\upnu)=0$
then we have
\ce
0& =&(\lambda-L^0_t){\bf R}_{\lambda,t}(\upnu)=\lambda{\bf R}_{\lambda,t}(\upnu)-\sum_{0\leq i\leq n}L^0_t(\zeta^iR^{(i)}_{\lambda,t}(\zeta^i(\upnu)))\\
& =&\sum_{0\leq i\leq n}\lambda(\zeta^iR^{(i)}_{\lambda,t}(\zeta^i(\upnu)))-\left(\sum_{0\leq i\leq n}\zeta^i(L^0_tR^{(i)}_{\lambda,t}(\zeta^i(\upnu)))+\sum_{0\leq i\leq n}L_{t,i}R^{(i)}_{\lambda,t}(\zeta^i(\upnu))\right)\\
& =&\sum_{0\leq i\leq n}\zeta^i(\lambda-L^0_t)R^{(i)}_{\lambda,t}(\zeta^i(\upnu))-\sum_{0\leq i\leq n}L_{t,i}R^{(i)}_{\lambda,t}(\zeta^i(\upnu))\\
& =&\sum_{0\leq i\leq n}(\zeta^i)^2(\upnu)-\sum_{0\leq i\leq n}L_{t,i}R^{(i)}_{\lambda,t}(\zeta^i
(\upnu)):=\upnu-T_\lambda (\upnu),
\de
where
\ce
T_\lambda(\upnu):=\sum_{0\leq i\leq n}L_{t,i}R^{(i)}_{\lambda,t}(\zeta^i(\upnu)).
\de
To finish the proof  of \eqref{e.critical1}, it suffices to show that for sufficiently large
$\lambda$  the operator $T_\lambda$ is a contraction  operator in $\mL^q_p((0,T)\times D)$. In fact, from
\ce
L_{t,k}u:=a^{ij}\zeta^k_{x^ix^j}u+2a^{ij}\partial_i\zeta^k\partial_ju
\de
 and  from (\ref{eq12}) and Theorem \ref{th0} (ii) it follows that
\ce
 \|T_\lambda \upnu\|_{\mL^q_p([0,T]\times D)}
&    \leq& C \lambda^{-1}\sum^n_{i=1}\|\zeta^i(\upnu)\|_{\mL^q_p([0,T]\times(B_{\rho_0}(z_i)\cap D))} \\
& &+ C\lambda^{-1}\|\zeta^0(\upnu)\|_{\mL^q_p([0,T]\times D)}\\
&\leq  &  C\lambda^{-1}\|\upnu\|_{\mL^q_p([0,T]\times D)}\,.
\de
Since the above  constant $C$ does not depend on $\lambda$, we may choose   $\lambda$ sufficiently large so   that $C \lambda^{-1}<1$. This means that
$T_\lambda$ is a contraction since $T_\lambda$ is a linear operator.
Thus $\upnu=0$ in $[0,T]\times D$  and
\eqref{e.critical1} is proved. The proof of lemma is   then completed.
\end{proof}

The following lemma will show that  the equation (\ref{7})  has a unique solution
satisfying some desired estimate.
\begin{lemma}\label{le11} There exists a constant $\lambda_1\geq1$, depending only on $d,p,q,\alpha,\kappa,\rho_0$ and $\diam(D)$, such that for any   $f\in \mL^q_p([0,T]\times D)$, there exists a unique solution $u\in{\mathring\mW}^{1,q}_p([0,T]\times D)$ of equation (\ref{7})
(for any fixed $\lambda\geq\lambda_1$). Furthermore, this solution satisfies
\begin{equation}\label{13}
\lambda\|u\|_{\mL^q_p([0,T]\times D)}+\lambda^{1/2}\|u\|_{\mW^{1,q}_p([0,T]\times D)}+\|\partial_tu\|_{\mL^q_p([0,T]\times D)}\leq C\|f\|_{\mL^q_p([0,T]\times D)}\,,
\end{equation}
where $C$ depending only on $d,p,q,\alpha,\kappa,\rho_0$ and $\diam(D)$.
\end{lemma}
\begin{proof}Define
\ce
F(v):=\sum_{0\leq i\leq n}\zeta^iR^{(i)}_{\lambda,t}(\zeta^if-L_{t,i}v)
\de
and let $u=F(v)$ for $v\in{\mathring\mW}^{1,q}_p([0,T]\times D)$. \\
Thus we have $u\in{\mathring\mW}^{1,q}_p([0,T]\times D)$. Owing to Theorem \ref{th0} (ii)  we have
\cen
\lambda\|u\|_{\mL^q_p([0,T]\times D)}&\leq&\sum_{0\leq i\leq n}\lambda\|\zeta^iR^{(i)}_{\lambda,t}(\zeta^if-L_{t,i}v)\|_{\mL^q_p([0,T]\times D)}\nonumber \\
&\leq&C\sum_{0\leq i\leq n}\|\zeta^if-L_{t,i}v\|_{\mL^q_p([0,T]\times D)}
\nonumber \\
&\leq& C\Big(\|f\|_{\mL^q_p([0,T]\times D)}+\|v\|_{\mW^{1,q}_p([0,T]\times D)}\Big).\label{e.l38_1}
\den
In the same way we can obtain
\cen
\lambda^{1/2}\|u\|_{\mW^{1,q}_p([0,T]\times D)}\leq C\Big(\|f\|_{\mL^q_p([0,T]\times D)}+\|v\|_{\mW^{1,q}_p([0,T]\times D)}\Big).\label{e.l38_2}
\den
Furthermore,  we have also
\cen
&&\|\partial_tu\|_{\mL^q_p([0,T]\times D)}=\Big\|\partial_t\left(\sum_{0\leq i\leq n}\zeta^iR^{(i)}_{\lambda,t}(\zeta^if-L_{t,i}v)\right)\Big \|_{\mL^q_p([0,T]\times D)}\nonumber\\
&&\qquad\qquad \leq  \sum_{1\leq i\leq n}\|\partial_t(\zeta^iR^{(i)}_{\lambda,t}(\zeta^if-L_{t,i}v)\|_{\mL^q_p([0,T]\times D)}\nonumber\\
&&\qquad\qquad \qquad +\|\zeta^0\partial_t(R^{(0)}_{\lambda,t}(\zeta^0f-L_{t,0}v))\|_{\mL^q_p([0,T]\times D)}
\nonumber \\
&&\qquad\qquad\qquad   +\|(\partial_t\zeta^0)R^{(0)}_{\lambda,t}(\zeta^0f-L_{t,0}v)\|_{\mL^q_p([0,T]\times D)}
\nonumber \\
&&\qquad\qquad\leq C\sum_{0\leq i\leq n}\|\zeta^if-L_{t,i}v\|_{\mL^q_p([0,T]\times D)}\nonumber\\
&&\qquad\qquad \leq C\Big(\|f\|_{\mL^q_p([0,T]\times D)}+\|v\|_{\mW^{1,q}_p([0,T]\times D)}\Big).
\label{e.l38_3}
\den
These three estimates
\eqref{e.l38_1}-\eqref{e.l38_3} yield
\cen
&&\lambda\|u\|_{\mL^q_p([0,T]\times D)}+\lambda^{1/2}\|u\|_{\mW^{1,q}_p([0,T]\times D)}+\|\partial_tu\|_{\mL^q_p([0,T]\times D)}\nonumber \\
&&\qquad\qquad\leq C\left(\|f\|_{\mL^q_p([0,T]\times D)}+\|v\|_{\mW^{1,q}_p([0,T]\times D)}\right).
\label{3.27}
\den
Let  $u_1=F(v_1)$ and $u_2=F(v_2)$. By the Affinity of $F$ and the  above inequality (\ref{3.27}) we have
\ce
\begin{split}
\lambda\|F(v_1-v_2)\|_{\mL^q_p([0,T]\times D)}&+\lambda^{1/2}\|F(v_1-v_2)\|_{\mW^{1,q}_p([0,T]\times D)}\\
&+\|\partial_tF(v_1-v_2)\|_{\mL^q_p([0,T]\times D)}\leq C\|v_1-v_2\|_{\mW^{1,q}_p([0,T]\times D)}.
\end{split}
\de
\\
Take  $\lambda\geq\lambda_1$ large enough so that $C\lambda^{-1}<1$  and   we see that $F(v)$ is a  contraction. Thus   $F(u)=u$ has a unique fixed point and this implies that equation (\ref{7}) admits  a unique solution $u\in{\mathring\mW}^{1,q}_p([0,T]\times D)$.
The inequality \eqref{13} follows straightforward from  the inequality \eqref{3.27}
with $v$ replaced by $u$.
\end{proof}

\subsection{Zero drift term and general  Dirichlet boundary condition}
\begin{theorem}\label{th12} Let  $\lambda_1$ be sufficiently large and let $\partial D\in C^2$. Assume that $\sigma$ satisfies $({\bf H}^\sigma)$.
Then for any $\lambda\geq\lambda_1$, $f\in \mL^q_p([0,T]\times D)$ and $g\in\mW^{2,q}_p([0,T]\times  D)$, there exists a unique solution $u_0\in\mW^{2,q}_p([0,T]\times D)$ to  the equation
\begin{equation}\label{14.0}
\lambda u_0-L^0_tu_0=f
\end{equation}
in $(0,T)\times D$  such that $u_0=g$ on $(0,T)\times\partial D$, and
\begin{equation}\label{14}
\begin{split}
\lambda\|u_0-g\|_{\mL^q_p([0,T]\times D)}
&+\lambda^{1/2}\|u_0-g\|_{\mW^{2,q}_p([0,T]\times D)}+\|\partial_t(u_0-g)\|_{\mL^q_p([0,T]\times D)}\\
&\leq C\left(\|f\|_{\mL^q_p([0,T]\times D)}+  C_{ \lambda}  \|g\|_{\mW^{2,q}_p([0,T]\times D)}\right),
\end{split}
\end{equation}
where $C$ depending only on $d,p,q,\alpha,\kappa,\rho_0$ and $\diam(D)$ and is independent of
$\lambda$,
and  $C_{\lambda} $ depends only on $\lambda$.
\end{theorem}
\begin{proof} Let us recall that $u_0=g$ on $(0,T)\times\partial D$ means that
$u_0-g\in {\mathring\mW}^{2,q}_p([0,T]\times D)$. Choose $\lambda_1$ to be sufficiently large such that both  Lemmas \ref{le10} and \ref{le11} hold true.
We shall apply Lemma \ref{le11}  to $u_0-g$.  Since $u_0-g\in{\mathring\mW}^{2,q}_p([0,T]\times D)$ and
\ce
\lambda (u_0-g)-L^0_t(u_0-g)=f+L^0_tg-\lambda g,
\de
by inequality (\ref{13}), we see
\ce
\lambda\|u_0-g\|_{\mL^q_p([0,T]\times D)}&+&\lambda^{1/2}\|u_0-g\|_{\mW^{2,q}_p([0,T]\times D)}+\|\partial_t(u_0-g)\|_{\mL^q_p([0,T]\times D)}\\
&&\leq C\|f+L^0_tg-\lambda g\|_{\mL^q_p([0,T]\times D)}\\
&&\leq C\left(\|f\|_{\mL^q_p([0,T]\times D)}+C_{\kappa,\lambda}\|g\|_{\mW^{2,q}_p([0,T]\times D)})\,. \right.
\de
The theorem is proved.
\end{proof}

In the above theorem, we need to assume $\lambda \ge \lambda_1$ for some $\lambda_1>0$.
We can improve this condition to $\lambda\ge 0$  by observing the following fact:  $u_0$ satisfies (\ref{14.0}) if and only if $v_\lambda(t,x):=e^{-\lambda t}u_0(t,x)$ satisfies
\ce
\partial_tv+L^\sigma v+e^{-\lambda t}f=0 ~~~in~~~ (0,T)\times D
\de
and $v_\lambda=e^{-\lambda t}g$ on $[0,T]\times\partial D$ for $\lambda\geq\lambda_1$.

\begin{theorem}\label{th12.1} Let $\partial D\in C^2$. Assume that $b=0$ and $\sigma$ satisfies $({\bf H}^\sigma)$ on the domain $D$.
Then for any $f\in \mL^q_p([0,T]\times D)$ and $g\in\mW^{2,q}_p([0,T]\times   D)$, there exists a unique solution $u_0\in\mW^{2,q}_p([0,T]\times D)$ to  the equation
\ce
\partial_t u_0+L^\sigma u_0=f
\de
in $(0,T)\times D$  such that $u_0=g$ on $(0,T)\times\partial D$, and
\begin{equation}\label{a.14}
\begin{split}
\|u_0-g\|_{\mW^{2,q}_p([0,T]\times D)}
&\leq C\left(\|f\|_{\mL^q_p([0,T]\times D)}+ \|g\|_{\mW^{2,q}_p([0,T]\times D)}\right),
\end{split}
\end{equation}
where $C$ depending only on $p,q,\alpha,\kappa,\rho_0$ and $\diam(D)$.

Furthermore, if $p,q\in(1,\infty)$,  $f\in\mL^q_p([0,T]\times D)$, $g\in\mW^{2,q}_p([0,T]\times D)$,   then for any $\beta\in[0,2)$ and $\gamma>1$ with $\frac{d}{p}+\frac{2}{q}<2-\beta+\frac{d}{\gamma}$, we have
\begin{equation}\label{b.14}
||u_0(t)||_{\mH^\beta_{\gamma}(D)}\leq C(T-t)^{(2-\beta)/2-d/2p-1/q+d/2\gamma}\\
\times \left(||f||_{\mL^q_p([t,T]\times D)}+||g||_{\mL^q_p([t,T]\times D)}\right),
\end{equation}
where $C=C(d,\kappa,p,q,\beta,\gamma,\alpha)$ is a positive constant independent of $t$.
\end{theorem}

\begin{proof}We only need to  prove estimate (\ref{b.14}) and we shall do this by using mollifying technique and weak convergence argument. Let $\varrho$ be a nonnegative smooth function on $[0,T]\times\mR^d$ with support in $\{(t,x)\in[0,T]\times\mR^d:t+|x|^2\leq1\}$ and $\int_{[0,T]\times\mR^d}\varrho(t,x)dxdt=1$. Set $\varrho_n(t,x):=n^{d+1}\varrho(nt,nx)$. Define
\ce
\sigma_n:=\sigma\ast\varrho_n,~~~~~~~~~~~~f_n:=f\ast\varrho_n, ~~~~~~~g_n:=g\ast\varrho_n.
\de
It is a classical fact (e.g. \cite[Chapter I]{Friedman}) that for a bounded domain $D$ of $\mR^d$, there exists the parabolic Dirichlet-Green function  for the operator $\partial_t+L^{\sigma_n}$, restricted on $D$ with  boundary value $g$, denoted by $G_{D,g,n} (s,x;t,y)$.

Define
\ce
\cT_{s,t}^n f (t,x):=\int_Df (t,y)G_{D,g,  n} (s,x;t,y)dy\,,\quad \forall \ f\in C^{1,2}([0, T]\times D)\,.
\de
The function $\cT_{s,t} ^nf_n (t,x)$ satisfies that for all $(t,x)\in[0,T]\times D$,
\begin{eqnarray}\label{017.1}
&&\partial_s\cT_{s,t}^n f_n(t,x)+L^{\sigma_n}\cT_{s,t}^n f_n(t,x)=0,\nonumber \\
&&\qquad \lim_{s\uparrow t}\cT_{s,t}^n f_n(t,x)=f_n(t,x).
\end{eqnarray}
Furthermore, for all $x,y\in D$ and $0\leq s\leq t\leq T$, using (1.4.9) with $\mu=d/2$, (1.4.15) and Lemma 1.4.3 in \cite[Chapter I]{Friedman},  we have for any $0<\kappa^*<\kappa$,
\begin{equation}\label{18.01}
|G_{D,g , n}(s,x;t,y)| \leq C_{d,\alpha,\kappa^*}(t-s)^{-\frac{d}{2}}exp\Big[-\frac{\kappa^*|x-y|^2}{2(t-s)}\Big]\,.
\end{equation}
Using \cite[(1.4.10)]{Friedman}  with $\mu=(d+1)/2$, \cite[(1.4.15)]{Friedman}  and   in \cite[Lemma 1.4.3]{Friedman},  we have for any $0<\kappa^*<\kappa$,
\begin{equation}\label{18.02}
|\nabla_xG_{D,g, n}(s,x;t,y)| \leq C_{d,\alpha,\kappa^*}(t-s)^{-\frac{d+1}{2}}exp\Big[-\frac{\kappa^*|x-y|^2}{2(t-s)}\Big] \,.
\end{equation}
Using \cite[(1.4.11)]{Friedman}  with $\mu=(d+2)/2$, \cite[(1.4.15)]{Friedman}  and   in \cite[Lemma 1.4.3]{Friedman},  we have for any $0<\kappa^*<\kappa$,
\begin{equation}\label{18.03}
|\nabla^2_xG_{D,g, n}(s,x;t,y)| \leq C_{d,\alpha,\kappa^*}(t-s)^{-\frac{d+2}{2}}exp\Big[-\frac{\kappa^*|x-y|^2}{2(t-s)}\Big] \,.
\end{equation}
In the above inequalities, the constant $C_{d,\alpha,\kappa^*}$ depends on $\alpha, \kappa, d$ and presumably also
on $n$. However, it depends on $n$
 via the maximum values of $\sigma_n$ and their derivatives on $D$, which
are again by the property of the mollifying operator, bounded by the maximum values of $\sigma$
and their derivatives on $D$. So the constant $C_{d,\alpha,\kappa^*}$ can be chosen so that it does not
depend  on $n$.
In fact,  from  \cite[ Chapter 9]{Friedman} if we take  the  fundamental solution (Green's function) of parabolic equations with coefficients depending on $t$ by freezing the original equation only on the spatial point (not at the  time point) we can obtain all the above estimates even if $\sigma_n$ is not uniformly H\"older continuous in $t$ with  $C_{d,\alpha,\kappa^*}$ depending  only on $\alpha, \kappa, d$ and the values of $\sigma$ on $D$.

By the gradient estimates (\ref{18.01}), (\ref{18.02}) and (\ref{18.03}), we have for all $p\in[1,\infty]$,
\ce
||\nabla_x^j\cT_{s,t}^n f_n(t,\cdot)||_{L_p(D)}\leq C(t-s)^{-j/2}||f_n(t,\cdot)||_{L_p(D)},~~~~~~j=0,1,2.
\de
Moreover, by Gagliardo-Nirenberg's  inequality for bounded domain (e.g. \cite[Theorem 2.25]{Taira}),   complex interpolation inequality (e.g. \cite[Theorem 2.1]{Taira}), and  the gradient estimates,  we have
for any $p,\gamma>1$ and $\beta\in[0,2)$, there exists a constant $C$ such that
\begin{eqnarray}\label{019.01}
&&\|\cT_{s,t}^n f_n(t,\cdot)\|_{\mH^\beta_\gamma(D)}  \nonumber\\
&\leq&\|(-\Delta)^{\beta/2} \cT_{s,t}^n f_n(t,\cdot)||_{ L_\gamma(D)}
+\|\cT_{s,t}^n f_n(t,\cdot)\|_{ L_\gamma(D)} \nonumber\\
&\le &C||\nabla \cT_{s,t}^nf_n(t,\cdot)||^{\beta/2 }_{L_{\frac{\beta \gamma}{2}} (D)}\cdot|| \cT_{s,t}^n f_n(t,\cdot)||^{1-\beta /2 }_{L_\infty (D)}   +\|\cT_{s,t}^n f_n(t,\cdot)\|_{ L_\gamma(D)} \nonumber \\
&\le &C||\nabla^2\cT_{s,t}^nf_n(t,\cdot)||^{\beta/2+d/(2p)-d/(2\gamma)}_{L_p (D)}\nonumber\\
&&\quad  \cdot|| \cT_{s,t}^n f_n(t,\cdot)||^{(2-\beta)/2-d/(2p)+d/(2\gamma)}_{L_p (D)}   +\|\cT_{s,t}^n f_n(t,\cdot)\|_{ L_p(D)} \nonumber \\
&\leq& C(t-s)^{-\beta/2-d/(2p)+d/(2\gamma)}||f(t,\cdot)||_{L_p(D)}.
\end{eqnarray}
Define
\ce
\cJ_{s,t}^nh(t,x):=\int_{\partial D}h(t,\xi)G_{D, g,n} (s,x;t,\xi)dS(\xi)\,,\quad \forall \ h\in C^{1,2}([0, T]\times D)\,,
\de
where $dS(\xi)$ is the surface "area"  element on $\partial D$.

Therefore, by the gradient estimates (\ref{18.01}), (\ref{18.02}) and (\ref{18.03}), we have for all $p\in[1,\infty]$,
\begin{eqnarray}\label{019.02}
||\nabla_x^j\cJ_{s,t}^n g_n(t,\cdot)||^p_{L_p(D)}&=&\int_{D}|\nabla_x^j\cJ_{s,t}^n g_n(t,x)|^pdx \nonumber \\
&\leq& C\int_{D}\int_{\partial D}|g_n|^p(t,\xi)\times |\nabla_x^jG_{D,g,n}(s,x;t,\xi)|^pdS(\xi)dx \nonumber \\
&\leq& C(t-s)^{-jp/2}\int_{\partial D}|g_n|^p(t,\xi)dS(\xi) \nonumber \\
&\leq& C(t-s)^{-jp/2}\|g_n\|^p_{L_p(D)}\nonumber \\
&\leq& C(t-s)^{-jp/2}\|g(t,\cdot)\|^p_{L_p(D)},~~~~~~j=0,1,2\,,
\end{eqnarray}
where $C$ depends on $\alpha,\kappa,d$, and the values of $\sigma$ on $D$.

Using the same method and inequalities for $\cT^n_{s,t}f_n(t,x)$, we can get
\begin{eqnarray}\label{019.03}
\|\cJ_{s,t}^n g_n(t,\cdot)\|_{\mH^\beta_\gamma(D)}&\leq& C(t-s)^{-\beta/2-d/(2p)+d/(2\gamma)}||g(t,\cdot)||_{L_p(D)}.
\end{eqnarray}

Now let us consider the function
\ce
u_{n,1}(s,x):=\int_s^T\int_Df_n(t,y)G_{D,g, n} (s,x;t,y)dydt\,.
\de
Then $u_{n,1}$ satisfies
\begin{equation}\label{20.001}
\left\{
                    \begin{array}{lll}
                      \partial_su_{n,1}+L^{\sigma_n}u_{n,1}+f_n=0, ~~(s,x)\in [0,T]\times D,~ \\
                     u_{n,1}(T,x)=0, ~~~x\in D,~~\\
                    \end{array}
                    \right.
\end{equation}
and
\begin{equation}\label{20.01}
\|u_{n,1}\|_{\mL^q_p([t,T]\times D)}\leq C\|f_n\|_{\mL^q_p([t,T]\times D)}
\end{equation}
by Theorem 1.6.9 in \cite{Friedman}.

On the other hand, we can also consider the function
\ce
u_{n,2}(s,x):=\int_s^T\int_{\partial D}G_{D,g, n} (s,x;t,\xi)(g_n(t,\xi)-u_{n,1}(t,\xi))dS(\xi)dt\,.
\de
This function  $u_{n,2}$ satisfies
\begin{equation}\label{20.02}
\left\{
                    \begin{array}{lll}
                      \partial_su_{n,2}+L^{\sigma_n}u_{n,2}=0, ~~(s,x)\in [0,T]\times D,~ \\
                     u_{n,2}(T,x)=0, ~~~x\in D\\
                     u_{n,2}(s,x)=g_n(s,x)-u_{n,1}(s,x), ~~(s,x)\in [0,T]\times \partial D.
                    \end{array}
                    \right.
\end{equation}
Now define
\ce
u_n(s,x):=u_{n,1}(s,x)+u_{n,2}(s,x), ~~~s\in [0,T],~~x\in D.
\de
From the expressions of $u_{n,1}(s,x)$ and $u_{n,2}(s,x)$  we see   $u_n\in\mW^{2,q}_p([0,T]\times D)$.
By (\ref{20.001}) and (\ref{20.02}),  it is easy to see that $u_n$ satisfies
\begin{equation}\label{21.1}
\left\{
                    \begin{array}{lll}
                     \partial_su_n+L^{\sigma_n}u_n+f_n=0, ~~(s,x)\in [0,T]\times D,~ \\
                     u_n(T,x)=0, ~~~x\in D\, , \\
                     u_n(s,x)|_{x \in \partial D}=g_n(s,x), ~~~s\in [0,T]\,.
                    \end{array}
                    \right.
\end{equation}
In fact, we can apply (\ref{a.14}) to obtain
\begin{equation}\label{sec4-eq20.1}
\begin{split}
\|u_n\|_{\mL^q_p([0,T]\times D)}&+\|\partial_tu_n\|_{\mL^q_p([0,T]\times D)}+\|\nabla^2_xu_n\|_{\mL^q_p([0,T]\times D)}\\
& \leq C(\|f_n\|_{\mL^q_p([0,T]\times D)}+\|g_n\|_{\mL^q_p([0,T]\times D)})\\
& \leq C(\|f\|_{\mL^q_p([0,T]\times D)}+\|g\|_{\mL^q_p([0,T]\times D)})\,,
\end{split}
\end{equation}
where the constant $C$ is independent of $n$ and the above last inequality follows from the property of mollifying operator.
This proves that $u_n$ is a bounded sequence in $\mW^{2,q}_p([0,T]\times D)$.
By the weak compactness of $\mW^{2,q}_p([0,T]\times D)$ (we refer to \cite[Page 347, Theorem 11.65]{Leoni} for detail), there exists  a subsequence still denoted by $u_n$ and a function $u_0\in\mW^{2,q}_p([0,T]\times D)$ with $u_0(T)=0$ such that $u_n$   converges weakly  to $u_0$.  In fact, for any $\varphi\in C_0^\infty([0,T]\times D)$, we have
\ce
\langle L^{\sigma_n}u_n-L^{\sigma}u_0,\varphi\rangle_{[0,T]\times D}&=&\int_0^T\int_{D}(L^{\sigma_n}u_n-L^{\sigma}u_0)(t,x)\varphi(t,x)dxdt\\
&=&  \int_0^T\int_D(L^{\sigma_n}u_n-L^{\sigma}u_n)(t,x)\varphi(t,x)dxdt  \\
& &\qquad+\int_0^T\int_D(L^{\sigma}u_n-L^{\sigma}u_0)(t,x)\varphi(t,x)dxdt\\
&:=&I_1+I_2 .
\de
By the boundedness of $\varphi$ and H\"older's  inequality,   we  have
\ce
I_1
&\leq&C\left(\int_0^T(\|\sigma_n(t)-\sigma(t)\|_{L_{\infty}(D)} \cdot\|\nabla^2_xu_n(t)\|_{L_p(D)}dt\right)\\
&\leq& C\left(\int_0^T(\|\sigma_n(t)-\sigma(t)\|_{L_{\infty}(D)} )^{\frac{q}{q-1}}dt\right)^{\frac{q-1}{q}}\cdot\|\nabla^2_xu_n\|_{\mL^q_p
([0,T]\times D)}.
\de
Letting  $n\rightarrow\infty$, we have by (\ref{sec4-eq20.1})
$$
\lim_{n\rightarrow\infty}I_1=0\,.
$$
As for $I_2$, using the definition of $L^{\sigma}$ and assumption of $H^{\sigma}$, we have
\ce
 I_2 &=& \frac{1}{2}\int_0^T\int_D\sum^d_{i,j}\sum^d_{k=1}(\sigma^{ik}\sigma^{jk})(t,x)\frac{\partial^2}{\partial x_i\partial x_j}(u_n-u_0)(t,x)\varphi(t,x)dxdt   \\
 &=& \frac{1}{2}\int_0^T\int_D\sum^d_{i,j} \frac{\partial^2}{\partial x_i\partial x_j}(u_n-u_0)(t,x)\left[\sum^d_{k=1}(\sigma^{ik}\sigma^{jk})(t,x)\right]  \varphi(t,x)dxdt\,. 
\de
Since  $u_n$   converges  weakly to $u_0$ in $\mW^{2,q}_p([0,T]\times D)$, we have $\nabla_x^2 u_n$   converges  weakly to $\nabla_x^2 u_0$ in $\mL^{q}_p([0,T]\times D)$.
On the other hand, by Assumption $({\bf H}^{\sigma})$ and the boundedness of $D$,
 $\si^{ik} $ is bounded on  $[0,T]\times D$ and hence
 $\left[\sum^d_{k=1}(\sigma^{ik}\sigma^{jk})\right]  \varphi $ is in the dual of $\mL^{q}_p([0,T]\times D)$. 
Thus,  the weak convergence of $\nabla_x^2 u_n$     to $\nabla_x^2 u_0$ in $\mL^{q}_p([0,T]\times D)$ implies 
$$
\lim_{n\rightarrow\infty}I_2=0.
$$
Hence,
\ce
\lim_{n\rightarrow\infty}\int_0^T\int_D(L^{\sigma_n}u_n-L^{\sigma}u_0)(t,x)\varphi(t,x)dxdt=0.
\de
Similarly,
\ce
\lim_{n\rightarrow\infty}\int_0^T\int_D(\partial_tu_n-\partial_tu_0)(t,x)\varphi(t,x)dxdt=-\lim_{n\rightarrow\infty}\int_0^T\int_D(u_n-u_0)(t,x)\partial_t\varphi(t,x)dxdt=0,
\de
since $u_n$   converges weakly to $u_0$.

Next,  we will show the convergence of boundary function, that is: $u_0(s,x)|_{\partial D}=g(s,x)$
for all $s\in[0,T]$.   Denote $t=x_{d+1}$, $a_{i(d+1)}=a_{(d+1)i}=0$ for all $i=1,2,\cdots, d+1 $.  Then     the  parabolic equation (\ref{21.1}) with $i=1,2,\cdots, d+1,$ $x=(x_1,\cdots,x_{d+1})$
can be considered  as  an elliptic equation on $[0,T]\times D$ in place of $D$.   The corresponding boundary of $[0,T]\times D$
\ce
\partial([0,T]\times D)=(\{t=0\}\times D)\cup (\{t=T\}\times D)\cup ([0,T]\times \partial D)
\de
is Lipschitz continuous.  Thus, we can use the  compactness of trace operator  (see \cite[Page 592-594,Theorem 18.1 and
Corollary 18.6]{Leoni})   to obtain $u_0(s,x)|_{\partial D}=g(s,x)$
for all $s\in[0,T]$.

Moreover, for any $p,q,\gamma\in(1,\infty)$ and $\beta\in[0,2)$ with $\frac{d}{p}+\frac{2}{q}<2-\beta+\frac{d}{\gamma}$, by (\ref{019.01}), (\ref{019.03})  and  H\"older's inequality, we have
\begin{eqnarray}\label{021}
||u_n(t)||_{\mH^\beta_{\gamma}(D)}~&\leq&\int_t^T||\cT^n_{t,r}f_n(r,\cdot)||_{\mH^\beta_{\gamma}(D)}dr +\int_t^T||\cJ^n_{t,r}g_n(r,\cdot)||_{\mH^\beta_{\gamma}(D)}dr\nonumber\\
&\qquad&+ \int_t^T||\cJ^n_{t,r}u_{n,1}(r,\cdot)||_{\mH^\beta_{\gamma}(D)}dr\nonumber\\
&\leq& C\int_t^T(r-t)^{-\beta/2-d/(2p)+d/2\gamma}\Big(||f_n(r,\cdot)||_{L_p(D)}  \nonumber\\
&\qquad&+||g_n(r,\cdot)||_{L_p(D)} +||u_{n,1}(r,\cdot)||_{L_p(D)}\Big)dr \nonumber\\
&\leq& C\left(\int_t^T(r-t)^{-\beta q^*/2-dq^*/(2p)+dq^*/2\gamma}dr\right)^{1/q^*}\nonumber\\
&\qquad& \times\Big(||f_n||_{\mL^q_p([t,T]\times D)} +
||g_n||_{\mL^q_p([t,T]\times D)}+||u_{n,1}||_{\mL^q_p([t,T]\times D)}\Big) \nonumber\\
&\leq& C(T-t)^{(2-\beta)/2-d/(2p)-1/q+d/2\gamma}\nonumber\\
&\qquad& \times\Big(||f_n||_{\mL^q_p([t,T]\times D)} +
||g_n||_{\mL^q_p([t,T]\times D)}+||u_{n,1}||_{\mL^q_p([t,T]\times D)}\Big) \nonumber\\
&\leq& C(T-t)^{(2-\beta)/2-d/(2p)-1/q+d/2\gamma}\nonumber\\
&\qquad& \times\Big(2||f_n||_{\mL^q_p([t,T]\times D)} +
||g_n||_{\mL^q_p([t,T]\times D)}\Big) \nonumber\\
&\leq& C(T-t)^{(2-\beta)/2-d/(2p)-1/q+d/2\gamma}\nonumber\\
&\qquad& \times\Big(2||f||_{\mL^q_p([t,T]\times D)}+||g||_{\mL^q_p([t,T]\times D)}\Big),
\end{eqnarray}
where $q^*=\frac{q}{q-1}$ is the conjugate of $q$ and $C=C(d,\kappa,p,q,\alpha,\beta,\gamma) $
is a positive constant  and where in the last inequality we used the fact that
$$
\lim_{n\rightarrow\infty}\|f_n-f\|_{\mL^q_p([0,T]\times D)}=0.
$$
Letting $n\to \infty$ yields
$$
||u_0(t)||_{\mH^\beta_{\gamma}(D)} \leq C(T-t)^{(2-\beta)/2-d/(2p)-1/q+d/2\gamma}\left(||f||_{\mL^q_p([t,T]\times D)}+||g||_{\mL^q_p([t,T]\times D)}\right),
$$
where $C=C(d,\kappa,p,q,\alpha,\beta,\gamma)>0$. This shows the theorem.
\end{proof}

\subsection{General drift term and general  Dirichlet boundary condition}
Before stating  the following main result of this section,
Let us recall \cite[Theorem 13.7.2]{Krylov1}   that a function in  $\cap_{k=1}^\infty \mW^{k,q}_p([0,T]\times \partial D) $  can be continuously extended  to  $D$ as a $
\cap_{k=1}^\infty \mW^{k,q}_p([0,T]\times D)$ function.
Here is the main result in this section.

\begin{theorem}\label{th17} Let $\partial D\in C^2$.  Assume that $\sigma$ satisfies $({\bf H}^\sigma)$ and $b$ satisfies $ ({\bf H}^b)$ on the domain $D$.  Then for any $f\in \mL^q_p([0,T]\times D)$ and $g\in\mW^{2,q}_p([0,T]\times D)$, there exists a unique solution $u\in\mW^{2,q}_p([0,T]\times D)$ to  the equation
\begin{equation}\label{015}
\partial_tu+L^\sigma u+ b \cdot\nabla u =f ~~~in~~~ [0,T]\times D
\end{equation}
such that $u=g$ on $[0,T]\times\partial D$. Moreover, there exists a constant $C$ depending only on $d,p,\kappa, \diam(D)$,and $\rho_0$ such that
\begin{eqnarray}\label{a.16}
||u-g||_{\mW^{2,q}_p([0,T]\times D)}&\leq& C\left(\|b\|_{\mL^q_p([0,T]\times D)}\exp\{CT^{q\delta/3}\|b\|^q_{\mL^q_p([0,T]\times D)}\}+1\right)  \nonumber \\
&&\qquad \times \Big(\|f\|_{\mL^q_p([0,T]\times D)}+\|g\|_{\mL^q_p([0,T]\times D)}\Big)+C\|g\|_{\mW^{2,q}_p([0,T]\times D)}\,.\nonumber\\
\end{eqnarray}
Furthermore, if $p,q\in(1,\infty)$,  $f\in\mL^q_p([0,T]\times D)$, $g\in\mW^{2,q}_p([0,T]\times D)$,  then for any $\beta\in[0,2)$ and $\gamma>1$ with $\frac{d}{p}+\frac{2}{q}<2-\beta+\frac{d}{\gamma}$,
we have
\begin{eqnarray}\label{016}
||\nabla u(t)||_{\cC^{\delta/2}(D)}&\leq& C(T-t)^{\delta/3}\exp\{C_1(T-t)^{q\delta/3}\|b\|^q_{\mL^q_p([0,T]\times D)}\} \nonumber\\
&&\qquad \times \left(\|f\|^q_{\mL^q_p([0,T]\times D)}+\|g\|^q_{\mL^q_p([0,T]\times D)}\right)\,,
\end{eqnarray}
where $\delta>0, C=C(d,p,q,\beta,\gamma,\alpha)$ is a positive constant independent of $t$.

\end{theorem}
\begin{proof} By  standard continuity argument, we only need to prove the a priori estimates (\ref{a.16}) and (\ref{016}).
Letting $\delta:=\frac{1}{2}-\frac{d}{2p}-\frac{1}{q}>0$, by (\ref{sec2-eq1}), (\ref{b.14}) with suitable choices of $\beta$ and $\gamma$ such that $\beta-\frac{\delta+1}{2}>\frac{d}{\gamma}$, we have
\ce
\|\nabla u\|^q_{\mL^\infty_\infty([0,T]\times D)}&\leq& \|\nabla u(t)\|^q_{\cC^{\delta/2}(D)}
\leq C\|u(t)\|^q_{\mH^{\beta}_{\gamma}(D)}\\
&\leq& C(T-t)^{q\delta/3}\cdot\int_t^T\|(b\cdot \nabla u)(s)+f(s)\|^q_{L_p(D)}+\|g(s)\|^q_{L_p(D)}ds\\
&\leq& C(T-t)^{q\delta/3}\int_t^T\|b(s)\|^q_{L_p(D)}\cdot\|\nabla u(s)\|^q_{L_{\infty}(D)}ds\\
&&\qquad+C(T-t)^{q\delta/3}\left(\|f\|^q_{\mL^q_p([0,T]\times D)}+\|g\|^q_{\mL^q_p([0,T]\times D)}\right)\\
&\leq& C(T-t)^{q\delta/3}\int_t^T\|b(s)\|^q_{L_p(D)}\cdot\|\nabla u(s)\|^q_{\cC^{\delta/2}(D)}ds \\
&&\qquad+C(T-t)^{q\delta/3}\left(\|f\|^q_{\mL^q_p([0,T]\times D)}+\|g\|^q_{\mL^q_p([0,T]\times D)}\right)\,.
\de
Now  Gronwall's inequality implies
\begin{eqnarray}\label{sec3-16}
\|\nabla u(t)\|_{\cC^{\delta/2}(D)}&\leq& C_1(T-t)^{\delta/3}\exp\{C_1(T-t)^{q\delta/3}\|b\|^q_{\mL^q_p([0,T]\times D)}\} \nonumber\\
&&\qquad \times \left(\|f\|^q_{\mL^q_p([0,T]\times D)}+\|g\|^q_{\mL^q_p([0,T]\times D)}\right).
\end{eqnarray}
This is \eqref{016}.  Since
\ce
\partial_tu+L^\sigma u=f- b\cdot\nabla u,
\de
we  have  by (\ref{a.14}) and (\ref{sec3-16})
\ce
&&\|u-g\|_{\mW^{2,q}_p([0,T]\times D)}\\
&&\le C\left(\|f-b\cdot\nabla u \|_{\mL^q_p([0,T]\times D)}+ \|g\|_{\mW^{2,q}_p([0,T]\times D)}\right) \\
&&\le C\left(\|f\|_{\mL^q_p([0,T]\times D)}+\| b\|_{\mL^q_p([0,T]\times D)}\cdot\|\nabla u\|_{\mL^\infty_\infty([0,T]\times D)}\right)\\
&&\qquad\qquad +C\|g\|_{\mW^{2,q}_p([0,T]\times D)}\\
&&\leq C\left(\|b\|_{\mL^q_p([0,T]\times D)}\exp\{CT^{q\delta/3}\|b\|^q_{\mL^q_p([0,T]\times D)}\}+1\right) \\
&&\qquad\qquad \times\Big(\|f\|_{\mL^q_p([0,T]\times D)}+\|g\|_{\mL^q_p([0,T]\times D)}\Big)
+C\|g\|_{\mW^{2,q}_p([0,T]\times D)}.
\de
This is (\ref{a.16}).
\end{proof}

\section{\bf Krylov-type estimates}\label{s.4}
Our first result in this section is a localized version of the  Krylov type estimate for the solution of a
stochastic differential equation on a domain $D$ (before its first exit time from this domain).
\begin{theorem}\label{th19}
Suppose that $\sigma, b$ satisfy $({\bf H}^{\sigma})$ and $({\bf H}^b)$ on the domain $D$, respectively. Suppose that   $(X_t, 0\le t\le \tau_D)$ is the weak solution to
the following  stochastic differential equation:
\begin{equation}\label{sec4-equ01}
dX_t=b(t,X_t)dt+\sigma(t,X_t)dB_t,\qquad X_0=x\in D\,,
\end{equation}
where  $\tau_D$ is   the  first  exit time     of $X_t$ from $D$. Assume  $p,q\in(1,\infty)$ with $\frac{d}{p}+\frac{2}{q}<2$.
For any $\delta\in(0,1-\frac{d}{2p}-\frac{1}{q})$, there exists a positive constant $C=C(\kappa,\alpha,p,q,d, \delta  )$, such that for any $f\in\mL^q_p([0,T]\times D)$, we have
\begin{equation}\label{sec4-equ1}
\mE_x\left(\int_{r\wedge\tau_D}^{s\wedge\tau_D}\left| f(\eta,X_\eta)\right| d\eta  \mid\cF_{r\wedge\tau_D}\right)
\leq C(s-r)^\delta||f||_{\mL^q_p([0,T]\times D)}\,,\quad \forall \ 0\le r\le s\le T, x\in D\,.
\end{equation}
\end{theorem}
\begin{proof}
Let $p'=d+1$. Since $\mL^{p'}_{p'}([0,T]\times D)\cap\mL^q_p([0,T]\times D)$ is dense in $\mL^q_p([0,T]\times D)$, it suffices  to  show  (\ref{sec4-equ1}) for
\ce
f\in\mL^{p'}_{p'}([0,T]\times D)\cap\mL^q_p([0,T]\times D).
\de
Fix   $s\in[0,T]$.   By Theorem \ref{th17}, there exists a unique solution $u\in\mW^{2,p'}_{p'}([0,s]\times D)\cap\mW^{2,q}_p([0,s]\times D)$ to  the following   parabolic partial differential equation  on $[0,s]$
with terminal condition
\begin{equation}\label{sec4-eq02}
\left\{
                    \begin{array}{lll}
                     \partial_t u+L^\sigma u+b\cdot\nabla u+f=0,   ~~~in~~~[0,s]\times D,\\
                     u(t,x)=0, ~~~on~~~ [0,s]\times \partial D.\\
                    \end{array}
                    \right.
\end{equation}
Moreover, there  exists  some constant $C=C(K,\alpha,p,q,d)$  such that for all $r\in [0, s]$,
\begin{equation}\label{sec4-eq03}
||\partial_tu||_{\mL^{p'}_{p'}([r,s]\times D)}+||u||_{\mL^{p'}_{p'}([r,s]\times D)}+||\nabla^2u||_{\mL^{p'}_{p'}([r,s]\times D)}\leq C||f||_{\mL^{p'}_{p'}([r,s]\times D)}
\end{equation}
and
\ce
||\partial_tu||_{\mL^{q}_{p}([r,s]\times D)}+||u||_{\mL^{q}_{p}([r,s]\times D)}+||\nabla^2u||_{\mL^{q}_{p}([r,s]\times D)}\leq C||f||_{\mL^{q}_{p}([r,s]\times D)}.
\de
In particular, by (\ref{sec2-eq1}) and (\ref{016}) with vanishing boundary, for any $\delta\in(0,1-\frac{d}{2p}-\frac{1}{q})$, with suitable choices of $\beta$ and $\gamma$ such that $\beta-\frac{\delta+1}{2}>\frac{d}{\gamma}$  we  have
\begin{equation}\label{sec4-eq004}
\sup_{t\in[r,s]}||u(t,\cdot)||_{L_{\infty}(D)}\leq\sup_{t\in [r,s]}||u(t,\cdot)||_{\mH^\beta_p(D)}\leq C(s-r)^\delta||f||_{\mL^{q}_{p}([r,s]\times D)}\,. 
\end{equation}
Now we introduce a     space and time mollifier.
Let $\rho$ be a nonnegative smooth function in $\mR_+^{d+1}$ with support in $\{(t,x)\in\mR_+^{d+1}:(t,x)\in[0,T]\times D\}$ and $\int_{\mR_+^{d+1}}\rho(t,x)dtdx=1$. Set $\rho_n=n^{d+1}\rho(nt,nx)$ and define
\begin{equation}
u_n(t,x):=u\ast\rho_n(t,x) \label{e.def_un}
\end{equation}
and
\begin{equation}\label{sec4-a.4}
f_n(t,x):=-(\partial_tu_n(t,x)+L^\sigma u_n(t,x)+b\cdot\nabla u_n(t,x))\,.
\end{equation}
[We use the same notations as previous section without confusion.] \
Then by  the property of convolutions, we have
\ce
&&||f_n-f||_{\mL^{p'}_{p'}([0,T] \times D)}\\
&\leq&||\partial_t(u_n-u)||_{\mL^{p'}_{p'}([0,T] \times D)}+||L^\sigma u_n-L^\sigma u||_{\mL^{p'}_{p'}([0,T] \times D)}+||b\cdot\nabla(u_n-u)||_{\mL^{p'}_{p'}([0,T] \times D)}\\
&\leq&||\partial_t(u_n-u)||_{\mL^{p'}_{p'}([0,T] \times D)}+\kappa||\nabla^2(u_n-u)||_{\mL^{p'}_{p'}([0,T] \times D)}\\
&&\qquad+||b||_{\mL^{p'}_{p'}([0,T] \times D)}||\nabla(u_n-u)||_{\mL^{\infty}_{\infty}([0,T] \times D)}\\
&\leq&||\partial_tu\ast\rho_n-\partial_tu||_{\mL^{p'}_{p'}([0,T] \times D)}+\kappa||\nabla^2u\ast\rho_n-\nabla^2u||_{\mL^{p'}_{p'}([0,T] \times D)}\\
&&\qquad+||b||_{\mL^{p'}_{p'}([0,T] \times D)}||\nabla u\ast\rho_n-\nabla u||_{\mL^{\infty}_{\infty}([0,T] \times D)}\\
&\leq&||f\ast\rho_n-f||_{\mL^{p'}_{p'}([0,T] \times D)}+2\kappa||\nabla^2u\ast\rho_n-\nabla^2u||_{\mL^{p'}_{p'}([0,T] \times D)}\\
&&\qquad+2||b||_{\mL^{p'}_{p'}([0,T]  \times D)}||\nabla u\ast\rho_n-\nabla u||_{\mL^{\infty}_{\infty}([0,T] \times D)}\\
&\rightarrow&0 ~~~as ~~~n\rightarrow\infty.
\de
Now we use the following   classical Krylov's estimate \cite[Theorem 4, Page 54]{Krylov} which states
\ce
\mE\left( \int_0^{s\wedge\tau_D}f(\eta,X_{\eta})d\eta\right) \leq C||f||_{\mL^{p}_{p}([0,T] \times D)}
\de
for any  $p\ge d$,  where the constant $C=C(d, p,\kappa, \diam(D))$ depends on $d, p, \kappa, \diam(D)$ since we assume $f\in\mL^{p'}_{p'}([0,T] \times D)\cap\mL^q_p([0,T] \times D)$.

Thus
\begin{equation}\label{sec4-eq005}
\lim_{n\rightarrow\infty}\mE\left(\int_{r\wedge\tau_D}^{s\wedge\tau_D}|f_n(\eta,X_{\eta})-f(\eta,X_{\eta})|d\eta\right)\leq C\lim_{n\rightarrow\infty}||f_n-f||_{\mL^{p'}_{p'}([0,T] \times D)}=0.
\end{equation}
Now applying  It\^o's formula to $u_n(t,X_t)$, \, $0<t<T$, and by the definition (\ref{sec4-a.4}) of $f_n$ we have
\ce
u_n(t,X_t)=u_n(0,X_0)+\int_0^tf_n(s,X_s)ds+\sum^d_{k=1}\int_0^t\partial_iu_n(s,X_s)\sigma^{ik}(s,X_s)dB^k_s\,.
\de
In view of (\ref{sec4-eq004}),  we have
\ce
&&\sup_{\eta\in[r,s],x\in D}|\partial_iu_n(\eta,x)|\\
&\leq& n^{2d+1}\sup_{\theta\in(0,\eta-\frac{T}{n}),y\in \frac{D}{n}+x }|u(\eta,y)|\sup_{\eta\in[r,s],x\in D}\int_0^{\eta-\frac{T}{n}}\int_{\frac{D}{n}+x}(\partial_i\rho)(n(\eta-\theta),n(x-y))dyd\theta \\
&\leq& C_n
\de
for some constant $C_n$ which may depend on $n$. Thus,  we can apply  Doob's optional sampling theorem
to obtain
\ce
\mE\left(\int_{r\wedge\tau_D}^{s\wedge\tau_D}\partial_iu_n(\eta,X_{\eta})\sigma^{ik}(\eta,X_{\eta})dB^k_{\eta}\mid\cF_{r\wedge\tau_D}\right)=0.
\de
Hence,  we have
\ce
\left|\mE\left(\int_{r\wedge\tau_D}^{s\wedge\tau_D}f_n(\eta,X_{\eta})d\eta\mid\cF_{r\wedge\tau_D}\right)\right|&=&\left|\mE\left(u_n(s\wedge\tau_D,X_{s\wedge\tau_D})-u_n(r\wedge\tau_D,X_{r\wedge\tau_D})\mid\cF_{r\wedge\tau_D}\right)\right|\\
&\leq&2\sup_{\eta\in[r,s] ,x\in D}|u_n(\eta,x)|\leq2\sup_{\eta\in[r,s] ,x\in D}|u(\eta,x)|\\
&\leq& C(s-r)^\delta||f||_{\mL^{q}_{p}([r,s] \times D)}\,,
\de
where the  last inequality follows from   (\ref{sec4-eq004}).
Finally,   letting $n\rightarrow\infty$ and using (\ref{sec4-eq005})  we see
\ce
\left|\mE\left(\int_{r\wedge\tau_D}^{s\wedge\tau_D}f(\eta,X_{\eta})d\eta\mid\cF_{r\wedge\tau_D}\right)
\right|\leq C(s-r)^\delta||f||_{\mL^{q}_{p}([0,T] \times D)}.
\de
Since $f\in\mL^q_p([0,T] \times D)$ implies $|f|\in\mL^q_p([0,T] \times D)$ this proves the theorem.
\end{proof}

Next,  we  give a local stability result for  the solutions to equation
(\ref{sec4-equ01}).
Namely, we want to know how the solution depends
on the coefficients of the equation.
Consider two stochastic differential equations
\begin{equation}\label{two_equations}
\left\{\begin{split}
dX_t=&b(t,X_t)dt+\sigma(t,X_t)dB_t,\qquad X_0=x\,;\\
dX_t'=&b'(t,X_t')dt+\sigma'(t,X_t')dB_t,\qquad X_0'=x\, .
\end{split}\right.
\end{equation}
\begin{definition} Let $b,  b'$ satisfy  $({\bf H}^b)$  on $D$ and
$\si, \si'$ satisfy $({\bf H}^\sigma)$ on $D$.  We say that  the two equations have   tied weak solutions
if one can find a common probability space $(\Om, \cF, \PP)$  and a common Brownian
motion $(B_t, t\ge 0)$, and two stochastic processes $(X_t, 0\le t\le \tau_D)$ and
$(X_t', 0\le t\le  \tau_D')$  such that
\begin{equation}\label{two_weak}
\left\{\begin{split}
 X_{t\wedge \tau_D} =&x+\int_0^{t\wedge \tau_D}
 b(s,X_s)ds+\int_0^{t\wedge \tau_D} \sigma(s,X_s)dB_s \,;\\
 X_{t\wedge \tau_D'}' =&x'+\int_0^{t\wedge \tau_D'}
 b'(s,X_s')ds+\int_0^{t\wedge \tau_D } \sigma '(s,X_s')dB_s \, .
\end{split}\right.
\end{equation}
\end{definition}

First,  we will give a local stability result for  the solutions to
the equation  without drift coefficients.  Assume that $X^\sigma$ and $X^{\sigma'}$ are
two tied weak
solutions to \eqref{sec4-equ01} with drift term $b'=b=0$ and with two different coefficients $\sigma$ and $ \sigma'$.  We want to bound the difference of the solutions
$X^\sigma-X^{\sigma'}$  by the  Sobolev norm of $\sigma-\sigma'$. We cannot no longer use  Burkh\"older's  inequality  since   when $0<p_0<1$
we can only bound $\mE\left(\sup_{s\in[0,T ]}|X_s^{\sigma}(x)-X_s^{\sigma'}(x)|^{p_0}\right) $  by
$\mE\left(\int_0^t||\sigma(r,X^\sigma_r)-\sigma'(r,X^{\sigma'}_r)||^2dr\right)^{p_0/2} $
which cannot be bounded by $||\sigma-\sigma'||^{p_0}_{\mL^q_p([0,T] \times D)}  $.
Instead,   we   shall  use Lemma \ref{le6.2} in Appendix for  positive constant $p_0$ in place of  Burkh\"older's  inequality.

\begin{lemma}\label{le20} Assume that $\sigma, \sigma'$ satisfy $({\bf H}^\sigma)$ and let  $p,q\in (1,\infty)$ with $\frac{d}{p}+\frac{2}{q}<1$.
Let $X^{\sigma}$, $X^{\sigma'}$ be two tied  weak solutions to  equation (\ref{sec4-equ01}) with  $b'=b=0$, with the same initial condition,
 and with diffusion coefficients $\sigma$, $\sigma'$, respectively.
 Let
$$
 \tau_D=\inf \left\{ t\ge 0\,; \quad  X_t\not \in D\quad \hbox{or}\quad X_t'\not\in D\right\}\,.
$$
Then,    for any $p_0>0$,
there exists a constant $C=C(p_0, p,q,d,\kappa,\alpha)>0$  such that
\ce
\sup_{x\in D}\mE\left(\sup_{s\in[0,T\wedge\tau_D]}|X_s^{\sigma}(x)-X_s^{\sigma'}(x)|^{p_0}\right)\leq C||\sigma-\sigma'||^{p_0}_{\mL^q_p([0,T] \times D)}.
\de
\end{lemma}
\begin{proof}Set $Z_t=X_t^\sigma-X_t^{\sigma'}$. Then
\ce
Z_t=\int_0^t\left(\sigma(r,X^\sigma_r)-\sigma'(r,X^{\sigma'}_r)\right)dB_r.
\de
By It\^o's formula, we have
\ce
|Z_t|^2&=&\int_0^t||\sigma(r,X^\sigma_r)-\sigma'(r,X^{\sigma'}_r)||^2dr+2\int_0^t\left(\sigma(r,X^\sigma_r)-\sigma'(r,X^{\sigma'}_r)\right)^tZ_rdB_r\\
&=&\int_0^t\xi(r)dr+\int_0^t\eta(r)dB_r+\int_0^t|Z_r|^2\beta(r)dr+\int_0^t|Z_r|^2\alpha(r)dB_r,
\de
where
\begin{eqnarray}\label{0}
\xi(r)&:=&  ||\sigma(r,X^\sigma_r)-\sigma'(r,X^{\sigma'}_r)||^2- ||\sigma(r,X^\sigma_r)-\sigma(r,X^{\sigma'}_r)||^2,  \nonumber \\
\eta(r)&:=&2(\sigma(r,X^{\sigma'}_r)-\sigma'(r,X^{\sigma'}_r))^tZ_r,\nonumber \\
\beta(r)&:=& ||\sigma(r,X^\sigma_r)-\sigma(r,X^{\sigma'}_r)||^2/|Z_r|^2,\nonumber \\
\alpha(r)&:=&2(\sigma(r,X^\sigma_r)-\sigma(r,X^{\sigma'}_r))^tZ_r/|Z_r|^2.
\end{eqnarray}
Here, we have used the convention $\frac{0}{0}:=0$, that is, if   $\beta(r)=\alpha(r)=0$,
then $|Z_r|=0$.\\

By  the  inequality (\ref{sec4-equ1}) and by the fact that the local maximal operator $\cM_D$ is bounded on $\mL^p_q([0,T] \times D)$ for any $p\,, q>1$, we have that for any $0\leq s<t\leq T$,
\begin{eqnarray*}
&&\mE\left(\int_{s\wedge\tau_D}^{t\wedge\tau_D}(|\beta(r)|+|\alpha(r)|^2)dr\mid\cF_{s\wedge\tau_D}\right)  \\
&&\leq C\mE\left(\int_{s\wedge\tau_D}^{t\wedge\tau_D}\left(\cM_D|\nabla\sigma|^2(X_r^\sigma)+\cM_D|\nabla\sigma|^2(X_r^{\sigma'})\right)dr\mid\cF_{s\wedge\tau_D}\right)\\
&&\leq C(t-s)^\delta||\cM_D|\nabla\sigma|^2||_{\mL^{q/2}_{p/2}([0,T] \times D)}\leq C(t-s)^\delta|||\nabla\sigma|^2||_{\mL^{q/2}_{p/2}([0,T] \times D)}\\
&&\leq C(t-s)^\delta||\nabla\sigma||^2_{\mL^{q}_{p}([0,T] \times D)}\,,
\end{eqnarray*}
where $\cM_D$ is the local maximal operator (e.g.(\ref{sec6-eq01})). Now we want to bound $\xi^+$.
For any $\gamma\in(1\vee p_0,1/(2/q+d/p)\vee p_0)$, we have by   Theorem \ref{th19}
\begin{eqnarray*}
&&\mE\left(\int_{s\wedge\tau_D}^{t\wedge\tau_D}||\sigma(r,X^{\sigma'}_r)-\sigma'(r,X^{\sigma'}_r)||^2dr\right)^{\gamma}\nonumber \\
& &\qquad\le \mE\left(\int_{s\wedge\tau_D}^{t\wedge\tau_D}||\sigma(r,X^{\sigma'}_r)-\sigma'(r,X^{\sigma'}_r)||^{2\gamma}dr\right)\nonumber \\
& &\qquad\le C(t-s)^\delta||||\sigma-\sigma'||^{2\gamma}||_{\mL^{q/(2\gamma)}_{p/(2\gamma)}([0,T] \times D)}
\nonumber \\
& &\qquad= C(t-s)^\delta||\sigma-\sigma'||^{2\gamma}_{\mL^q_p([0,T] \times D)}\,.
\end{eqnarray*}
This implies
\begin{equation}
\EE \left[\int_0^{T\wedge \tau} \xi^+(r) dr\right]^\gamma  \le ||\sigma-\sigma'||^{2\gamma}_{\mL^q_p([0,T] \times D)}\,. \label{sect4-equ05}
\end{equation}
Using  inequality (\ref{eq6.2}) in Appendix (Lemma \ref{le6.2})  with $p_0>0, \gamma_2=\gamma/p_0$ and $\gamma_3=\frac{2\gamma_2}{\gamma_2+1}$  we obtain
\begin{eqnarray}\label{sect4-equ06}
&&\mE\left(\sup_{s\in[0,T\wedge\tau_D]}|Z_s|^{2p_0}\right)\nonumber \\
&\leq& C\left(\left |\left |\left (\int_0^{T\wedge\tau_D}\xi^{+}(r)dr\right )^{p_0}\right |\right |_{L_{\gamma_2}(\Omega)}+\left |\left |\left (\int_0^{T\wedge\tau_D}|\eta(r)|^2dr\right )^{p_0/2}\right |\right |_{L_{\gamma_3}(\Omega)}\right)\nonumber \\
&\leq& C\left|\left|\left(\int_0^{T\wedge\tau_D}|Z_r|^2||\sigma(r,X^{\sigma'}_r)-\sigma'(r,X^{\sigma'}_r)||^2dr\right)^{p_0/2 }\right|\right|_{L_{\gamma_3}(\Omega)} \nonumber \\
&\;&+ C ||\sigma-\sigma'||^{2p_0}_{\mL^q_p([0,T] \times D)}\,.
\end{eqnarray}
Now   using  H\"older's inequality,  we have
\begin{eqnarray}
&& \mE\left(\sup_{s\in[0,T\wedge\tau_D]}|Z_s|^{2p_0}\right)\nonumber \\
&\leq&C\left(\mE\sup_{s\in[0,T\wedge\tau_D]}|Z_s|^{2p_0}\right)^{1/2}\times\left|\left|\left(\int_0^{T\wedge\tau_D}||\sigma(r,X^{\sigma'}_r)-\sigma'(r,X^{\sigma'}_r)||^2dr\right)^{p_0/2}\right|\right|_{L_{\gamma_2}(\Omega)} \nonumber \\
&\;&+ C ||\sigma-\sigma'||^{2p_0}_{\mL^q_p([0,T] \times D)} \nonumber \\
&\leq&C\left(\mE\sup_{s\in[0,T\wedge\tau_D]}|Z_s|^{2p_0}\right)^{1/2}\times\left[ ||\sigma-\sigma'||^{2p_0}_{\mL^q_p([0,T] \times D)}\right]^{1/2} + C ||\sigma-\sigma'||^{2p_0}_{\mL^q_p([0,T] \times D)} \nonumber \\
&\leq&\frac{1}{2}\left(\mE\sup_{s\in[0,T\wedge\tau_D]}|Z_s|^{2p_0}\right) + C ||\sigma-\sigma'||^{2p_0}_{\mL^q_p([0,T] \times D)} \,. \nonumber\\
\end{eqnarray}
Thus, we obtain for any $p_0>0$,
\ce
\mE\left(\sup_{s\in[0,T\wedge\tau_D]}|Z_s|^{2p_0}\right)\leq 2C||\sigma-\sigma'||^{2p_0}_{\mL^q_p([0,T]\times D)}.
\de
This proves the lemma.
\end{proof}

The next theorem is about the stability on the drift coefficients of \eqref{sec4-equ01} on domain $D$.
\begin{theorem}\label{th21} Assume that $b, b'$ satisfy $({\bf H}^b)$  on the domain $D$  with   $p, q$ satisfying condition $\frac{d}{p}+\frac{2}{q}<1$, and assume that  $\sigma$ satisfies $({\bf H}^\sigma)$ on $D$. Let $X_t^{b,\sigma}$ and $X_t^{b',\sigma}$ be two tied weak  solutions to (\ref{sec4-equ01}) associated with coefficient pairs $(b,\sigma)$,  $(b',\sigma)$, respectively. Then for any $p_0>0$,
\begin{equation}\label{sect4-equ06}
\sup_{x\in D}\mE\left(\sup_{s\in[0,T\wedge\tau_D]}|X_s^{b', \sigma}(x)-X_s^{b, \sigma}(x)|^{p_0}\right)\leq C||b-b'||^{p_0}_{\mL^q_p([0,T]\times D)}.
\end{equation}
\end{theorem}
\begin{proof} Since $b(t,x), b'(t,x)\in\mL^q_p([0,T]\times D)$,   Theorem \ref{th17} implies that for any $f,f'\in\mL^q_p([0,T]\times D)$ and $g, g'\in\mW^{2,q}_p([0,T]\times \partial D)$,  there exist unique solutions $u^b, u^{b'}\in\mW^{2,q}_p([0,T]\times D)$ to the equations
\ce
\partial_tu^{b}+L^{\sigma}u^{b}+b\cdot\nabla u^{b}+f=0, ~~~u^{b}(T,x)=0,
\de
and
\ce
\partial_tu^{b'}+L^{\sigma}u^{b'}+b'\cdot\nabla u^{b'}+f'=0, ~~~u^{b'}(T,x)=0,
\de
such that $u^{b}-g\in{\mathring\mW}^{2,q}_p([0,T]\times D)$, $u^{b'}-g'\in{\mathring\mW}^{2,q}_p([0,T]\times D)$, respectively.

In the above equations we take  $f=b^{\ell}, f'=b'^{\ell}$ for $\ell=1,2,\ldots,d$. The corresponding solutions are denoted by  ${\bf u}^b(t,x):=(u^{b,1}(t,x),\ldots,u^{b,d}(t,x))$, ${\bf u}^{b'}(t,x):=(u^{b',1}(t,x),\ldots,u^{b',d}(t,x))$.
Set
\ce
\Phi^b(t,x):=x+{\bf u}^b(t,x),~~~~\Phi^{b'}(t,x)
:=x+{\bf u}^{b'}(t,x), ~~~~~in~~~~~[0,T]\times D.
\de
Let $\delta:=\frac{1}{2}-\frac{d}{2p}-\frac{1}{q}>0$. By 
(\ref{016}), there is a $C=C(\kappa,\alpha,p,q,d)>0$ such that for all $t\in[s_0,t_0]\subseteq[0,T]$,
\ce
 ||\nabla{\bf u}^b(t,\cdot)||_{\cC^\delta(D)}
 &\leq&  C(t_0-s_0)^{\delta/3}\exp\{C_1(T-t)^{q\delta/3}\|b\|^q_{\mL^q_p([0,T]\times D)}\} \nonumber\\
&&\qquad   \times \left(\|f\|^q_{\mL^q_p([0,T]\times D)}+\|g\|^q_{\mL^q_p([0,T]\times D)}\right).
\de
For given positive constant $M$, let us choose $\varepsilon=\varepsilon(\delta,p,q,C,M)>0$ small enough so that for all $t_0-s_0\leq\varepsilon$ and $||b||_{\mL^q_p([s_0,t_0]\times D)}\leq M$,
\ce
\sup_{s_0\leq t\leq t_0}||\nabla{\bf u}^b(t,\cdot)||_{\cC^\delta(D)}\leq\frac{1}{2}.
\de
In particular, we have
\ce
|{\bf u}^b(t,x)-{\bf u}^b(t,y)|\leq\frac{|x-y|}{2},~~~~\quad s_0\leq t\leq t_0.
\de
Thus by definition of $\Phi^b(t,x)=x+{\bf u}^b(t,x)$
we have
\begin{equation}\label{e.4.1}
\frac{1}{2}|x-y|\leq|\Phi^b(t,x)-\Phi^b(t,y)|\leq\frac{3}{2}|x-y|
\end{equation}
for all $t_0-s_0\leq\varepsilon$, any $x,y\in D$ and $||b||_{\mL^q_p([s_0,t_0]\times D)}\leq M$.
In the same way, for all $t_0-s_0\leq\varepsilon$, any $x,y\in D$ and $||b'||_{\mL^q_p([s_0,t_0]\times D)}\leq M$, we also have
\ce
\frac{1}{2}|x-y|\leq|\Phi^{b'}(t,x)-\Phi^{b'}(t,y)|\leq\frac{3}{2}|x-y|.
\de
Next,  we shall  verify the following two estimates:
\begin{equation}\label{sec4-eq01}
\nabla(\nabla\Phi^b\cdot\sigma)(t,x)\in\mL^q_p([s_0,t_0]\times D)
\end{equation}
and
\begin{equation}\label{sec4-eq02}
||\Phi^{b'}-\Phi^b||_{\mL^\infty_\infty([s_0,t_0]\times D)}+||\nabla\Phi^b-\nabla\Phi^{b'}||_{\mL^q_p([s_0,t_0]\times D)}\leq C||b'-b||_{\mL^q_p([s_0,t_0]\times D)}.
\end{equation}
Indeed,
\ce
&&||\nabla(\nabla\Phi^b\cdot\sigma)||_{\mL^q_p([s_0,t_0]\times D)}\\
&&\qquad \quad \leq \kappa ||\nabla^2\Phi^b||_{\mL^q_p([s_0,t_0]\times D)}+||\nabla\Phi^b||_{\mL^\infty_\infty([s_0,t_0]\times D)}\cdot||\nabla\sigma||_{\mL^q_p([s_0,t_0]\times D)}\\
&&\qquad \quad\leq \kappa||\nabla^2{\bf u}^b||_{\mL^q_p([s_0,t_0]\times D)}+||\nabla\Phi^b||_{\mL^\infty_\infty([s_0,t_0]\times D)}\cdot||\nabla\sigma||_{\mL^q_p([s_0,t_0]\times D)}\\
&&\qquad \quad< \infty.
\de
Thus \eqref{sec4-eq01} holds true.
On the other hand, if we  denote  $w:=w^{b,b'}:={\bf u}^b-{\bf u}^{b'}$, then $w$ satisfies the following equation
\begin{equation}\label{sec4-2}
\left\{
                  \begin{array}{lll}
                  \partial_tw+L^\sigma w+b\cdot\nabla w=f-f'+(b'-b)\cdot\nabla{\bf u}^{b'} ~~~in ~~                               [0,T]\times D, \\
                  w(T,x)=0, x\in D,\\
                  w(t,x)=0,~~~t\in(0,T),~~~x\in \partial D.
                  \end{array}
                    \right.
\end{equation}
As above, using  the   definition of $\cC^\delta$
(\eqref{e.def_holder}),   and (\ref{sec2-eq1}),   (\ref{016}), choosing  suitable $\delta,\beta$  such that  $\beta-(\delta+1)>\frac{d}{p}$,  we have
\begin{eqnarray}\label{4.15}
||\nabla{\bf u}^{b'}||_{\mL^\infty_\infty([s_0,t_0]\times D)}&=&\sup_{t\in[s_0,t_0]}||\nabla{\bf u}^{b'}(t)||_{L_{\infty}(D)}  \leq\sup_{t\in[s_0,t_0]}||{\bf u}^{b'}(t)||_{\cC^{\delta+1}(D)}  \nonumber\\
&\leq&\sup_{t\in[s_0,t_0]}||{\bf u}^{b'}(t)||_{\mH^\beta_p(D)} \nonumber\\
&\leq&C(t_0-s_0)^{\delta}\Big(||f'||_{\mL^q_p([s_0,t_0]\times D)}+||g'||_{\mL^q_p([s_0,t_0]\times D)}\Big) \nonumber\\
&<&\infty.
\end{eqnarray}
Therefore, by inequalities (\ref{a.16}) in Theorem  \ref{th17} and (\ref{4.15}), we get
\begin{eqnarray}\label{4.015}
||\nabla\Phi^b-\nabla\Phi^{b'}||_{\mL^q_p([s_0,t_0]\times D)}
&=&||\nabla w||_{\mL^q_p([s_0,t_0]\times D)}\leq||w||_{\mW^{2,q}_p([s_0,t_0]\times D)} \nonumber \\
&\leq &   C||f-f'||_{\mL^q_p([s_0,t_0]\times D) } \nonumber\\
&+&  C||\nabla{\bf u}^{b'}||_{\mL^\infty_\infty([s_0,t_0]\times D)}
\times||b-b'||_{\mL^q_p([s_0,t_0]\times D)} \nonumber\\
&\leq & C||b-b'||_{\mL^q_p([s_0,t_0]\times D)}\,.
\end{eqnarray}
On the other hand, from  the  inequalities (\ref{sec2-eq1}), (\ref{016}) in Theorem  \ref{th17} and (\ref{4.15})    it follows
\ce
||\Phi^{b'}-\Phi^b||_{\mL^\infty_\infty([s_0,t_0]\times D)}&=&||w||_{\mL^\infty_\infty([s_0,t_0]\times D)}=\sup_{t\in[s_0,t_0]}||w(t)||_{L_{\infty}(D)}\\
&\leq&\sup_{t\in[s_0,t_0]}||w(t)||_{  \cC^\delta(D)}
\leq C\sup_{t\in[s_0,t_0]}||w(t)||_{\mH^\beta_p(D)}\\
&\leq &C||f-f'+(b'-b)\cdot\nabla{\bf u}^{b'}||_{\mL^q_p([s_0,t_0]\times D)}\\
&\leq &C||b-b'||_{\mL^q_p([s_0,t_0]\times D)}\times(1+||\nabla{\bf u}^{b'}||_{\mL^\infty_\infty([s_0,t_0]\times D)})\\
&\leq&C||b-b'||_{\mL^q_p([s_0,t_0]\times D)}.
\de
Combining this estimate with (\ref{4.015}) yields  (\ref{sec4-eq02}).

By the generalized It\^o  formula (e.g. \cite[Lemma 4.3]{Zhang2}), $X^{b,\sigma}_t(x)$ solves SDE (\ref{sec4-equ01}) on $[0,T\wedge\tau_D]$ with initial value $x\in D$ if and only if $Y^b_t(y):=\Phi^b(t,X^{b,\sigma}_t(x))$ solves the following SDE on $[0,T\wedge\tau_D]$ with initial value in $D$, where $D$ is also the image  of the function $\Phi^b(t,x)$ by Lemma \ref{le6.02},
\begin{equation}\label{sec4-eq03}
dY^b_t=(\nabla\Phi^b\cdot\sigma)\circ(t,\Phi^{b,-1}(t,Y^b_t))dB_t:=\Theta^b(t,Y_t)dB_t\,,
\end{equation}
where
\begin{equation}\label{e.1}
\Theta^b(t,y):= (\nabla\Phi^b\cdot\sigma)\circ(t,  \Phi^{b,-1}(t,y) )
\end{equation}
and   $\Phi^{b,-1}(t,y)$ is the inverse of $\Phi^{b}(t,\cdot ):D\to D$ with respect to spatial  variable.
Let $\Theta^{b'}$ be defined as above (\ref{e.1}) with $b$ replaced by $b'$, and $\Phi^{b',-1}(t,y)$ is the inverse of $\Phi^{b'}(t,y)$ with respect to spatial  variable.   To study the stability of our original equation we now reduce it to the stability of the equation of \eqref{sec4-eq03} and to study the later equation, we need to know how $\Theta$ depends on $b$.  First, we have
\ce
\Theta^b(t,y)-\Theta^{b'}(t,y)&=&(\nabla\Phi^b\cdot\sigma)\circ(t,\Phi^{b,-1}(t,y))-(\nabla\Phi^b\cdot\sigma)\circ(t,\Phi^{b',-1}(t,y))\\
&\;&+((\nabla\Phi^b-\nabla\Phi^{b'})\cdot\sigma)\circ(t,\Phi^{b',-1}(t,y)):=I_1(t,y)+I_2(t,y).
\de
For the above first term $I_1(t,y)$, by   \cite[the inequality (2.2)]{Zhang} we have
\ce
|I_1(t,y)|&\leq &C|\Phi^{b,-1}(t,y)-\Phi^{b',-1}(t,y)|\\
&&\qquad \times \left(\cM_D|\nabla(\nabla\Phi^b\cdot\sigma)|(t,\Phi^{b,-1}(t,y))+\cM_D|\nabla(\nabla\Phi^b\cdot\sigma)|(t,\Phi^{b',-1}(t,y))\right)\,,
\de
where $\cM_D$ is the Hardy-Littlewood maximal operator which is bounded operator in any $\mL_p^q([0,T]\times D)$ space. Noticing that
\ce
\sup_{y\in D}|\Phi^{b,-1}(t,y)-\Phi^{b',-1}(t,y)|&=&\sup_{x\in D}|x-\Phi^{b',-1}\circ(t,\Phi^b(t,x))|\\
&=&\sup_{x\in D}|\Phi^{b',-1}\circ(t,\Phi^{b'}(t,x))-\Phi^{b',-1}\circ(t,\Phi^b(t,x))|\\
&\leq&  ||\nabla\Phi^{b',-1} (t,\cdot)||_{L_{\infty}(D)}\times||\Phi^{b'}(t,\cdot)-\Phi^b(t,\cdot)||_{L_{\infty}(D)}
\de
and by the change of variables,  \cite[(2.3)]{Zhang}, (\ref{sec4-eq01}) and (\ref{sec4-eq02})  we have
\begin{eqnarray}
||I_1||_{\mL^q_p([s_0,t_0]\times D)}&\leq&C||\cM_D|\nabla(\nabla\Phi^b\cdot\sigma)|(\cdot,\Phi^{b,-1})+\cM_D|\nabla(\nabla\Phi^b\cdot\sigma)|(\cdot,\Phi^{b',-1})||_{\mL^q_p([s_0,t_0]\times D)}\nonumber \\
&&\times||\Phi^{b,-1}-\Phi^{b',-1}||_{\mL^\infty_\infty([s_0,t_0]\times D)}\nonumber \\
&\leq&C||\nabla\Phi^{b',-1}||_{\mL^\infty_\infty([s_0,t_0]\times D)}\times||\Phi^{b'}-\Phi^b||_{\mL^\infty_\infty([s_0,t_0]\times D)}\nonumber \\
&&\times||\cM_D|\nabla(\nabla\Phi^b\cdot\sigma)|||_{\mL^q_p([s_0,t_0]\times D)}\nonumber \\
&\leq&C||\nabla(\nabla\Phi^b\cdot\sigma)||_{\mL^q_p([s_0,t_0]\times D)}\times||\Phi^{b'}-\Phi^b||_{\mL^\infty_\infty([s_0,t_0]\times D)}\nonumber \\
&\leq&C||b'-b||_{\mL^q_p([s_0,t_0]\times D)}.\label{e.est_i1}
\end{eqnarray}
For the second
term $I_2(t,y)$, by the change of variables, boundedness  of $\sigma$ and (\ref{sec4-eq02}), we have
\begin{eqnarray}
||I_2||_{\mL^q_p([s_0,t_0]\times D)}
&=&||(\nabla\Phi^b-\nabla\Phi^{b'})\cdot\sigma||_{\mL^q_p([s_0,t_0]\times D)}\nonumber \\
&\leq& \kappa||\nabla\Phi^b-\nabla\Phi^{b'}||_{\mL^q_p([s_0,t_0]\times D)} \nonumber \\
&\leq& C||b'-b||_{\mL^q_p([s_0,t_0]\times D)}.\label{e.est_i2}
\end{eqnarray}
Combining \eqref{e.est_i1} and \eqref{e.est_i2} yields
\begin{equation}\label{sec4-eq04}
||\Theta^b-\Theta^{b'}||_{\mL^q_p([s_0,t_0]\times D)}\leq C||b-b'||_{\mL^q_p([s_0,t_0]\times D)}.
\end{equation}
Dividing  $[0,T]$ into subintervals,  applying the above estimates on each interval and piecing  them together,
we then obtain
\begin{equation}\label{sec4-equ04}
||\Theta^b-\Theta^{b'}||_{\mL^q_p([0,T]\times D)}\leq C||b-b'||_{\mL^q_p([0,T]\times D)}.
\end{equation}
Finally, combining the above estimate \eqref{sec4-equ04}  with  Lemma \ref{le20} gives
\ce
\sup_{x\in D}\mE\left(\sup_{t\in[0,T\wedge\tau_D]}|X_t^{b,\sigma}(x)-X_t^{b'\sigma}(x)|^{p_0}\right)
&\leq&\sup_{y\in D}\mE\left(\sup_{t\in[0,T\wedge\tau_D]}|Y_t^b(y)-Y_t^{b'}(y)|^{p_0}\right)\\
&\leq&C||\Theta^b-\Theta^{b'}||^{p_0}_{\mL^q_p([0,T]\times D)}\\
&\leq& C||b-b'||^{p_0}_{\mL^q_p([0,T]\times D)}.
\de
This completes the proof of the theorem.
\end{proof}

Finally, we concluded this subsection with the following theorem:
\begin{theorem}\label{th22} Assume that on the domain $D$, the coefficients   $b, b'$ satisfy $({\bf H}^b)$ and that coefficients  $\sigma, \sigma'$ satisfy $({\bf H}^\sigma)$. Let $X_t^{b,\sigma}$ and $X_t^{b',\sigma'}$ be the tied weak solutions to (\ref{sec4-equ01}) associated with coefficients pairs $(b, \sigma)$ and $(b', \sigma')$, respectively. Then for all $p_0>0$, there is a constant $C=C_{p_0, p, q, \kappa, \alpha}$ depending on $p_0, p, q,  \kappa, \alpha $, 
and the constants appeared in
 assumptions $({\bf H}^b)$ and $({\bf H}^\sigma)$, such that
\begin{equation}\label{sect4-equ07}
\sup_{x\in D}\mE\left(\sup_{t\in[0,T\wedge\tau_D]}|X_t^{b,\sigma}(x)-X_t^{b',\sigma'}(x)|^{p_0}\right)\leq C\left(||b-b'||^{p_0}_{\mL^q_p([0,T]\times D)}+||\sigma-\sigma'||^{p_0}_{\mL^q_p([0,T]\times D)}\right).
\end{equation}
\end{theorem}
\begin{proof}
Since
\ce
&&\mE\left(\sup_{t\in[0,T\wedge\tau_D]}|X_t^{b,\sigma}(x)-X_t^{b',\sigma'}(x)|^{p_0} \right)\\
&\leq&2^{p_0-1}\mE\left(\sup_{t\in[0,T\wedge\tau_D]}|X_t^{b,\sigma}(x)-X_t^{b',\sigma}(x)|^{p_0}\right)\\
& &\qquad +2^{p_0-1}\mE\left(\sup_{t\in[0,T\wedge\tau_D]}|X_t^{b',\sigma}(x)-X_t^{b',\sigma'}(x)| ^{p_0}\right)\\
&=:& J_1(x)+J_2(x).
\de
By Theorem \ref{th21}, we have already that
\ce
\sup_{x\in D}J_1(x) \leq C||b-b'||^{p_0}_{\mL^q_p([0,T]\times D)}.
\de
Next,  we devote our effort to deal with $J_2(x)$.
We shall use the  Zvonkin transformation as in the proof of Theorem  \ref{th21}.
 For any $f,f'\in\mL^q_p([0,T]\times D)$ and $g, g'\in\mW^{2,q}_p([0,T]\times \partial D)$,  Theorem \ref{th17} implies that there exist unique solutions $u^\sigma, u^{\sigma'}\in\mW^{2,q}_p([0,T]\times D)$ to
 the following two equations:
\ce
\partial_tu^{\sigma}+L^{\sigma}u^{\sigma}+b\cdot\nabla u^{\sigma}+f=0, ~~~u^{\sigma}(T,x)=0,
\de
and
\ce
\partial_tu^{\sigma'}+L^{\sigma'}u^{\sigma'}+b\cdot\nabla u^{\sigma'}+f'=0, ~~~u^{\sigma'}(T,x)=0,
\de
such that $u^{\sigma}-g\in{\mathring\mW}^{2,q}_p([0,T]\times D)$, $u^{\sigma'}-g'\in{\mathring\mW}^{2,q}_p([0,T]\times D)$, respectively.

Taking  $f=f'=b^{\ell}$ for $\ell=1,2,\ldots,d$  and letting
\begin{equation*}
{\bf u}^\sigma(t,x):=(u^{\sigma,1}(t,x),\ldots,u^{\sigma,d}(t,x))\,,\quad {\bf u}^{\sigma'}(t,x):=(u^{\sigma',1}(t,x),\ldots,u^{\sigma',d}(t,x))
\end{equation*}
  and
\ce
\Phi^\sigma(t,x):=x+{\bf u}^\sigma(t,x),~~~~\Phi^{\sigma'}(t,x):=x+{\bf u}^{\sigma'}(t,x), ~~~~~in~~~~~[0,T]\times D.
\de
Applying the same procedure  as that in the proof of  Theorem \ref{th21}    we have
for all $t_0-s_0\leq\varepsilon$, any $x,y\in D$ and $||\sigma||_{\mL^q_p([s_0,t_0]\times D)},   ||\sigma'||_{\mL^q_p([s_0,t_0]\times D)}\leq M$,
\ce
\frac{1}{2}|x-y|\leq|\Phi^\sigma(t,x)-\Phi^\sigma(t,y)|\leq\frac{3}{2}|x-y|
\de
and
\ce
\frac{1}{2}|x-y|\leq|\Phi^{\sigma'}(t,x)-\Phi^{\sigma'}(t,y)|\leq\frac{3}{2}|x-y|.
\de
As in the proof of Theorem \ref{th21} we need to  verify
\begin{equation}\label{sec4-eq05}
\nabla(\nabla\Phi^\sigma\cdot\sigma)(t,x)\in\mL^q_p([0,T]\times D)
\end{equation}
and
\begin{equation}\label{sec4-eq06}
||\Phi^{\sigma'}-\Phi^\sigma||_{\mL^\infty_\infty([s_0,t_0]\times D)}+||\nabla\Phi^\sigma-\nabla\Phi^{\sigma'}||_{\mL^q_p([s_0,t_0]\times D)}\leq C||\sigma-\sigma'||_{\mL^q_p([s_0,t_0]\times D)}.
\end{equation}
Indeed,
\ce
||\nabla(\nabla\Phi^\sigma\cdot\sigma)||_{\mL^q_p([s_0,t_0]\times D)}&\leq&\kappa||\nabla^2\Phi^\sigma||_{\mL^q_p([s_0,t_0]\times D)}\\
&\qquad& +||\nabla\Phi^\sigma||_{\mL^\infty_\infty([s_0,t_0]\times D)}\cdot||\nabla\sigma||_{\mL^q_p([s_0,t_0]\times D)}\\
&\leq&\kappa||\nabla^2{\bf u}^\sigma||_{\mL^q_p([s_0,t_0]\times D)}\\
&\qquad& + ||\nabla\Phi^\sigma||_{\mL^\infty_\infty([s_0,t_0]\times D)}\cdot||\nabla\sigma||_{\mL^q_p([s_0,t_0]\times D)}\\
&<&\infty.
\de
On the other hand, if we  denote  $v:=v^{\sigma,\sigma'}:={\bf u}^{\sigma}-{\bf u}^{\sigma'}$, then $v$ satisfies following equation
\begin{equation}\label{e.4.1}
\left\{
                  \begin{array}{lll}
                  \partial_tv+L^\sigma v +b\cdot\nabla v=(L^{\sigma'}-L^\sigma)\cdot{\bf u}^{\sigma'}  ~~~in ~~                               [0,T]\times D, x\in D,\\
                  v(T,x)=0,\\
                  v(t,x)=0,~~~t\in(0,T),~~~x\in \partial D.
                  \end{array}
                    \right.
\end{equation}
As above,  by the   definition of $\cC^\delta$
(\eqref{e.def_holder}),    and (\ref{sec2-eq1}),  (\ref{016}), choosing  suitable $\delta,\beta$  such that  $\beta-(\delta+2)>\frac{d}{\gamma}$, we have
\begin{eqnarray}\label{e.4.2}
||\nabla^2{\bf u}^{\sigma'}||_{\mL^\infty_\infty([s_0,t_0]\times D)}&=&\sup_{t\in[s_0,t_0]}||\nabla^2{\bf u}^{\sigma'}(t)||_{L_{\infty}(D)}\leq\sup_{t\in[s_0,t_0]}||{\bf u}^{\sigma'}(t)||_{\cC^{\delta+2}(D)} \nonumber\\
&\leq&\sup_{t\in(s_0,t_0)}||{\bf u}^{\sigma'}(t)||_{\mH^\beta_{\gamma}(D)} \nonumber\\
&\leq&C(t_0-s_0)^{\delta}\Big(||f'||_{\mL^q_p([s_0,t_0]\times D)}+||g'||_{\mL^q_p([s_0,t_0]\times D)}\Big) \nonumber\\
&<&\infty.
\end{eqnarray}
Therefore, by inequalities (\ref{a.16}) in Theorem  \ref{th17} and (\ref{e.4.2}), we get
\begin{eqnarray}\label{e.4.3}
||\nabla\Phi^{\sigma}-\nabla\Phi^{\sigma'}||_{\mL^q_p([s_0,t_0]\times D)}
&=&||\nabla v||_{\mL^q_p([s_0,t_0]\times D)}\leq||v||_{\mW^{2,q}_p([s_0,t_0]\times D)} \nonumber \\
&\leq & C||(L^{\sigma'}-L^\sigma)\cdot{\bf u}^{\sigma'}||_{\mL^q_p([s_0,t_0]\times D) } \nonumber\\
&\leq & C||(\sigma'-\sigma)\cdot\nabla^2{\bf u}^{\sigma'}||_{\mL^q_p([s_0,t_0]\times D) } \nonumber\\
&\leq & C||\sigma-\sigma'||_{\mL^q_p([s_0,t_0]\times D)}\cdot||\nabla^2{\bf u}^{\sigma'}||_{\mL^\infty_\infty([s_0,t_0]\times D)} \nonumber\\
&\leq & C||\sigma-\sigma'||_{\mL^q_p([s_0,t_0]\times D)}\,.
\end{eqnarray}
On the other hand,  by the  inequalities (\ref{sec2-eq1}), (\ref{016}) in Theorem \ref{th17} and (\ref{e.4.2}), we obtain
\ce
||\Phi^{\sigma'}-\Phi^\sigma||_{\mL^\infty_\infty([s_0,t_0]\times D)}&=&||v||_{\mL^\infty_\infty([s_0,t_0]\times D)}=\sup_{t\in[s_0,t_0]}||v(t)||_{L_{\infty}(D)}\\
&\leq&\sup_{t\in[s_0,t_0]}||v(t)||_{\cC^\delta(D)}
\leq C\sup_{t\in[s_0,t_0]}||v(t)||_{\mH^\beta_p(D)}\\
&\leq &C||(L^{\sigma'}-L^\sigma)\cdot{\bf u}^{\sigma'}||_{\mL^q_p([s_0,t_0]\times D)}\\
&\leq&C||\sigma-\sigma'||_{\mL^q_p([s_0,t_0]\times D)}.
\de
Combining this estimate with (\ref{e.4.3}), we get (\ref{sec4-eq06}).

Then dividing  $[0,T]$ into subintervals and  applying the last estimates  on each interval and piecing  them together, we see  that \eqref{sec4-eq05} is verified.

By generalized It\^o's formula   (e.g. \cite[Lemma 4.3]{Zhang2})
and by Lemma \ref{le6.02},
  $X^{b,\sigma}_t(x)$ satisfies  SDE (\ref{sec4-equ01}) on $[0,T\wedge\tau_D]$ with initial value $x\in D$ if and only if $Y^\sigma_t(y):=\Phi^\sigma(t,X^{b,\sigma}_t(x))$ satisfies  the following SDE on $[0,T\wedge\tau_D]$ with initial value in $D$, where $D$ is also   the image domain of function $\Phi^\sigma(t,y)$: 
    \begin{equation}\label{sec4-eq032}
dY^\sigma_t=(\nabla\Phi^\sigma\cdot\sigma)\circ(t,\Phi^{\sigma,-1}(t,Y^\sigma_t))dB_t:=\Theta^\sigma(t,Y^\sigma_t)dB_t,
\end{equation}
where
\begin{equation}\label{4.25}
\Theta^\sigma(t,y)=(\nabla\Phi^\sigma\cdot\sigma)\circ(t,\Phi^{\sigma,-1}(t,y)).
\end{equation}
Let $\Theta^{\sigma'}$ be defined as in (\ref{4.25}) via  $\sigma$ replaced by $\sigma'$. We have
\ce
\Theta^\sigma(t,y)-\Theta^{\sigma'}(t,y)&=&(\nabla\Phi^\sigma\cdot\sigma)\circ(t,\Phi^{\sigma,-1}(t,y))-(\nabla\Phi^\sigma\cdot\sigma)\circ(t,\Phi^{\sigma',-1}(t,y))\\
&\;&+(\nabla\Phi^\sigma\cdot\sigma-\nabla\Phi^{\sigma'}\cdot\sigma')\circ(t,\Phi^{\sigma',-1}(t,y))\\
&:=&J_{21}(t,y)+J_{22}(t,y).
\de
We are going to bound
$J_{21}(t,y)$ and $J_{22}(t,y)$ separately. First, we bound  $J_{21}(t,y)$.
Since  $\cM_D$ is bounded operator in any integrable functions, we have
\ce
|J_{21}|
&\leq& C|\Phi^{\sigma,-1}(t,y)-\Phi^{\sigma',-1}(t,y)|\\
&&\times\left(\cM_D|\nabla(\nabla\Phi^\sigma\cdot\sigma)|(t,\Phi^{\sigma,-1}(t,y))+\cM_D|\nabla(\nabla\Phi^\sigma\cdot\sigma)|(t,\Phi^{\sigma',-1}(t,y))\right).
\de
Noticing that
\ce
\sup_{y\in D}|\Phi^{\sigma,-1}(t,y)-\Phi^{\sigma',-1}(t,y)|&=&\sup_{x\in D}|x-\Phi^{\sigma',-1}\circ(t,\Phi^\sigma(t,x))|\\
&\leq&||\nabla\Phi^{\sigma',-1}\circ(t,\cdot)||_{L_{\infty}(D)}\times||\Phi^{\sigma'}(t,\cdot)-\Phi^\sigma(t,\cdot)||_{L_{\infty}(D)}
\de
and by the change of variables,  \cite[(2.3)]{Zhang}, (\ref{sec4-eq05}),  and (\ref{sec4-eq06}) we see
\ce
||J_{21}||_{\mL^q_p([s_0,t_0]\times D)}&\leq&C ||\cM_D|\nabla(\nabla\Phi^\sigma\cdot\sigma)|(\cdot,\Phi^{\sigma,-1})+\cM_D|\nabla(\nabla\Phi^\sigma\cdot\sigma)|(\cdot,\Phi^{\sigma',-1})||_{\mL^q_p([s_0,t_0]\times D)}\\
&&\times||\Phi^{\sigma,-1}-\Phi^{\sigma',-1}||_{\mL^\infty_\infty([s_0,t_0]\times D)} \\
&\leq&C||\nabla\Phi^{\sigma',-1}||_{\mL^\infty_\infty([s_0,t_0]\times D)}\times||\Phi^{\sigma'}-\Phi^\sigma||_{\mL^\infty_\infty([s_0,t_0]\times D)}\\
&&\times||\cM_D|\nabla(\nabla\Phi^\sigma\cdot\sigma)|||_{\mL^q_p([s_0,t_0]\times D)}\\
&\leq&C||\nabla(\nabla\Phi^\sigma\cdot\sigma)||_{\mL^q_p([s_0,t_0]\times D)}\times||\Phi^{\sigma'}-\Phi^\sigma||_{\mL^\infty_\infty([s_0,t_0]\times D)}\\
&\leq&C||\sigma'-\sigma||_{\mL^q_p([s_0,t_0]\times D)}.
\de
For $J_{22}(t,y)$, by the change of variables, boundedness  of $\sigma$ and (\ref{sec4-eq06}), we have
\ce
||J_{22}||_{\mL^q_p([s_0,t_0]\times D)}&=&||\nabla\Phi^\sigma\cdot\sigma-\nabla\Phi^{\sigma'}\cdot\sigma'||_{\mL^q_p([s_0,t_0]\times D)}\\
&\leq&||(\nabla\Phi^\sigma-\nabla\Phi^{\sigma'})\cdot\sigma||_{\mL^q_p([s_0,t_0]\times D)}+||\nabla\Phi^{\sigma'}(\sigma-\sigma')||_{\mL^q_p([s_0,t_0]\times D)}\\
&\leq&\kappa||\nabla\Phi^\sigma-\nabla\Phi^{\sigma'}||_{\mL^q_p([s_0,t_0]\times D)}\\
& &\qquad  +||\nabla\Phi^{\sigma'}||_{\mL^\infty_\infty([s_0,t_0]\times D)}\times||\sigma-\sigma'||_{\mL^q_p([s_0,t_0]\times D)}\\
&\leq&C||\sigma-\sigma'||_{\mL^q_p([s_0,t_0]\times D)}.
\de
Thus, we arrive at
\begin{equation}\label{sec4-eq040}
||\Theta^\sigma-\Theta^{\sigma'}||_{\mL^q_p([s_0,t_0]\times D)}\leq C||\sigma-\sigma'||_{\mL^q_p([s_0,t_0]\times D)}.
\end{equation}
Dividing  $[0,T]$ into subintervals and  applying the estimate (\ref{sec4-eq040}) on each interval and piecing them together, we obtain
\begin{equation}\label{sec4-eq041}
||\Theta^\sigma-\Theta^{\sigma'}||_{\mL^q_p([0,T]\times D)}\leq C||\sigma-\sigma'||_{\mL^q_p([0,T]\times D)}.
\end{equation}
Finally,  combining the above inequality  (\ref{sec4-eq041}) and Lemma \ref{le20}, we get
\ce
\sup_{x\in D}\mE\left(\sup_{t\in[0,T\wedge\tau_D]}|X_t^{b',\sigma}(x)-X_t^{b'\sigma'}(x)|^{p_0}\right)
&\leq&\sup_{y\in D}\mE\left(\sup_{t\in[0,T\wedge\tau_D]}|Y_t^\sigma(y)-Y_t^{\sigma'}(y)|^{p_0}\right)\\
&\leq&C||\Theta^\sigma-\Theta^{\sigma'}||^{p_0}_{\mL^q_p([0,T]\times D)}\\
&\leq& C||\sigma-\sigma'||^{p_0}_{\mL^q_p([0,T]\times D)}.
\de
This combined with Theorem \ref{th21}  completes the  proof of (\ref{sect4-equ07}).
\end{proof}

\section{\bf Existence and uniqueness of strong solution}\label{s.5}

In this  section, we shall study the strong solution to
 a  stochastic differential equation whose coefficients
 satisfy the conditions $({\bf H}^b$),  $ ({\bf H}^{\sigma})$
 and $({\bf H}^L)$.

Before doing this we need to study the solution of  a stochastic differential equation on
bounded domain $D$, which is of very interesting on its own.  Assume that   $b:[0,\infty)\times D\rightarrow\mR^d, \sigma:[0,\infty)\times D \rightarrow\mR^d\otimes\mR^d$ are measurable  and satisfy the conditions $({\bf H}^b$),  $ ({\bf H}^{\sigma})$
on the domain $D$. We first study the strong solution of  the following stochastic differential equation on $D$:
\begin{equation}\label{sec5-equ1}
  X_t=x+\int_0^tb(s,X_s)ds+\int_0^t\sigma(s,X_s)dB_s, ~~~t\in[0,T],~~~x\in D\,,
\end{equation}

\begin{theorem}\label{t00} Let $D$ be any  nonempty bounded (connected)
domain of $\mR^d$ such that  $\partial D\in C^2$.   Assume that $b,\sigma  $  satisfy  $({\bf H}^b)$ and$({\bf H}^{\sigma})$ on the domain $D$.
Then, the equation \eqref{sec5-equ1} has a unique solution $X_t$ up to positive stopping time
$\tau_D$, which is the first exit time of the process $X_t$ from the domain $D$. Namely, there are a stopping time $\tau_D$ and
unique stochastic process $(X_t\,, 0\le t\le \tau_D)$ such that
\begin{equation}\label{sec5-equ2}
  X_{t\wedge \tau_D} =x+\int_0^{t\wedge \tau_D} b(s,X_s)ds+\int_0^{t\wedge \tau_D}\sigma(s,X_s)dB_s, ~~~
  \forall \ t\in[0,\infty)\,.
\end{equation}
\end{theorem}
\begin{proof}
It is obvious    that   $b$ can be extended from $D$ to $\RR^d$ so that it satisfies the condition
$({\bf H}^b)$ on the whole space $\RR^d$ (in fact, we can require that the extension has compact
support). From  Proposition \ref{prop6.01}   it follows that $\si$ can also be extended
from $D$ to $\RR^d$ so that it satisfies the condition
$({\bf H}^\si)$ on the whole space $\RR^d$ (in fact, we can also require that the extension has compact
support). We   denote these two extensions  by $\tilde b$ and $\tilde \si$. Thus,  by \cite[Theorem 5.1]{Zhang}
the equation \eqref{sec5-equ1} on the whole space $\RR^d$
\begin{equation*}
    \tilde X_t =x+\int_0^t \tilde b(s, \tilde X_s)ds+\int_0^t\tilde \sigma(s,  \tilde X_s)dB_s, ~~~
  \forall \ t\in[0,\infty)\,.
\end{equation*}
 has a unique strong solution
$(\tilde X_t, 0\le t<\infty)$. Define $\tau_D=\inf \left\{ t>0 \,; \tilde X_t\not\in D\right\}$.
Then the process
$$
X_t=\tilde X_t \,, \quad 0\le t\le  \tau_D\
$$
satisfies \eqref{sec5-equ2}.

Since a strong solution is always a weak solution the uniqueness of the solution to the  equation (\ref{sec5-equ1}) follows from Theorem \ref{th22}.
\end{proof}
It is obvious that $\tau_D$ appeared in the above context is also the life time of the process $X_t$ in the domain $D$
$$
\tau_{D}=\inf \left\{ t>0\,; \ X_t\not\in D\right\}
$$
 and $X_{\tau_D}\in \partial D$ of $\tau_D$ is finite due the continuity of the process $X_t$.

\begin{theorem}\label{th11}  Assume that $\sigma$ and $b$ satisfy conditions $({\bf H}^\sigma)$, $({\bf H}^b)$, on any bounded domain $D$ and  assume $({\bf H}^{L})$ holds  true. Then the SDE (\ref{sec1-eq1}) has a unique strong solution.  Namely, there is a unique process $(X(t), 0\le t<\infty)$ such that
\begin{equation}\label{sec5-eq2}
X_t=x+\int_0^tb(s,X_s)ds+\int_0^t\sigma(s,X_s)dB_s\,, \quad \forall \ t\in [0, \infty)\,.
\end{equation}
\end{theorem}
\begin{proof}
Let
$$
D_R:=\left\{ x\in \RR^d\,; \ |x|\le R\right\}
$$
be the centered ball of radius $R$ in $\RR^d$. Then by Theorem \ref{t00}
equation \eqref{sec5-eq2} has a unique solution $X_t^R$ on $D_R$ with the life time $\tau_R$.

It is obvious that the stopping time $\tau_R$ is an (almost surely) increasing function of $R$ and then it converges to a stopping time $ \tau:=\lim_{R\to\infty}\tau_R $ as $R\rightarrow\infty$. By the uniqueness of the above theorem (Theorem \ref{t00})  we have that when $R\le R'$, then $X^R_t=X^{R'}_t$ for all $t\leq\tau_R$. We can define the process $(X_t, 0\le t\le \tau)$ so that $X_t=X^R_t$ if $t\leq\tau_R$  without ambiguity. Then $X$ satisfies equation (\ref{sec5-eq2}) when $t\leq\tau$:
\begin{equation*}
X_{t\wedge \tau}=x+\int_0^{t\wedge \tau}
b(s,X_s)ds+\int_0^{t\wedge \tau}\sigma(s,X_s)dB_s\,, \quad \forall \ t\ge 0\,.
\end{equation*}
Such solution  is clearly also unique.  Thus,
the proof of the theorem is completed if we can prove $\tau =\infty$ a.s.
and this is the objective of the following.

Applying  It\^o's formula to $e^{-Ct}V(t,x)$
we obtain
\ce
de^{-Ct}V(t,X_t)
&=&(e^{-Ct}\partial_tV(t,X_t)-Ce^{-Ct}V(t,X_t))dt+e^{-Ct}\partial_xV(t,X_t)dX_t\\
&&\qquad+e^{-Ct}\frac{1}{2}\nabla_x^2V(t,X_t))d\langle X,X \rangle_t\\
& =&e^{-Ct}(L_tV(t,X_t)-CV(t,X_t)))dt+e^{-Ct}\langle \nabla V(t,X_t),\sigma(t,X_t)\rangle dB_t.
\de
Thus we have by  the assumption  $({\bf H}^{L})$,
\ce
&&e^{-C\tau_R}V(\tau_R,X_{\tau_R})\leq V(0,x)+\int_0^{\tau_R}e^{-Cs}\langle \nabla V(s,X_s),\sigma(s,X_s)\rangle dB_s\,.
\de
Replacing  the above $\tau_R$ by $\tau_{R, N}=\tau_R\wedge N$  for any positive integer $N$ yields
 \ce
&&e^{-C\tau_{R, N} }V(\tau_{R, N},X_{\tau_{R, N}}  )\leq V(0,x)+\int_0^{\tau_{R, N}}e^{-Cs}\langle \nabla V(s,X_s),\sigma(s,X_s)\rangle dB_s\,.
\de
Since $\sigma$ is uniformly bounded and $\nabla V(s,X_s)I_{\{0\leq s\leq \tau_R\}}\leq C$, taking the expectation yields
\begin{equation}\label{sec5-eq04}
\mE e^{-C\tau_{R, N}}V(\tau_{R, N},X_{\tau_{R, N}})\leq  V(0,x)<\infty.
\end{equation}
Denote  $A_{N,R}:=\{\tau_R\leq N\}$.   The above
inequality (\ref{sec5-eq04})
yields
\ce
e^{-CN}\mE [V(\tau_R,X_{\tau_R})I_{A_{N,R}}]\leq\mE e^{-C\tau_{R, N}}V(\tau_{R, N},X_{\tau_{R, N}})\leq  V(0,x)<\infty.
\de
Thus,
\ce
\mP (A_{N,R})\leq\frac{e^{ CN}   V(0,x_0)}{ \inf _{  0\le t<\infty , |x|=R}V(t, x)}\,.
\de
Letting $R\rightarrow\infty$,  the condition $\lim_{|x|\rightarrow\infty}\inf _{  0\le t<\infty , |x|=R} V(t,x)=\infty$  gives
$$
\lim_{R\rightarrow\infty}
\mP(A_{N,R})=0\,,
$$
 implying in particular  $\mP(\tau\leq N)=0$ for any integer $N$. Thus $\lim_{R\uparrow\infty}\tau_R=\infty$ \   a.s.   \ This completes the proof of the theorem.
\end{proof}

\section{\bf  Appendix}\label{s.a}
Let $D$ be a  domain in $\mathbb{R}^d$ and $D^c=\mathbb{R}^d \backslash D$.
Let $f$ be a measurable function  from the   domain $D$ to $\mathbb{R}^d$. The local Hardy-Littlewood maximal operator $\mathcal{M}_D$ is defined by
\begin{equation}\label{sec6-eq01}
\mathcal{M}_D f(x):=\sup_{0<r<dist(x,D^c)}\frac{1}{|A(x,r)|}\int_{A(x,r)}|f(y)|dy,
\end{equation}
where $A(x, r)$ is a ball in $\mathbb{R}^d$ centered at $x$ with radius $r$ and $|A(x,r)|$ denotes the volume
of the ball $A(x, r)$. When $D=\mathbb{R}^d$, the operator $\mathcal{M}_D$ coincides with the classical centered Hardy-Littlewood maximal operator $\mathcal{M}$ (see  (2.4)).  It is  well known that $\mathcal{M}$ is $L^p(\mathbb{R}^d)$ bounded for $1<p\leq\infty$   (\cite[Theorem1.1]{Stein}). A simple observation   yields  that { $\mathcal{M}_Df(x)\leq \mathcal{M}(fI_{D})(x)$ for all $x\in D$ with $I_{\cdot}$ denoting  the indicator   function. Therefore, $\mathcal{M}_D$ is also bounded on $L^p(D)$ for $1<p\leq\infty$.

\begin{proposition}\label{prop6.01}
Assume $D$ is bounded  and  $\partial D$ is Lipschitz, $1\leq p<\infty$. Let $\bar D=D\cup \partial D$ be continuously embedded in another   open set $ V$. Then there exists a bounded linear operator $Q$,
\ce
Q: W^1_p(D)\rightarrow W^1_p(\mR^d)
\de
such that
\ce
Qf=f ~~~on~~~ D\,,\quad \forall \ f\in W^1_p(D)
\de
and the support of $Qf\subseteq V$.
Furthermore, if $f$ is   H\"older continuous on $D$, then $Qf$  is also
 H\"older continuous on $\mR^d$ of the same H\"older exponent.
\end{proposition}
\begin{proof}
The first part of the theorem  stems from \cite[Theorem 4.4.1]{Evans}. So, we    only need to prove that if $f$ is  H\"older continuous on $D$, then $Qf$ is  H\"older continuous on $\mR^d$. We will use the same
extension operator $Q$ introduced in the proof of \cite[Theorem 4.4.1]{Evans}, which we recall briefly
now.

We follow the same notations of  \cite[Theorem 4.4.1]{Evans}.
For any point   $x=(x_1,\cdots,x_d)\in \mR^d$, let us write $x=(x',x_d)$,  where  $x'=(x_1,\cdots,x_{d-1})\in \mR^{d-1}$, $x_d\in \mR$.

Given $x\in \mR^d$, and $r, h>0$, denote  the open cylinder
\ce
C(x,r,h)\equiv \{y\in \mR^d\,;  \quad |x'-y'|<r, |y_d-x_d|<h\}.
\de
Since $\partial U$ is Lipschitzian,  for each $x\in \partial D$,  there exist (upon rotating and
relabeling the coordinate axes if necessary ) two numbers $r, h>0$ and a Lipschitz function $\gamma: \mR^{d-1}\rightarrow \mR$
such that
\ce
D\cap C(x,r,h)=\{y\in \mR^d\,;\quad |x'-y'|<r, \quad \gamma(y')<y_d<x_d+h\}.
\de
Fix $x\in\partial D$.   With $r, h, \gamma$ as above, write
\ce
C:= C(x,r,h),~~C':= C(x,r/2,h/2),~~~D^+:= C'\cap D,  ~~~D^-:= C'\backslash \overline{D}.
\de
Set
\ce
\begin{cases}
f^+(y)=f(y)  &\qquad\hbox{if $y\in\overline{D}^+ $}\,,\\
f^-(y)=f(y',2\gamma(y')-y_d), &\qquad\hbox{if $y\in\overline{D}^-$}\,. \\
\end{cases}
\de
Note $f^+=f^-=f$ on $\partial D\cap C'$.

First, we assume that $\supp(f)\subseteq  D\cap C'(x,r,h)$. The extension of $f$
  is defined as follows
(see the proof of \cite[Theorem 4.4.1]{Evans}):
\begin{equation}\label{e.6.01}
Qf=\begin{cases}
f^+ &\qquad  ~~~on ~~~\overline{D}^+, \\
f^-&\qquad ~~~on ~~~\overline{D}^- ,\\
0    &\qquad ~~~ on ~~~\mR^d\backslash (\overline{D}^+\cup \overline{D}^-).
\end{cases}
\end{equation}
It is proved in \cite[Theorem 4.4.1]{Evans} that if $f\in W^1_p(D)$,  then
$Qf\in  W^1_p(\mR^d)$.

We want to show that if $f$ is H\"older continuous,
$$
|f(y)-f(z)|\le C|y-z|^\alpha\quad \forall \  y, z\in D,
$$
for some constant $\al\in (0, 1)$ and
positive constant $C$,  then $Qf$ is also H\"older continuous with the same exponent $\al$:
$$
|Qf(y)-Qf(z)|\le C| y-z|^\alpha\quad \forall \   y,z\in \mR^d.
$$

We shall prove the above by dividing our discussion into three cases.

Case I: When  $y,z\in \overline{D}^+$, or $y,z\in \overline{D}^-$, or $y,z\in \mR^d \backslash  (\overline{D}^+\cup   \overline{D}^-)$, it is obvious that $Qf$ is H\"older continuous.

Case II:  When  $y\in \overline{D}^+$ and $z\in \overline{D}^-$, by the definitions of $\overline{D}^+,\overline{D}^-$,  it holds  $y_d>\gamma(y')$ and $z_d<\gamma(z')$.  By the H\"older continuity of $f$, we have
\ce
|Qf(y)-Qf(z)|&=&|f(y',y_d)-f(z',2\gamma(z')-z_d)|\leq C\Big(|y'-z'|^{\alpha}+|y_d+z_d-2\gamma(z')|^{\alpha}\Big)\\
&&\leq C\Big(|y'-z'|^{\alpha}+|y_d-\gamma(y')+z_d-\gamma(z')|^{\alpha}+|\gamma(y')-\gamma(z')|^{\alpha}\Big)\\
&&\leq C \Big( |y'-z'|^{\alpha}+|y_d-\gamma(y')+z_d-\gamma(z')|^{\alpha}\Big)\,.
\de
If  $y_d-\gamma(y')>|z_d-\gamma(z')|$,  then
\ce
|y_d-\gamma(y')+z_d-\gamma(z')|&=&y_d-\gamma(y')+z_d-\gamma(z')\leq |y_d-\gamma(y')-z_d+\gamma(z')|\\
&&\leq|y_d-z_d|+|\gamma(y')-\gamma(z')|\leq|y_d-z_d|+C|y'-z'|\,.
\de
If  $y_d-\gamma(y')\leq |z_d-\gamma(z')|$,   then
\ce
|y_d-\gamma(y')+z_d-\gamma(z')|&=&-y_d+\gamma(y')-z_d+\gamma(z')\leq|y_d-\gamma(y')-z_d+\gamma(z')|\\
&&\leq|y_d-z_d|+C |y'-z'|.
\de
Thus,  we obtain
\ce
|Qf(y)-Qf(z)|\leq C \left(|y'-z'|^{\alpha}+|y_d-z_d |^{\alpha}\right)\le C|y-z|^\al .
\de
Case III: When  $y\in \overline{D}^+$, and $z\in \mR^d\backslash(\overline{D}^+\cup \overline{D}^-)$, let $u\in \partial \overline{D}^+$ be the intersection of  boundary $\partial \overline{D}^+$  and the line
connecting the two points $y,z$.  By the continuity of $f$ we see  $ f(u)=0$. Then we have
\ce
|Qf(y)-Qf(z)|=|f(y)|=|f(y)-f(u)|\leq C|y-u|^{\alpha}\leq C|y-z|^{\alpha}.
\de
Similarly, when  $y\in \overline{D}^-$, and $z\in \mR^d\backslash(\overline{D}^+\cup \overline{D}^-)$, it also holds
\ce
|Qf(y)-Qf(z)|\leq C|y-z|^{\alpha}.
\de
Thus, this proposition is proved in case $f$ with support in $C'\cap \overline{D}$.

Now we remove  the restriction $\supp(f)\subseteq C'\cap \overline{D}$. Since $\partial D$ is compact, we can cover $\partial D$ with finitely many cylinders $C'_k=C(x_k,r_k/2,h_k/2)$ with each $x_k\in \partial D$, $k=1,2,\cdots,N$. Let $\{\eta_k\}^N_{k=0}$ be a sequence of smooth functions (see \cite[Theorem 2.13]{rudin}
for a construction) such that
$$
\begin{cases}
0\leq \eta_k\leq 1\quad\hbox{and}\quad \supp(\eta_k)\subseteq C'_k, ~~~(k=1,2,\cdots,N),\\
0\leq \eta_0\leq 1\quad\hbox{and}\quad\supp(\eta_0)\subseteq D,\\
\sum^N_{k=0}\eta_k=1,  ~~~ on ~~~D.
\end{cases}
$$
Define $Q(\eta_kf)$ ($k=1,2,\cdots,N$) as above (\ref{e.6.01}) and let $Qf:=\sum^N_{k=1}Q(\eta_kf)+\eta_0f$. For any $y,z\in \mR^d$, we have  by the above argument
\ce
|Qf(y)-Qf(z)|&\leq& \sum^N_{k=1}|Q(\eta_kf)(y)-Q(\eta_kf)(z)|+|\eta_0(y)f(y)-\eta_0(z)f(z)|\\
&\leq& CN|y-z|^{\alpha}+|\eta_0(y)f(y)-\eta_0(z)f(z)|.
\de
To show the H\"older continuity of $\eta_0  f$, we also divide
our discussion into three cases: Case I: both $y$ and $z$ are in $D$, Case II: both $y$ and $z$ are not in $D$ and
Case III: $y\in D$ and $z\not\in D$.  Case II is easy.

Case I:  $y,z\in D$.  In this case we have
\ce
|\eta_0(y)f(y)-\eta_0(z)f(z)|& \leq& |\eta_0(y)(f(y)-f(z))|+|(\eta_0(y)-\eta_0(z))f(z)|\\
& \leq & |f(y)-f(z)|+|\eta_0(y)-\eta_0(z)|\leq C|y-z|^{\alpha}\,.
\de

Case III: $y\in D$ and $z\not\in D$.  In this case we notice that
$\eta_0(z)=0$ when $z\not\in D$ and we   have   then
\ce
|\eta_0(y)f(y)-\eta_0(z)f(z)|& =& |\eta_0(y)f(y) |=|\eta_0(y)f(y)-\eta_0( z)f(y )|\\
&\le&  |f(y)| |\eta_0(y)-\eta_0( z)| \leq C|y-z|^{\alpha}\,.
\de
Then, for any $y,z\in \mR^d$, we have
\ce
|Qf(y)-Qf(z)|\leq C|y-z|^{\alpha}.
\de
Finally, the proof of this proposition is completed.
\end{proof}

\begin{lemma}\label{le6.02} Assume $D$ is a bounded, connected with $C^2$ boundary domain in $\mR^d$. Assume that $\sigma$ and $b$ satisfy conditions $({\bf H}^\sigma)$, $({\bf H}^b)$, respectively. Let  $\Phi(t,x):=x+{\bf u}(t,x)$ with ${\bf u}(t,x)=\left(u^1(t,x),u^2(t,x),\ldots,u^d(t,x)\right)$ satisfying following  PDE:
\begin{equation}
\partial_tu^{\ell}+L^{\sigma}u^{\ell}+b\cdot\nabla u^{\ell}+b^{\ell}=0, ~~~\ell=1,2,\cdots,d
\end{equation}
and  ${\bf u}(t,x)=0$ for $x\in\partial D$, then the image of $\Phi(t,x)$ is $D$.
\end{lemma}
\begin{proof}Since the image of $\Phi(t,x)$ has the same boundary as $D$, and $D$ is a bounded, connected with $C^2$ boundary domain,  therefore by the Jordan-Brouwer Separation Theorem (see \cite{Guillemin}),we have that the image of $\Phi(t,x)$ is $D$.
\end{proof}

\begin{lemma}\label{le6.1}
Assume that $\{\beta(t) , t\in [0, T]\}$ is a nonnegative measurable process adapted to a flow $\{\mathcal{F}_t , t\in [0, T]\}$ of $\sigma$-algebras on some probability space $(\Omega, \mathcal{F}, P)$ (this means that $\mathcal{F}_s\subseteq\mathcal{F}_t\subseteq\mathcal{F}$ for $0<s<t<T$), and the random variable $\beta(t)$ is $\mathcal{F}_t$ measurable for each $t\in[0, T)$. Further, assume that for $0<s<t<T$,  and any stopping time $\tau$,
$$
\mathbb{E}\left \{\int_{s\wedge\tau}^{t\wedge\tau}\beta(r)dr|\mathcal{F}_{s\wedge\tau}\right \}\leq \rho(s,t),
$$
where $\rho(s,t)$ is a nonrandom interval function satisfying the following conditions:\\
(i)\quad$\rho(t_1,t_2)\leq \rho(t_3,t_4)$ if $(t_1,t_2)\subseteq(t_3,t_4)$.\\
(ii)\quad$\lim_{h\rightarrow0}\sup_{0<s\leq t\leq T,t-s\leq h}\rho(s,t)=\kappa_0, \kappa_0\geq0$.\\
Then for an arbitrary real $\lambda<\kappa_0^{-1}$,
$$
\mathbb{E}\exp\left \{\lambda\int_0^{T\wedge\tau}\beta(r)dr\right \}<\infty.
$$
\end{lemma}
\begin{proof} First we introduce the truncated process
\begin{align*}
\begin{split}
\beta_N(t)=\left\{
              \begin{array}{ll}
                \beta(t), ~~~&if~~~ \beta(t)\leq N,\\
                0, ~~~&if~~~  \beta(t)> N,
                \end{array}
                 \right.
\end{split}
\end{align*}
where $t\in[0, T \wedge \tau]$ and $N$ is a positive number ($N \rightarrow \infty$ later). Since $\beta_N(t)$ is a bounded process, for all real $\lambda$,
$$
\mathbb{E}\exp\left \{\lambda\int_0^{T\wedge\tau}\beta_N(r)dr\right \}\leq e^{\lambda N T}<\infty.
$$
For $0<s<t<T$ and for real $\lambda$ let
$$
\varphi_{\lambda}(s,t,N):=\mathbb{E}\left [\exp\left \{\lambda\int_{s\wedge \tau}^{t\wedge\tau}\beta_N(r)dr\right \}|\mathcal{F}_{s\wedge \tau}\right ].
$$
Using  the identity
$$
\exp\left \{\lambda\int_s^t\beta_N(r)dr\right \}=1+\lambda\int_s^t\beta_N(r)\exp\left \{\lambda\int_r^t\beta_N(\theta)d\theta\right\}dr 
$$
and taking conditional expectation on both sides, we get an equation for the function $\varphi_{\lambda}(s,t,N)$
\begin{align}
\varphi_{\lambda}(s,t,N)&=1+\lambda\mathbb{E}\left [\int_{s\wedge \tau}^{t\wedge \tau}\beta_N(r)\exp\left \{\lambda\int_{r\wedge \tau}^{t\wedge \tau}\beta_N(\theta)d\theta\right\}dr|\mathcal{F}_{s\wedge \tau}\right ]\\ \nonumber
&=1+\lambda\mathbb{E}\left [\mathbb{E}\left [\int_{s\wedge \tau}^{t}\beta_N(r)1_{r\leq\tau}\exp\left \{\lambda\int_{r\wedge \tau}^{t\wedge \tau}\beta_N(\theta)d\theta\right\}|\mathcal{F}_{r\wedge \tau}dr\right ]|\mathcal{F}_{s\wedge \tau}\right ]\\\nonumber
&=1+\lambda\mathbb{E}\left [\left [\int_{s\wedge \tau}^{t}\mathbb{E}\beta_N(r\wedge\tau)1_{r\leq\tau}\exp\left \{\lambda\int_{r\wedge \tau}^{t\wedge \tau}\beta_N(\theta)d\theta\right\}|\mathcal{F}_{r\wedge \tau}dr\right ]|\mathcal{F}_{s\wedge \tau}\right ]\\\nonumber
&=1+\lambda\mathbb{E}\left [\left [\int_{s\wedge \tau}^{t\wedge \tau}\beta_N(r\wedge\tau)\mathbb{E}\exp\left \{\lambda\int_{r\wedge \tau}^{t\wedge \tau}\beta_N(\theta)d\theta\right\}|\mathcal{F}_{r\wedge \tau}dr\right ]|\mathcal{F}_{s\wedge \tau}\right ]\\\nonumber
&=1+\lambda\mathbb{E}\left [\int_{s\wedge \tau}^{t\wedge \tau}\beta_N(r)\varphi_{\lambda}(r,t,N)dr|\mathcal{F}_{s\wedge \tau}\right].
\end{align}
Thus, by induction,
$$
\varphi_{\lambda}(s,t,N)=\sum_{k=0}^\infty\varphi^{(k)}_{\lambda}(s,t,N),
$$
where $\varphi^{(0)}_{\lambda}(s,t,N)\equiv 1$,  and
$$
\varphi^{(k+1)}_{\lambda}(s,t,N):=\lambda\mathbb{E}\left [\int_{s\wedge \tau}^{t\wedge \tau}\beta_N(r)\varphi^{(k)}_{\lambda}(r,t,N)dr|\mathcal{F}_{s\wedge \tau}\right].
$$
To estimate $\varphi^{(k)}_{\lambda}(s,t,N)$ uniformly with respect to $N$, we observe that
$$
\mathbb{E}\left \{\int_{s\wedge\tau}^{t\wedge\tau}\beta_N(r)dr|\mathcal{F}_{s\wedge\tau}\right \}\leq
\mathbb{E}\left \{\int_{s\wedge\tau}^{t\wedge\tau}\beta(r)dr|\mathcal{F}_{s\wedge\tau}\right \} \leq \rho(s,t).
$$
By induction on $k$, we now easily arrive at the inequalities by using condition (i)
$$
\varphi^{(k)}_{\lambda}(s,t,N)\leq (\lambda\rho(s,t))^k, ~~k=0,1,\ldots ,
$$
It then follows from (ii) that for a fixed $\lambda<\kappa_0^{-1}$ ( We use the  convention $0^{-1}=\infty$) there is an $h_0>0$ such that $\lambda \rho(s,t)<1$ if $t-s<h_0$. Therefore, for the given $\lambda<\kappa_0^{-1}$, we obtain
\begin{equation}\label{eq1.01}
\varphi_{\lambda}(s,t,N)=\sum_{k=0}^\infty\varphi^{(k)}_{\lambda}(s,t,N)\leq \sum_{k=0}^\infty (\lambda\rho(s,t))^k \leq (1-\lambda\rho(s,t))^{-1}< \infty
\end{equation}
for all $0<s<t<T$ with $t-s<h_0$. Now partition $[0, T\wedge\tau]$ by points
$0= t_0\wedge\tau <t_1\wedge\tau <\ldots < t_n\wedge\tau= T\wedge\tau$, so that $\max_{1\leq k\leq n}(t_k-t_{k-1})< h_0$.
An easy   computation deduces that $t_k\wedge\tau-t_{k-1}\wedge\tau\leq t_k - t_{k-1}$. Then by (\ref{eq1.01})
\begin{align}
&\mathbb{E}\exp\left \{\lambda\int_0^{T\wedge\tau}\beta_N(r)dr\right \}\\ \nonumber
&=\mathbb{E}\prod_{i=0}^{n-1}\exp\left \{\lambda\int_{t_i\wedge\tau}^{t_{i+1}\wedge\tau}\beta_N(r)dr\right \}\\\nonumber
&=\mathbb{E}\left [\mathbb{E}\left [\prod_{i=0}^{n-1}\exp\left \{\lambda\int_{t_i\wedge\tau}^{t_{i+1}\wedge\tau}\beta_N(r)dr\right \}|\mathcal{F}_{t_{n-1}\wedge\tau}\right ]\right ]\\\nonumber
&=\mathbb{E}\left [\prod_{i=0}^{n-2}\exp\left \{\lambda\int_{t_i\wedge\tau}^{t_{i+1}\wedge\tau}\beta_N(r)dr\right \}\mathbb{E}\left [\exp\left \{\lambda\int_{t_{n-1}\wedge\tau}^{t_{n}\wedge\tau}\beta_N(r)dr\right \} |\mathcal{F}_{t_{n-1}\wedge\tau}\right ]\right ]\\\nonumber
&=\mathbb{E}\left [\prod_{i=0}^{n-2}\exp\left \{\lambda\int_{t_i\wedge\tau}^{t_{i+1}\wedge\tau}\beta_N(r)dr\right \}\varphi_{\lambda}(t_{n-1},t_n,N)\right ]\\\nonumber
&\leq (1-\lambda\rho(t_{n-1},t_n))^{-1}\mathbb{E}\prod_{i=0}^{n-2}\exp\left \{\lambda\int_{t_i\wedge\tau}^{t_{i+1}\wedge\tau}\beta_N(r)dr\right \} \\\nonumber
&\leq \ldots \leq \prod_{i=1}^{n}(1-\lambda\rho(t_{i-1},t_i))^{-1}.
\end{align}
Finally,  by dominated  convergence theorem, we get
\begin{eqnarray*}
\mathbb{E}\exp\left \{\lambda\int_0^{T\wedge\tau}\beta(r)dr\right \}&=&\lim_{N\rightarrow \infty} \mathbb{E}\exp\left \{\lambda\int_0^{T\wedge\tau}\beta_N(r)dr\right \} \\ \nonumber
&\leq & \prod_{i=1}^{n}(1-\lambda\rho(t_{i-1},t_i))^{-1}<\infty\,.
\end{eqnarray*}
The proof of the lemma is completed.
\end{proof}

\begin{lemma}\label{le6.2}
Let $(\xi(t))_{t\in[0,T]}$, $(\zeta(t))_{t\in[0,T]}$ and $(\beta(t))_{t\in[0,T]}$ be three real-valued measurable $\mathcal{F}_t$-adapted processes.  Let $(\eta(t))_{t\in [0,T]}$ and $(\alpha(t))_{t\in [0,T]}$ be two $\mathbb{R}^d$-valued measurable $\mathcal{F}_t$-adapted processes. Suppose there exist $c>0$ and $\delta\in(0,1)$ such that for any $0<s\leq t\leq T$ and any stopping time $\tau$,
\begin{equation}\label{eq6.1}
\mathbb{E}\left \{\int_{s\wedge\tau}^{t\wedge\tau}|\beta(r)|+|\alpha(r)|^2dr|\mathcal{F}_{s\wedge\tau}\right \}\leq c(t-s)^{\delta}
\end{equation}
and suppose that
$$
\xi(t)=\int_0^t\zeta(r)dr+\int_0^t\eta(r)dB_r+\int_0^t\xi(r)\beta(r)dr+\int_0^t\xi(r)\alpha(r)dB_r.
$$
Then for any $p>0$ and $\gamma_2, \gamma_4>1$, we have
\begin{equation}\label{eq6.2}
\mathbb{E}\left (\sup_{0\leq t\leq T\wedge\tau}\xi^p(t)\right )\leq C\left(\left |\left |\left (\int_0^{T\wedge\tau}\zeta^{+}(r)dr\right )^p\right |\right |_{L_{\gamma_4}(\Omega)}+\left |\left |\left (\int_0^{T\wedge\tau}|\eta(r)|^2dr\right )^{p/2}\right |\right |_{L_{\gamma_2}(\Omega)}\right)\,,
\\
\end{equation}
where $C=C(c,\delta,p,\gamma_2,\gamma_4)>0$.
\end{lemma}
\begin{proof}Write
$$
M(t):=\exp\left \{\int_0^t\alpha(r)dB_r-\frac{1}{2}\int_0^t|\alpha(r)|^2dr+\int_0^t\beta(r)dr\right \}.
$$
Applying It\^o’s formula to $\xi(t)M^{-1}(t)$, we can see that
\begin{align}
&d(\xi(t)M^{-1}(t)) \\ \nonumber
&=M^{-1}(t)d\xi(t)-\xi(t)M^{-1}(t)(\alpha(t)dB_t-|\alpha(t)|^2dt+\beta(t)dt)
+\langle d\xi, dM^{-1} \rangle_t \\ \nonumber
&=M^{-1}(t)\eta(t)dB_t+M^{-1}(t)(\zeta(t)-\langle \alpha(t),\eta(t)\rangle )dt
\end{align}
and then
\begin{equation}\label{eq1.02}
\xi(t)=M(t)\left (\int_0^tM^{-1}(r)\eta(r)dB_r+\int_0^tM^{-1}(r)(\zeta(r)-\langle \alpha(r),\eta(r)\rangle )dr\right ).
\end{equation}
By (\ref{eq6.1}) and Lemma \ref{le6.1}, we have for any $p\in \mathbb{R}$,
\begin{equation}\label{eq1.2}
\mathbb{E}\exp\left \{p\int_0^{T\wedge\tau}|\alpha(r)|^2dr+p\int_0^{T\wedge\tau}|\beta(r)|dr\right \}\leq C<\infty,
\end{equation}
where the positive constant $C$ depends only  on $c,p,\delta$. Furthermore, for any $p\in \mathbb{R}$, by (\ref{eq1.2}),
$$
\exp\left \{p\int_0^{t\wedge\tau}\alpha(r)dB_r-\frac{p^2}{2}\int_0^{t\wedge\tau}|\alpha(r)|^2dr\right \}
$$
is an uniformly integrable exponential martingale. Now we introduce some notations for the sake of simplification below,
$$
I_1(t):=\exp\left \{p\int_0^t\alpha(r)dB_r-\frac{p^2p_1}{2}\int_0^t|\alpha(r)|^2dr\right \},
$$
$$
I_2(t):=\exp\left \{\frac{p^2p_1-p}{2}\int_0^t|\alpha(r) ^2dr+p\int_0^t\beta(r)dr\right \}
$$
with $p_1>1$,  $p\in \RR$,
$$
I_3(t):=\int_0^tM^{-1}(r)\eta(r)dB_r,
$$
and
$$
I_4(t):=\int_0^tM^{-1}(r)(\zeta(r)-\langle \alpha(r),\eta(r)\rangle )dr\,.
$$
Thus, by H\"older's inequality, Doob's maximal inequality \cite[Theorem B]{Wang} and Lemma \ref{le6.1}, we have for any $p\in \mathbb{R}$,
\begin{align}\label{eq1.3}
&\mathbb{E}(\sup_{0\leq t\leq T\wedge\tau}|M(t)|^p)\\ \nonumber
&=\mathbb{E}\sup_{0\leq t\leq T\wedge\tau}\exp\left \{p\int_0^t\alpha(r)dB_r-\frac{p}{2}\int_0^t|\alpha(r)|^2dr+p\int_0^t\beta(r)dr\right \}\\ \nonumber
&=\mathbb{E}(\sup_{0\leq t\leq T\wedge\tau}I_1(t)I_2(t))\\ \nonumber
&\leq  \left (\mathbb{E}(\sup_{0\leq t\leq T\wedge\tau}I_1)^{p_1}\right )^{1/p_1}\times \left (\mathbb{E}(\sup_{0\leq t\leq T\wedge\tau}I_2)^{p_2}\right )^{1/p_2}\\\nonumber
&\leq \left (\frac{p_1}{p_1-1}\right )\left (\mathbb{E}\exp\left \{pp_1\int_0^{T\wedge\tau}\alpha(r)dB_r-\frac{p^2p^2_1}{2}\int_0^{T\wedge\tau}|\alpha(r)|^2dr\right \}\right )^{1/p_1}\\\nonumber
& \times \left (\mathbb{E}(\sup_{0\leq t\leq T\wedge\tau}I_2)^{p_2}\right )^{1/p_2}\\\nonumber
&\leq \left (\frac{p_1}{p_1-1}\right ) \left (\mathbb{E}(\sup_{0\leq t\leq T\wedge\tau}I_2)^{p_2}\right )^{1/p_2} \\\nonumber
&<\infty,
\end{align}
where $p_2>1$ is the conjugate of $p_1$.
Finally, for any $p>0$, by (\ref{eq1.02}), H\"older's inequality, (\ref{eq1.3}) and Burkh\"older's inequality \cite[corollary 4.2, Chapter IV]{Revuz},
\begin{align}
&\mathbb{E}(\sup_{0\leq t\leq T\wedge\tau}(\xi(t))^p)\\\nonumber
&=\mathbb{E}(\sup_{0\leq t\leq T\wedge\tau}|M(t)|^p(I_3(t)+I_4(t))^p)\\\nonumber
&\leq C_p (\mathbb{E}(\sup_{0\leq t\leq T\wedge\tau}|M(t)|^p)^{\gamma_1})^{1/\gamma_1} \times (\mathbb{E}(\sup_{0\leq t\leq T\wedge\tau}I^p_3(t))^{\gamma_2})^{1/\gamma_2}\\\nonumber
&\qquad +C_p (\mathbb{E}(\sup_{0\leq t\leq T\wedge\tau}|M(t)|^p)^{\gamma_3})^{1/\gamma_3} \times (\mathbb{E}(\sup_{0\leq t\leq T\wedge\tau}I^p_4(t))^{\gamma_4})^{1/\gamma_4}\\\nonumber
&\leq C_{p,\gamma_1,\gamma_3,c,\delta}\left (\left (\mathbb{E}(\sup_{0\leq t\leq T\wedge\tau}I^p_3(t))^{\gamma_2}\right )^{1/\gamma_2}+\left (\mathbb{E}(\sup_{0\leq t\leq T\wedge\tau}I^p_4(t))^{\gamma_4}\right )^{1/\gamma_4}\right )\\\nonumber
&\leq  C_{p,\gamma_1,\gamma_3,c,\delta}\left (\mathbb{E}\left (\int_0^{T\wedge\tau}|\eta(r)|^2dr\right )^{\frac{p\gamma_2}{2}}\right )^{1/\gamma_2}
+C_{p,\gamma_1,\gamma_3,c,\delta}\left (\mathbb{E}(\sup_{0\leq t\leq T\wedge\tau}I^p_4(t))^{\gamma_4}\right )^{1/\gamma_4} \\ \nonumber
&\leq  C_{p,\gamma_1,\gamma_3,c,\delta} \left(\left|\left|\left(\int_0^{T\wedge\tau}|\eta(r)|^2dr\right)^{\frac{p}{2}}\right|\right|_{L_{\gamma_2}(\Omega)}
+\left|\left|\left(\int_0^{T\wedge\tau}\zeta^+(r)dr\right)^p\right|\right|_{L_{\gamma_4}(\Omega)}\right),
\end{align}
where $\gamma_1,\gamma_2,\gamma_3,\gamma_4>1$,  $\gamma_1,\gamma_2$ and $\gamma_3,\gamma_4$ are two pairs of conjugate numbers.  The proof of  \ref{eq6.2} is completed.
\end{proof}


\begin{thebibliography}{999}
\bibitem{Adams}Adams, R. A.; Fournier, J.  J. F. Sobolev spaces. Second edition. Pure and Applied Mathematics (Amsterdam), 140. Elsevier/Academic Press, Amsterdam, 2003. 


\bibitem{Crippa} Crippa, G. and De Lellis, C.  Estimates and regularity results for the DiPerna-Lions flow. J. Reine Angew. Math. 616. 15-46 (2008).

\bibitem{David} David, G., Neil, T. Elliptic Partial Differential Equations of Second Order.Grundlehren der mathematischen Wissenschaften (1998).


\bibitem{analysis2} Dieudonn\'e, J. Treatise on analysis. Vol. II. Enlarged and corrected printing. Translated by I. G. Macdonald. With a loose erratum. Pure and Applied Mathematics, 10-II. Academic Press, New York-London, 1976.

\bibitem{Evans}L.C. Evans and R, Gariepy. Measure Theory and Fine Properties of Functions. CRC Press, Boca Raton, FL (1992).

\bibitem{Friedman} Friedman, A. Partial Differential Equations of Parabolic Type. Prentice-Hall, Englewood Cliffs, NJ. (1964).

\bibitem{Flandoli}Flandoli, F.,Gubinelli, M. and Priola, E.:Well-posedness of the transport equation by stochastic perturbation. Invent. Math. 180 C53, (2010).



\bibitem{Guillemin} Guillemin. V, Pollack. A. Differential Topology.  Prentice-Hall, Inc., Englewood Cliffs, N.J., (1974).

\bibitem{Hu}  Hu, Y.  Analysis on Gaussian spaces. World Scientific Publishing. (2016).

\bibitem{HX} Hu, Y. and Xi, Y. Parameter estimation for threshold Ornstein-Uhlenbeck processes from discrete observations.J. of Comput. and Appl. Math.  411 (2022), no.  114264.




\bibitem{Krylov}   Krylov, N.V.   Controlled Diffusion Processes, Applications of  Mathematics 14, translated by A.B. Aries, (1980).

\bibitem{Krylov1}  Krylov, N.V. Lectures on Elliptic and Parabolic Equations in Sobolev Spaces. Graduate Studies in Mathematics 96. Amer. Math. Soc., Providence. (2008).



\bibitem{KR} Krylov, N. V. and R\"ockner, M.  Strong solutions of stochastic equations with singular time dependent drift. Probab. Theory Related Fields 131 (2005),   154-196.

\bibitem{kufner} Kufner, A.; John, O.; Fu\u c\`ik, S.  Function spaces. Monographs and Textbooks on Mechanics of Solids and Fluids, Mechanics: Analysis. Noordhoff International Publishing, Leiden; Academia, Prague, 1977


\bibitem{Leoni} Giovanni Leoni: A First Course in Sobolev Spaces. Second Edition. American Mathematical Society
Providence, Rhode Island. (2017).




\bibitem{Revuz} D. Revuz, M. Yor. Continuous Martingales and Brownian Motion. Springer-Verlag, (1991).

\bibitem{rudin} Rudin, W. Real and complex analysis. Third edition. McGraw-Hill Book Co., New York, 1987.


\bibitem{Stein}Stein, E. M.  Singular Integrals and Differentiability Properties of Functions. Princeton Mathematical Series, No. 30. Princeton Univ. Press, Princeton, NJ. (1970).

\bibitem{SV} Stroock, D. W.  and  Varadhan, S. R. Srinivasa. Multidimensional diffusion processes. Reprint of the 1997 edition. Classics in Mathematics. Springer-Verlag, Berlin, 2006.

\bibitem{Triebel}Triebel, H.  Theory of Function Spaces II. Birkhauser/Springer Basel AG, Basel. (2010).

\bibitem{Triebel1}Triebel, H.  Theory of Function Spaces. Birkhauser/Springer Basel AG, Basel. (2010).

\bibitem{Taira}Taira, K. Analytic Semigroups and Semilinear Initial-Boundary Value Problems. London Mathematical Society Lecture Note Series 223. Cambridge Univ. Press, Cambridge. (1995).

\bibitem{Wang} Wang, G. Sharp maximal inequalities for conditionally symmetric martingales and Brownian
motion. Proc. Amer. Math. Soc. 112 (2), 579-586 (1991).

\bibitem{vere} Veretennikov, A. Ju. Strong solutions of stochastic differential equations. Teor. Veroyatn. Primen. 24, 348-360 (1979).

\bibitem{Xia}  Xia, P.;    Xie, L.;  Zhang, X. and   Zhao, G. $L^q(L^p)$ theory of stochastic differential equations. Stochastic Processes and their Applications. $\mathbf(130)$ (2020), 5188-5211.

\bibitem{Zhang} Zhang, X.  Stochastic differential equations with sobolev diffusion and singular drift and applications. The Annals of Applied Probability. Vol 26 (2016),  2697-2732.

\bibitem{Zhang1}  Zhang, X.   Zhao, G.  Heat kernel and ergodicity of SDEs with distributional drifts. Available at arXiv:1710.10537v2.

\bibitem{Zhang2} Zhang,X.  Stochastic homeomorphism flows of SDEs with singular drifts and Sobolev diffusion coefficients. Electron. J. Probab. 16, 1096-1116 (2011).


\bibitem{ZY} Zhang, S.-Q. Yuan, C. A Zvonkin's transformation for stochastic differential equations with singular drift and applications. J. Differential Equations 297, 277-319 (2021).


\bibitem{Zvonkin} Zvonkin, A. K. A transformation of the phase space of a diffusion process that will remove the drift. Mat. Sb. 93(135) 129-149, 152 (1974).
\end{thebibliography}
\end{document}